\documentclass[10pt, reqno]{amsart}
\usepackage{pxfonts, amssymb, mathrsfs}

\usepackage{color}
\usepackage{graphicx}

\definecolor{deepgreen}{cmyk}{1,0,1,0.5}

\usepackage{slashed}

\newcommand{\HH}{\mathcal{H}}

\newcommand{\K}{\mathcal{K}}

\newcommand{\NN}{\mathcal{N}}

\newcommand{\cV}{\mathcal{V}}

\newcommand{\R}{\mathbb{R}}
\newcommand{\Sp}{\mathbb{S}}
\newcommand{\Z}{\mathbb{Z}}

\newcommand{\al}{\alpha}

\newcommand{\ga}{\gamma}

\newcommand{\e}{\varepsilon}
\newcommand{\fy}{\varphi}

\newcommand{\la}{\lambda}

\newcommand{\La}{\Lambda}

\newcommand{\p}{\partial}

\makeatletter

\newcommand{\Rmnum}[1]{\expandafter\@slowromancap\romannumeral #1@}
\makeatother

\newcommand{\I}{\infty}

\newcommand{\ti}{\widetilde}

\newcommand{\ang}[1]{\left\langle{#1}\right\rangle}
\newcommand{\abs}[1]{\left\lvert{#1}\right\rvert}

\newcommand{\ant}[1]{\begin{align*}\begin{split} #1 \end{split}\end{align*}}
\newcommand{\EQ}[1]{\begin{equation}\begin{split} #1 \end{split}\end{equation}}

\newcommand{\bmat}[1]{\begin{bmatrix} #1 \end{bmatrix}}

\newcommand{\Del}[1]{}

\newcommand{\pt}{&}

\numberwithin{equation}{section}

\newcommand{\mif}{{\ \ \text{if} \ \ }}

\newcommand{\mas}{{\ \ \text{as} \ \ }}

\newcommand{\spa}{\operatorname{span}}
\newcommand{\Proj}{\operatorname{Proj}}
\newcommand{\Res}{\operatorname{Res}}

\newtheorem{thm}{Theorem}[section]
\newtheorem*{thm*}{Theorem}

\newtheorem{prop}[thm]{Proposition}
\newtheorem{cor}[thm]{Corollary}
\newtheorem{lemma}[thm]{Lemma}

\theoremstyle{definition}

\theoremstyle{remark}
\newtheorem{remark}[thm]{Remark}

\newcommand{\tdk}{\tilde{k}}

\newcommand{\pipp}{\pi_R^\perp}
\newcommand{\pip}{\pi_R}
\newcommand{\hrr}{\mathcal{H}(r\geq R)}
\newcommand{\ld}{\lambda_1}
\newcommand{\ldd}{\lambda_2}
\newcommand{\E}{\mathcal{E}}
\newcommand{\Ho}{\mathcal{H}_{\ell,0}}
\newcommand{\Hn}{\mathcal{H}_{\ell,n}}
\newcommand{\En}{\mathcal{E}_{\ell,n}}
\newcommand{\Hd}{\mathcal{H}}
\newcommand{\N}{\mathcal{N}}
\newcommand{\app}{3}
\newcommand{\ldp}{\lambda_{\tdk-P}}
\newcommand{\mup}{\mu_{k-P}}
\newcommand{\rpt}{\varrho_{k-P}(t)}
 \newcommand{\rp}{\varrho_{k-P}}
 \newcommand{\thpt}{\vartheta_{\tdk-P}(t)}
  \newcommand{\thp}{\vartheta_{\tdk-P}}

\newcommand{\rpto}{\varrho_{k-P-1}(t)}
 \newcommand{\rpo}{\varrho_{k-P-1}}
 \newcommand{\thpto}{\vartheta_{k-P}(t)}
  \newcommand{\thpo}{\vartheta_{k-P}}

\title[Soliton resolution for exterior wave maps]{Stable soliton resolution for exterior  wave maps in all equivariance classes}

\author{Carlos Kenig, Andrew Lawrie,  Baoping Liu,  \and Wilhelm Schlag}

\begin{document}
\begin{abstract} In this paper we consider finite energy $\ell$--equivariant wave maps from $\mathbb{R}^{1+3}_{t,x}\backslash{(\mathbb{R}\times B(0,1))}\rightarrow \Sp^3$ with a Dirichlet boundary condition at $r=1$, and for all $\ell \in \mathbb{N}$. Each such  $\ell$-equivariant wave map has a fixed integer-valued topological degree, and in each degree class there is a unique harmonic map, which minimizes the energy for maps of the same degree.  We prove that an arbitrary $\ell$-equivariant exterior wave map with finite energy scatters to the unique harmonic map in its degree class, i.e., soliton resolution. This extends the recent results of the first, second, and fourth authors on the $1$-equivariant equation to higher equivariance classes, and thus completely resolves a conjecture of Bizo\'{n}, Chmaj and Maliborski, who observed this asymptotic behavior numerically.   The proof relies crucially on exterior energy estimates for the free radial wave equation in dimension $d = 2 \ell +3$, which are established  in the companion paper~\cite{KLLS}. 

\end{abstract}

\thanks{Support of the National Science Foundation DMS-1265249 for the first author, and   DMS-1160817 for the fourth author is gratefully acknowledged. The second author is supported by an NSF postdoctoral fellowship.}

 \maketitle
 
\section{Introduction}
In this paper we give a complete description of  the asymptotic dynamics for the $\ell$-equivariant
wave map equation
\ant{
U: \mathbb{R}^{1+3}_{t,x}\backslash{(\mathbb{R}\times B(0,1))}\rightarrow \Sp^3, 
}
with a Dirichlet condition on the boundary of the unit ball $B(0,1) \subset \mathbb{R}^3$ and initial data of finite energy. To be precise, consider the Lagrangian
\ant{
\mathcal{L}(U, \partial_t U)=\int_{\mathbb{R}^{1+3}_{t,x}\backslash{(\mathbb{R}\times B(0,1))}}\frac12  \left(-|\partial_t U|_g^2 +\sum_{j=1}^3|\partial_x U|^2_g \right)dtdx,
}
where $g$ is the round metric on $\Sp^3$, and where we only consider functions for which the boundary of the unit cylinder $\R \times B(0, 1)$ gets mapped to a fixed point on the $3$-sphere,  i.e,  $U(t, \partial B(0,1))=N$, where $N \in \Sp^3$ is say, the north pole.
%
Under the usual $\ell$-equivariant assumption, for $\ell\in \mathbb{N}$, 
the Euler-Lagrange equation associated with this Lagrangian reduces to an equation for the azimuth angle $\psi$ measured from the north pole on $\Sp^3$, namely 
\ant{
\psi_{tt} - \psi_{rr} -\frac{2}{r}\psi_r +\frac{\ell(\ell+1)}{2r^2}\sin (2\psi) =0.
}
The Dirichlet boundary condition then becomes $\psi(t, 1) = 0$ for all $t \in \R$ and thus the Cauchy problem under consideration is, 
\EQ{ \label{EE1}
&\psi_{tt} - \psi_{rr} -\frac{2}{r}\psi_r +\frac{\ell(\ell+1)}{2r^2}\sin (2\psi) =0,\quad  r\geq 1,\\
&\psi(t,1)=0, \quad \forall t,\\
&\psi(0,r)=\psi_0(r),  \, \, \psi_t(0,r)=\psi_1(r),
}
and solutions $\vec \psi(t):= (\psi(t), \psi_t(t))$ to~\eqref{EE1} will be referred to as $\ell$-equivariant exterior wave maps. The  conserved energy for~\eqref{EE1} is given by 
\begin{equation*}
\mathcal{E}_{\ell}(\psi,\psi_t)=\int_1^\infty\frac12  \left(\psi_t^2+\psi_r^2 + \frac{\ell(\ell+1)\sin^2 \psi}{r^2}\right)r^2\,dr.
\end{equation*}
A simple analysis of the last term in the integrand above yields topological information on the wave map if we require the energy to be finite. Indeed, any $\vec \psi(t,r)$  with finite energy and continuous dependence on $t\in I=(t_0,t_1)$ must satisfy $\psi(t, \infty)=n\pi, \forall t\in I$, where $n\in \mathbb{Z}$.  Given the fact that $\psi$ measures the azimuth angle from the north pole, and $\psi(t, 1) = 0$ for all $t \in I$, this means that the integer $\abs{n}$ measures the winding number, or \textit{topological degree} of the map. Note that the case $n  \ge 0$ covers the entire range $n\in \Z$ by the symmetry $\psi \mapsto -\psi$. 

In what follows we will refer to $n \ge 0$ as the \emph{degree} of the map, and we will denote by $\E_{\ell, n}$ the connected component of the metric space of all initial data $(\psi_0, \psi_1)$ with finite energy, obeying the boundary condition $\psi_0(1)  = 0$ and of degree $n$, i.e., 
\ant{
\E_{\ell,n}=\left\{(\psi_0,\psi_1) \mid  \E_\ell(\psi_0,\psi_1)<\infty,  \, \, \psi_0(1)=0,  \, \, \lim_{r\rightarrow+\infty}\psi_0(r)=n\pi\right\}.
}

There are several appealing features of this model that make it an ideal setting in which to study soliton resolution. First, by removing the unit ball in $\R^3$ and imposing the Dirichlet boundary condition, we break the scaling symmetry. This removes the super-criticality at $r=0$ of the $3d$ wave maps problem and effectively renders the problem subcritical relative to the energy.  Global well-posedness in the energy space is then an immediate consequence. Second, the removal of the unit ball also gives rise to an infinite family of stationary solutions $(Q_{\ell,n}(r),0)$, indexed by their topological degree $n \in \mathbb{N}$; see Section~\ref{sec:hm}. In particular, the solution  $(Q_{\ell,n}(r),0)$ satisfies 
\ant{
Q_{\ell, n}(1)  = 0, \quad  \lim_{r \to \infty} Q_{\ell, n}(r)  = n \pi.
}
Moreover, $(Q_{\ell,n}(r),0)$ minimizes the energy in $\E_{\ell, n}$ and is the unique stationary solution in this degree class. Both of these features are in stark contrast to the same equation on $\R^{1+3}$ which is super-critical relative to the energy, is known to develop singularities in finite time,  and has no nontrivial finite energy stationary solutions, see for example Shatah~\cite{Shatah}, and Shatah and Struwe~\cite{ShatahStr}.

For a fixed equivariance class $\ell  \in \mathbb{N}$, the  natural topology in which to place a degree $n=0$ solution is the \emph{energy space} $\mathcal{H}_{\ell,0}=\dot{H}^1_0\times L^2 (\R^3_*)$ with  norm
\begin{equation}
\|\vec{\psi}\|_{\Ho}^2:=\int_1^\infty (\psi_t^2+\psi_r^2)\, r^2\,dr, \hspace{1cm}\vec{\psi}=(\psi,\psi_t) \label{H0norm}.
\end{equation}
Here $\R^3_*:=\R^3\backslash B(0,1)$, and $\dot{H}^1_0(\R^3_*)$ is the  completion of  smooth functions on $\R^3_* $ with compact
support under the first norm on the right-hand side of (\ref{H0norm}).  For $n\geq 1$, denote by $\Hn :=\E_{\ell,n}-(Q_{\ell,n}, 0)$ with ``norm"
\begin{equation*}
\|\vec{\psi}\|_{\Hn}:=\|\vec{\psi}-(Q_{\ell,n}, 0)\|_{\Ho}.
\end{equation*}
We remark that  the boundary condition at $  r=\infty$ is  now $\vec{\psi}-(Q_{\ell,n}, 0)\rightarrow 0$ as $r\rightarrow \infty$ with this notation.

The exterior model was first introduced in the physics literature in~\cite{BSSS}, as an easier alternative to the Skyrmion equation. Recently,~\eqref{EE1} was proposed by  Bizon, Chmaj, and Maliborski in~\cite{BCM12} as a model to study the problem of relaxation to the ground states given by various equivariant harmonic maps.   Both \cite{BSSS, BCM12} stress the analogy of the stationary equation with that of the damped pendulum by demonstrating the existence and uniqueness of the ground state harmonic maps via a phase-plane analysis.  The numerical simulations in \cite{BCM12} indicate that for each equivariance class $\ell\geq 1$,  and each topological class $n\geq 0$, every solution scatters to the unique harmonic map $Q_{\ell,n}$ that lies in $\E_{\ell,n}$, giving evidence that the soliton resolution conjecture  holds true in this exterior model.  More recently, this conjecture was verified for $1$-equivariant (or  co-rotational) exterior wave maps with topological degree $n =0$ by the second and fourth author in~\cite{LS13}, and then for all topological degrees $n \ge 0$ by the first, second and fourth authors in~\cite{KLS}.



In this paper, we  verify the \textit{soliton resolution conjecture} for the exterior wave map problem  for all topological degree classes $n \ge 0$ in the remaining equivariant classes, $\ell \ge 2$.   Our main result is as follows. 
\begin{thm} \label{MainThm}For any smooth energy data in $\E_{\ell,n}$ there exists a unique global smooth solution to \textnormal{(\ref{EE1})}, which scatters to the harmonic map ($Q_{\ell,n}, 0$).
\end{thm}

Here ``scattering to the harmonic map $(Q_{\ell, n}, 0)$" means that for each solution $\vec \psi(t)$ to~\eqref{EE1} we can find solutions $\vec \fy_L^{\pm}$ to the linear equation 
\ant{
\varphi_{tt}-\varphi_{rr} -\frac{2}{r}\varphi_{r}+\frac{\ell(\ell+1)}{r^2}\varphi=0,  \quad r\geq 1,  \quad \varphi(t,1)=0.
} 
so that 
\begin{equation*}
\vec \psi(t)=(Q_{\ell,n},0) + \vec \fy_L^{\pm}(t) + o_{\Ho}(1), \mas  t \to  \pm \infty.
\end{equation*}




We emphasize that only the scattering statement in Theorem~\ref{MainThm} is difficult to prove. We employ the concentration compactness/rigidity method developed by the first author and Merle in~\cite{KM06, KM08}.  Given the previous work,~\cite{LS13, KLS}, we can quickly reduce the proof of Theorem~\ref{MainThm} to the rigidity argument, where the goal is to show that any solution to~\eqref{EE1} with a pre-compact trajectory in the energy space must be a harmonic map; see Sections~\ref{sec:pre},~\ref{sec:cc}. In Section~\ref{sec:rig}, the  rigidity argument is carried out using a version of the `channels of energy' argument introduced by Duyckaerts, the first author, and Merle in~\cite{DKM4, DKM5}. Here we adapt the approach in~\cite{KLS} to all higher equivariance  classes. The proof relies crucially on~\emph{exterior energy estimates} for the free radial wave equation in dimension $d = 2 \ell +3$ where $\ell$ is the equivariance class. These estimates were established  in~\cite{DKM1} for dimension $d=3$, in~\cite{KLS} for dimension~$d=5$,  and in the companion paper~\cite{KLLS} for \emph{all odd dimensions}; see Theorem~\ref{Exterior-Estimate}.




\section{Preliminaries}\label{sec:pre}
In this section we briefly review a few basic properties of the harmonic maps $Q_{\ell,n}$, and reduce the $\ell$-equivariant wave map problem to an exterior semi-linear wave equation in $\R_*^{d}:=\R^d\backslash B(0,1)$,   with a Dirichlet boundary condition at $r=1$, and with $d := 2\ell+3$. 

\subsection{Exterior harmonic maps} \label{sec:hm}
In each energy class $\E_{\ell,n}$, there is a unique finite energy exterior harmonic map, which is a minimizer of the energy $\E_{\ell,n}$ and also a static solution 
  to (\ref{EE1}), i.e.
\EQ{
\label{HarmonicMap} &\partial_{rr}Q_{\ell,n} +\frac{2}{r}\partial_rQ_{\ell,n}  = \frac{\ell(\ell+1)}{2r^2}\sin (2Q_{\ell,n})\\
&Q_{\ell,n}(1)=0,  \quad 
\lim_{r\rightarrow \infty}Q_{\ell,n}(r)=n\pi
}
As in~\cite{KLS} we  change variables, setting  $s:=\log r$, and  $\phi(s):=Q_{\ell,n}(r)$. The equation~\eqref{HarmonicMap} becomes 
\begin{equation} \label{ode} \phi_{ss}+\phi_{s}=\frac{\ell(\ell+1)}{2}\sin (2\phi), \hspace{0.5cm}\phi(0)=0, \hspace{0.5cm}\phi(\infty)=n\pi,
\end{equation}
which in mechanics describes the motion of a pendulum with constant friction. Noting that~\eqref{ode} can be written as an autonomous system in the plane, we can perform  a standard analysis of the phase portrait to deduce the following result -- we refer the reader to~\cite[Lemma~$2.1$]{KLS} for the details of the identical argument when $\ell =1$. 
\begin{lemma} \label{lem:Q-asymptotic}For all $\alpha\in \R$, there exists a unique solution $Q_{\ell, \alpha}\in \dot{H}^1(\R^3_*)$ to \textnormal{(\ref{HarmonicMap})}  with
\begin{equation}\label{Q-asymptotic}
Q_{\ell, \alpha}(r)=n\pi-\frac{\alpha}{r^{\ell+1}} + O(r^{-3(\ell+1)})\hspace{1cm}\text{ as }r\rightarrow \infty
\end{equation}
The $O(\cdot)$ is uniquely determined by $\alpha$ and vanishes for
$\alpha=0$. Moreover, there exist a unique $\alpha_0>0$ such that
$Q_{\alpha_0}(1)=0$, we will denote it as $Q_{\ell,n}$.
\end{lemma}

\subsection{Reduction to an exterior wave equation in high dimensions}\label{Reduction} At this point we fix an arbitrary equivariance class $\ell  \ge 2$ and topological degree $n  \ge 0$ for the remainder of the paper. We reduce~\eqref{EE1} to a semi-linear equation in ~$\R^{2 \ell +3}_*$. To perform this reduction, we first linearize~\eqref{EE1} about the unique $\ell$-harmonic map of degree $n$, $Q_{\ell, n}$. As we have fixed $\ell$ and $n$, we will simplify notation by writing $Q= Q_{\ell, n}$ and we note that when $n =0$ we have $Q \equiv 0$.


For each solution $\vec \psi$ to~\eqref{EE1} we define $\vec \fy$ by 
 \begin{equation*}
\vec  \psi: =(Q, 0) + \vec \fy .
 \end{equation*}
Using the equations for $\vec \psi$ and for $Q$ we see that $\vec \fy$ solves 
 \EQ{ \label{eq:phi}
& \varphi_{tt}-\varphi_{rr}-\frac{2}{r}\varphi_r + \frac{\ell(\ell+1)\cos(2Q)}{r^2}\varphi = Z(r,\varphi)\\
 &\varphi(t,1)=0,  \quad \varphi(t,\infty)=0 \hspace{1cm}\forall t,\\
 &\vec{\varphi}(0)=(\psi_0-Q, \psi_1),
  }
where here 
\[Z(r,\varphi):=\frac{\ell(\ell+1)}{2r^2}[2\varphi- \sin (2\varphi)]\cos (2Q) +(1-\cos(2\varphi))\sin 2Q
\]
The left-hand-side of~\eqref{eq:phi} has more dispersion than a wave equation in $3d$ due to the strong repulsive potential 
\ant{
 \frac{\ell(\ell+1)\cos(2Q)}{r^2} =  \frac{ \ell(\ell+1)}{r^2}  + O(r^{-2 \ell -4}) \mas r \to \infty
}
where we have used the asymptotic behavior of $Q$ from~\eqref{Q-asymptotic} in the expansion above. Indeed, the coefficient $\ell(\ell+1)$ in front of the $r^{-2}$ term indicates that we have the 
the same dispersion as a $d=2\ell+3$-dimensional wave equation. This is made precise by the following standard reduction.

We define $\vec u$ by setting  $\varphi=r^\ell u$.  Then ${\vec u}$ solves the following equation. 
\EQ{
&u_{tt} - u_{rr} - \frac{2\ell+2}{r}u_r + V(r)u  = \N(r,u), \hspace{1cm} r\geq 1\\
&u(t,1)=0, \hspace{1cm}\forall t\in \R\\
&\vec{u}(0) =(u_0, u_1)
\label{u eq}
}
where
\begin{equation}
\begin{aligned}
V(r) &: = \frac{\ell(\ell+1)(\cos2Q -1)}{r^2}
\\
\N(r,u)& := F(r, u) + G(r, u)\\
F(r,u) &:= \frac{\ell(\ell+1)}{r^{\ell+2}}  \sin^2(r^\ell u)\sin 2Q\\
G(r,u) &:= \frac{\ell(\ell+1)}{2r^{\ell+2}}(2r^\ell u-\sin (2r^\ell u))\cos 2Q
\end{aligned}\label{VFG}
\end{equation}
The potential   $V(r)$ is real-valued, radial, bounded, and smooth. Using the asymptotics of $Q$ in  (\ref{Q-asymptotic}), we see that $V$ has the asymptotics
\begin{equation}\label{Vbound}
V(r) =O(r^{-2\ell-4}), \hspace{1cm} \text{as } r\rightarrow \infty
\end{equation}
For the nonlinearity~$\NN = F+ G$ we have 
\begin{equation}
\begin{aligned}
|F(r,u)| &\leq  C_0 r^{-3} |u|^2\\
|G(r,u)| &\leq C_0 r^{2\ell-2}|u|^3
\end{aligned}\label{FGbound}
\end{equation}
The constant $C_0$ here depends only on $d=2\ell+3$ and $Q$.

We will consider  radial initial data  $$(u_0, u_1)\in\mathcal{H} :=\dot{H}^1_0\times L^2 (\R^d_*),$$ where
 \begin{equation*}
\|(u_0,u_1)\|_{\Hd}^2 :=\int_1^\infty  [(\partial_r u_0(r))^2+u_1(r)^2] \, r^{2\ell+2}\, dr
\end{equation*}
and  $\dot{H}^1_0(\R^d_*)$ is the completion under the first norm on the right-hand side above
of all smooth radial compactly supported functions on  $\R^d_*$, with $d = 2 \ell + 3$.

For the remainder of the paper we will work exclusively in the ``$u$-formulation", \eqref{u eq}, rather than with the $\ell$-equivariant wave map angle $\psi(t, r)$ as in~\eqref{EE1}. 
In fact, we claim that   
the Cauchy problem (\ref{EE1}) with data $(\psi_0, \psi_1)\in \E_{\ell,n}$ is equivalent to the problem (\ref{u eq}) with initial data
\begin{equation*}
\HH \ni 
(u_0, u_1): =\frac{1}{r^\ell}(\psi_0-Q, \psi_1).
\end{equation*}
To see this, it remains to  prove that 
\begin{equation} \label{fy u}
\|\vec{u}\|_{\Hd}\simeq \|\vec{\psi}\|_{\Hn}
\end{equation}
Indeed, if we set $\varphi_0=\psi_0-Q= r^{\ell} u$, then we have
\[\varphi_r=r^\ell u_r + \ell r^{\ell-1}u =r^\ell u_r + \ell \frac{\varphi}{r}\]
Note the following versions of Hardy's inequality
\ant{
&\int_1^\infty \varphi^2(r) \, dr\lesssim \int_1^\infty \varphi_r^2(r) \,  r^2 \,  dr \\
&\int_1^\infty u^2(r)\, r^{2\ell}\, dr\lesssim \int_1^\infty u_r^2(r)\,r^{2\ell+2} \, dr
}
These are  proved first for radial functions $\fy \in C^{\infty}_0( \R^3_*)$, respectively $u \in C^{\infty}_0(\R^{2 \ell +3}_*)$ via integration by parts, using the boundary condition $\varphi(1)=0$ and $u(1) = 0$. 
Hence, from~\eqref{fy u} we obtain  
\[\int_1^\infty \varphi_r^2(r) \, r^2\,dr\simeq \int_1^\infty u_r^2(r) \,  r^{2\ell+2}dr.\]
Therefore for each class $\E_{\ell,n}$, the map
\[(\psi_0,\psi_1)\rightarrow \frac{1}{r^\ell} (\psi_0 -Q(r), \psi_1)\]
is an isomorphism between the spaces $\En$ and $\Hd$ respectively.

We thus prove the analogous version of Theorem~\ref{MainThm} in the ``$u$-formulation."  It is clear from the definition of $\vec u(t)$ that Theorem~\ref{MainThm} is true if and only if every solution $\vec u(t)$ to~\eqref{u eq} scatters as $t \to \pm \infty$. Scattering here means that solutions to~\eqref{u eq} approach  free waves in $ \R \times \R^d_*$ in the space $\HH$. A free wave in this context is a solution to~\eqref{u eq} with $V =  \NN = 0$. 

 We prove the following equivalent reformulation of Theorem~\ref{MainThm}.

\begin{thm}\label{thm:main} For any initial data $\vec u(0) \in \HH$, there exist a unique, global-in-time solution $\vec u(t) \in \HH$  to~\eqref{u eq}. Moreover $\vec u(t)$ scatters to free waves as  $t \to \pm \infty$. 
\end{thm} 


\section{Small data theory and Concentration compactness} \label{sec:cc}

The proof of Theorem~\ref{thm:main}, and hence  of  the equivalent statement Theorem~\ref{MainThm}, proceeds via the concentration compactness/ rigidity method introduced by the first author and Merle in~\cite{KM06, KM08}. The argument can be divided into three separate steps, namely $(1)$ a small data theory, i.e., a proof of Theorem~\ref{thm:main} for initial data with small enough $\HH$-norm; $(2)$ a concentration compactness argument. If Theorem~\ref{thm:main} fails, then there exists a critical element, which is a minimal non-scattering solution with a pre-compact trajectory in $\HH$. The critical element is nonzero, by step $(1)$.  The main ingredient here is an analogue of the  nonlinear profile decomposition of Bahouri and Gerard,~\cite{BG},  adapted to~\eqref{u eq} along with a nonlinear perturbation theory; and finally $(3)$ a rigidity argument. Here one shows that any solution with a pre-compact trajectory as in step $(2)$ must be identically $ \equiv 0$. This contradicts the existence of the critical element from step $(2)$ and completes the proof. 

In this section we give a very brief outline of steps $(1)$ and $(2)$ above. We omit many details, since nearly identical arguments have been carried out in detail for the case $\ell =1$ in both~\cite{LS13} and~\cite{KLS}. 

\subsection{Small data theory}\label{small-cauchy} In this section we establish the small data theory for~\eqref{u eq}. The main ingredient  here are Strichartz estimates for the linear inhomogeneous wave equation perturbed by the radial potential $V$. Indeed, consider 
\EQ{
&u_{tt} - u_{rr} - \frac{d-1}{r}u_r + V(r)u  = \N, \hspace{1cm} r\geq 1\\
&u(t,1)=0, \, \, \, \forall t, \quad 
\vec{u}(0) =(u_0, u_1)\in \Hd.
\label{Lu eq}
}
The conserved energy for \eqref{Lu eq} with  $\NN=0$ is  
\begin{align*}
\E_L(u, u_t) = \frac{1}{2} \int_1^{\infty} (u_t^2 + u_r^2 + V(r)u^2) \, r^{d-1} \,dr
\end{align*}
As shown in~\cite{LS13, KLS} this energy has an important positive definiteness property, namely, 
\ant{
\E_L(u, u_t) = \frac{1}{2} ( \|u_t\|_2^2 + \ang{ H u \mid u} ),\quad H= -\Delta +V
}
It is shown in~\cite{BCM12, LS13} that $H$ is a nonnegative self-adjoint operator in $L^2(\R^d_*)$ (with a Dirichlet condition
at $r=1$), and moreover, that the threshold energy zero is regular; this means that if $Hf=0$ where $f\in H^2\cap \dot H^1_0$
then $f=0$. It is standard to conclude from this spectral information that for some constants $0<c_1<c_2$, 
\ant{
c_1 \|  f\|_{\dot H^1_0 }^2\le  \ang{ Hf\mid f}\le c_2\|  f\|_{\dot H^1_0 }^2 \quad \forall\;f\in  \dot H^1_0(\R^d_*) 
}
In the sequel we will sometimes write $\| \vec u\|_{\E}^2:=\E_L(\vec u)$, which satisfies 
\EQ{\label{norm comp}
\| \vec u\|_{\E} \simeq \| \vec u\|_{\HH} \quad \forall \vec u\in\HH(\R^d_*)
}

We call a triple $(p,q,\gamma)$ admissible if
\begin{equation*}
p>2, q\geq 2,  \quad 
\frac{1}{p}+\frac{d}{q}=\frac{d}{2}-\gamma, \quad 
\frac{1}{p}\leq \frac{d-1}{2}(\frac12-\frac{1}{q})
\end{equation*}
\begin{thm}[Strichartz estimate]\label{Strichartz} Let $(p,q,\gamma), (r,s,\rho)$ be admissible triples, then any solution ${u}$ to \textnormal{(\ref{Lu eq})} with radial initial data satisfies
\begin{equation*}
\||\nabla|^{-\gamma}\nabla u\|_{L^p_tL^q_x(\R^d_*)}\lesssim \|\vec{u}(0)\|_{\Hd}+\||\nabla|^\rho  \N\|_{L^{r'}_tL^{s'}_x(\R^d_*)}, 
\end{equation*}
where here $r', s'$ are the conjugates of $r, s$.
\end{thm}

\begin{remark}The case when potential $V=0$ is proved 
in~\cite{Sogge1}.  The case with $V$ as in~(\ref{VFG}) can be proved by adapting the argument in~\cite[Proposition 5.1]{LS13} to dimension $d = 2 \ell + 3$. As in~\cite{LS13} the proof can be reduced to deducing localized energy estimates for~\eqref{Lu eq}  with $V$ as in~\eqref{VFG}. The local energy estimates are proved using the distorted Fourier transform relative to the self-adjoint Schr\"odinger operator $H_V = -\Delta + V$ on $L^2(\R^d_*)$, and rely crucially on decay properties of the corresponding spectral measure.  It is essential here $ H$ has no negative spectrum and that the edge of the continuous spectrum for $H$ is neither  an eigenvalue nor a resonance; see~\cite[Section 5]{LS13} for more details. 
\end{remark}

A standard consequence of the Strichartz estimates is the following small data scattering theory. For a time interval $I$, we denote by $S(I)$ the space $S(I):= L^{\frac{2(d+1)}{d-2}}_{t, x}( I \times \R^d_*)$ with norm 
\EQ{
\| u \|_{S(I)}:= \| u\|_{L^{\frac{2(d+1)}{d-2}}_{t, x}( I \times \R^d_*)}
}

\begin{thm} \label{thm:gwp}The exterior Cauchy problem for~\eqref{u eq} is globally well-posed in $ \HH:=\dot H^1_{0}\times L^2(\R^d_*)$. Moreover, a solution $\vec u$ scatters as $t\to\pm \infty$ to free waves, i.e., solutions $\vec u_L^{\pm}\in \HH$ of  
\EQ{\label{uL free}
\Box u_L^\pm =0, \; r\ge1, \;  u_L^{\pm}(t, 1)=0, \; \forall t\ge0
}
if and only if $$\|u\|_{S(\R_{\pm})}<\infty,$$ where $\R_+:=[0, \infty)$ and $\R_-:=(- \infty, 0]$. In particular, there exists a constant $\delta_0>0$ small 
so that if $\|\vec u(0)\|_{\HH}<\delta_0$, then $\vec u$ scatters to free waves as $t \to\pm\infty$. 
\end{thm}

\begin{remark} \label{rem:st}
Theorem~\ref{thm:gwp} follows from the standard contraction mapping argument based on Theorem~\ref{Strichartz} and we omit the proof here. The reason that we can use a Strichartz norm with the same scaling as the energy space is that the nonlinearity $\NN$ is effectively subcritical with respect to energy for radial functions on  $\R^{d}_*$. 
Indeed, the key point here is the Strauss estimate which gives  
\ant{
 \abs{f(r)}  \le C r^{\frac{2-d}{2}} \| f \|_{\dot{H}^1_0(\R^d_*)}
}
for radial functions $f \in \dot{H}^1_0(\R^d_*)$. To use this to prove Theorem~\ref{thm:gwp} for $\ell = 1$ and $d=5$ see the argument after~\cite[Proposition~$3.2$]{KLS}. For $\ell \ge 2$  we have $d \ge 7$ and one can use the more delicate arguments from~\cite{BCLPZ} and the harmonic analysis machinery on exterior domains from~\cite{KVimrn}; see also~\cite{R14}. As a heuristic we note that using~\eqref{FGbound} with $d = 2 \ell +3$ together with the Strauss estimate applied to the nonlinearity $\NN(r, u)  = F(r, u) + G(r, u)$ one has,  
\EQ{ \label{grupw}
\abs{ G(r, u)} \lesssim r^{d-5}\abs{ u}^{\frac{2(d-4)}{d-2}}  \abs{u}^{\frac{d+2}{d-2}} \lesssim r^{-1} \| u\|_{\dot{H}^1_0( \R^d_*)}^{\frac{2d-8}{d-2}}   \abs{u}^{\frac{d+2}{d-2}} \\
\abs{F(r, u)} \lesssim r^{-3} \abs{u}^{\frac{d-6}{d-2}} \abs{u}^{\frac{d+2}{d-2}} \lesssim r^{-\frac{d}{2}} \|u\|_{\dot{H}^1_0(\R^d_*)}^{\frac{d-6}{d-2}}  \abs{u}^{\frac{d+2}{d-2}}
}
where here, as always, we have $r \ge 1$. Note that $(d+2)/(d-2)$ is the energy critical power in $\R^{1+d}$. 
\end{remark}

\subsection{Concentration compactness} By the  concentration compactness methodology in~\cite{KM06, KM08},
 we claim that if Theorem~\ref{MainThm}, (and hence Theorem~\ref{thm:main}), fails,
 we can construct a \textit{critical element}, which is a global non-scattering solution to~\eqref{u eq} with minimal energy and has a pre-compact trajectory in $\HH$. 
Indeed, following an argument that is identical to~\cite[Proof of Proposition~$3.6$]{KLS} we deduce the following result. See also~\cite{Bu}. 

\begin{prop}\label{critical-element}
Suppose that Theorem~\ref{MainThm} fails.  Then there exists a
nonzero, global  solution $\vec{u}_*(t) \in \HH$ to \textnormal{(\ref{u eq})},  such that the trajectory
\[\mathcal{K}:=\{\vec{u}_*(t)|t\in \R\}\]
is pre-compact in $\Hd=\dot{H}^1\times L^2(\R^{d }_*)$. We call $\vec u_*(t)$ a critical element. 
\end{prop}

The key ingredients in the proof of Proposition~\ref{critical-element} are a Bahouri-G\'erard profile decomposition and a nonlinear perturbation theory, see~\cite[Lemma~$3.4$ and Lemma~$3.5$]{KLS}. We omit the details and just  formulate the concentration compactness principle relative to the linear
wave equation with a potential, i.e., ~\eqref{Lu eq} with $\NN=0$. 
We note that any solution to \eqref{Lu eq} with $\NN=0$, which
is in $S(\R)$ must scatter to ``free" waves, where ``free" is in the sense of Theorem~\ref{thm:gwp}.

\begin{lemma}\label{lem:BGd}\cite[Lemma $3.4$]{KLS}
Let $\{u_n\}$ be a sequence of free radial waves, which are uniformly bounded in $\HH=\dot H^1_{0}\times L^2(\R^d_*)$. After passing to a subsequence,  
there exists a sequence of solutions $V^j_L$ to~\eqref{Lu eq} with $\NN = 0$, which are bounded in $\HH$, and sequences of times $\{t_n^j\} \subset \R$ such that for  errors $w_{n, L}^k$ defined by
\ant{
  u_n(t) = \sum_{1\le j<k} V^j_L(t-t_n^j) + w_{n, L}^k(t) 
  }
we have for any $j<k$, 
\ant{
\vec w_{n, L}^k(t_n^j) \rightharpoonup 0
}
 weakly in $\HH$ as $n\to\I$,  as well as 
\ant{
 \lim_{n\to \infty} |t_n^j-t_n^k| = \infty    
}
and the errors $w_{n}^{k}$ vanish asymptotically
in the sense that
\ant{
 \pt \lim_{k\to \infty} \limsup_{n\to \infty} \|w_{n, L}^k\|_{(L^\I_tL^p_x\cap S)(\R\times\R^d_*)}=0 \quad \forall \; \frac{2d}{d-2}<p<\infty 
}
Finally, one has orthogonality of the free energy with a potential, cf.~\eqref{norm comp}, 
\ant{
 \| \vec u_n \|_{\E}^2 = \sum_{1\le j<k} \|\vec V^j_L\|_{\E}^2 + \|\vec\ga_n^k\|_{\E}^2 +o_n(1)
} 
as $n\to\infty$. 
\end{lemma}

\section{Exterior energy estimates}\label{Chapter-exterior}
In this section, we recall the exterior energy estimates proved
in~\cite{KLLS} for the free radial wave equation in $\R^{1+d}$ .

\subsection{Exterior energy estimates for free radial waves in all odd dimensions} \label{sec:ext}
 We note that we will use  the notation $[x]$ denotes the largest integer $k \in \Z, k \leq x$.
\begin{thm}\cite[Theorem~$2$]{KLLS} \label{Exterior-Estimate}In any odd dimension $d>0$, every radial energy solution of $\Box u=0, u(0)=f, u_t(0)=g $ in $\R^{1+d}_{t,x}$   satisfies  the following estimate:
For every $R>0$
\begin{equation}{\label{extd} \max_{\pm}\lim_{t\to\pm\infty} \int_{r\geq |t|+R}
|\nabla_{t,x} u(t,r)|^2\, r^{d-1}\, dr \ge \frac12 \|
\pi_{R}^\perp\, (f,g)\|_{\dot H^1\times L^2(r\geq R;r^{d-1}dr)}^2
}\end{equation} Here
 \[P(R):=\textnormal{span}\left\{(r^{2k_1-d},0), (0,r^{2k_2-d}) \mid k_1=1, 2,\cdots [\frac{d+2}{4}]; k_2=1, 2,\cdots [\frac{d}{4}]\right\}\]
and $\pi_{R}^\perp$ denotes the orthogonal projection onto the
complement of the plane $P(R)$ in $(\dot H^1\times
L^2)(r\geq R;r^{d-1}dr)$.

The inequality becomes an equality for data of the form $(0,g)$
and $(f,0)$. Moreover,  the left-hand side of~\eqref{extd}
vanishes exactly for all data in~$P(R)$.
\end{thm}

We remark  that, the \textit{Cauchy matrix}
played an important role in the proof of this theorem.  Let us
review a few facts from linear algebra that will be important in the upcoming sections.

\begin{enumerate}
\item
Let $a_1,\cdots, a_k$ be linearly independent vectors in an  inner product space $(\cV, \ang{ \, , \, })$, which span the subspace $W$, that is,
  \ant{
  W=\spa\{a_1,\cdots a_k\}
  }
  For any vector $u\in \cV$, the orthogonal projection onto $W^{\perp}$ can be written as 
   \ant{
  \Proj_{W^\perp} u =u -(\lambda_1 a_1 +\cdots \lambda_k a_k) 
  }
where the coefficients satisfy
\ant{
\ang{ \Proj_{W^\perp} u, a_j\rangle =\langle   u, a_j} -\sum_i \lambda_i\ang{ a_i, a_j}=0
}
We set 
 \ant{
 U:=\bmat{\langle u, a_i \rangle}_{1\times k}, \quad \La =\bmat{\lambda_i}_{1\times k},\quad A=\bmat{\langle a_i, a_j\rangle}_{k\times k}
 }
so that 
\EQ{ \label{Lambda-formula}
U=\La A
}
Using that $A$ is symmetric, invertible and positive definite, we  compute
\EQ{ \label{proj0}
\|\Proj_{W^\perp} u\|^2 &= \ang{ u, u} -\sum_{i,j=1}^k \lambda_i\lambda_j \langle a_i, a_j\rangle \\
& =\langle u,u \rangle -\Lambda A \Lambda^t =\langle u,u \rangle -  U A^{-1}U^t
}
Let us simplify the notation by setting $A=[a_{ij}]$, $B=A^{-1}=[b_{ij}]$,
\EQ{\label{proj}
\|\Proj_{W^\perp} u\|^2 = \ang{ u, u} -   \sum_{i,j} b_{ij} \ang{ u,a_i } \ang{ u,a_j }
}

\item
Next, we introduce the  \textbf{Cauchy Matrix}~\cite{CauchyMatrix}, which is an $m\times m$ matrix of the form 
\ant{
A=\bmat{\frac{1}{x_i-y_j}}, \quad x_i-y_j\not=0; \, \, 1\le i,j\le m 
}
Its determinant  can be computed to be 
\ant{
\det(A)=\frac{\prod_{i<j} (x_i-x_j)(y_j-y_i)}{\prod_{i,j=1}^m
(x_i-y_j)}
}
from which we conclude that the Cauchy matrix is invertible. 
Using Cramer's rule,  we can obtain an explicit formula for its
inverse matrix 
\EQ{
\label{CauchyInverse}B=[b_{ij}]=A^{-1}\quad b_{ij} =(x_j-y_i) A_j(y_i)B_i(x_j)
}
 where $A_i(x)$ and $B_j(y)$ are the Lagrange polynomials for $(x_i), (y_j)$ respectively, that is, 
 \EQ{
 \label{poly} A_i(x) &=\frac{\prod_{\ell\not =i} (x-x_\ell)}{\prod_{\ell\not = i}(x_i-x_\ell)} =\prod_{1\leq \ell\leq  m, \ell\not = i} \frac{x-x_\ell}{x_i-x_\ell}\\
 \displaystyle B_j(y) &= \prod_{1\leq \ell\leq  m, \ell\not = j} \frac{y-y_\ell }{y_j-y_\ell }
 }
 \item
 Now we compute the explicit formula for the projection in Theorem~\ref{Exterior-Estimate}.
 If we set $\cV=L^2(r\ge R, r^{d-1}\,dr)$ 
with $a_i=r^{2i-d}$  
\EQ{\label{g-w}
W=\spa\Big\{r^{2i-d}\mid i=1,\cdots k=[\frac{d}{4}]; r \geq R\Big\}
}
  \EQ{ \label{gdata-A}
  a_{ij}(R)=\langle r^{2i-d}, r^{2j-d}\rangle_\cV =\int_R^\infty r^{2i-d +2j-d}r^{d-1}\,dr =\frac{R^{2i+2j-d}}{d-2i-2j}
  }
and we have a Cauchy matrix when $R=1$
\[A(1)=\Big[\frac{1}{d-2i-2j}\Big]_{k\times k}\]
Let $x_i=d-2i, y_j =2j$. Then using \eqref{CauchyInverse} and  \eqref{poly}, we deduce that 
\begin{equation}\label{Bformula}\begin{aligned}
b_{ij}(1) = & (d-2i-2j)\frac{\prod_{1\le \ell\le  k, \ell\not =j} (2i+2\ell-d)}{\prod_{1\leq \ell\le  k, \ell\not =j} (2\ell-2j)}\frac{\prod_{1\le \ell\le  k, \ell\not =i} (2j+2\ell-d)}{\prod_{1\le \ell\le  k, \ell\not =i} (2\ell-2i)}\\
=&\frac{1}{d-2i-2j}\frac{\prod_{1\le \ell\le  k } (2i+2\ell-d)}{\prod_{1\leq \ell\le  k, \ell\not =j} (2\ell-2j)}\frac{\prod_{1\le \ell\le  k} (2j+2\ell-d)}{\prod_{1\le \ell\le  k, \ell\not =i} (2\ell-2i)}
\end{aligned}\end{equation}
We thus obtain the inverse $B(R)=A(R)^{-1}$ with 
\EQ{ \label{gdata-B}
b_{ij}(R)=b_{ij}(1)R^{d-2i-2j}.
} 
Moreover,  we have established   the projection formula
\ant{
& \|\mathrm{Proj}_{W^{\perp}}g\|^2_{L^2(r\geq R, r^{d-1}\,dr)}\\
=&\int_R^\infty g^2(r) r^{d-1}\,dr -\sum_{i,j=1}^k\frac{R^{d-2i-2j}}{d-2i-2j}c_ic_j\int_R^\infty g(r)r^{2i-1}\,dr\int_R^\infty \overline{ g(r)}r^{2j-1}\,dr 
}
with 
\EQ{\label{ci-formula}
c_j=\frac{\prod_{1\leq \ell\leq  k } (d-2j-2\ell)}{\prod_{1\leq \ell\leq  k, \ell\not =j} (2\ell-2j)}, \quad 1\leq j\leq k=[\frac{d}{4}]
}

\vspace{\baselineskip}

\noindent Next,  let $\ti{V}=\dot{H}^1(r\geq R, r^{d-1}\,dr)$.  With   $\tilde{a}_i=r^{2i-d}$ one has 
\EQ{\label{tdw} 
\ti{W}=\spa\Big\{r^{2i-d}\mid i=1,\cdots \tdk:=[\frac{d+2}{4}]; r\geq R\Big\}
}
An identical computation as before gives  
\EQ{ \label{fdata-A}
\ti{a}_{ij}(R)=(2i-d)(2j-d)\frac{R^{2i+2j-d-2}}{d+2-2i-2j}=(2i-d)(2j-d)R^{2i+2j-d-2} \alpha_{ij}
}
We find the inverse of the Cauchy matrix $ \bmat{\al_{ij}}_{\ti k \times  \ti k}  := [\frac{1}{d+2-2i-2j}]_{\tdk\times \tdk}$, which is
\ant{
\bmat{\frac{1}{d+2-2i-2j}}_{\tdk\times \tdk}^{-1} =\bmat{ \frac{1}{d+2-2i-2j}\frac{\prod_{1\leq \ell \leq \tdk}(d+2-2\ell-2i)\prod_{1\leq \ell \leq \tdk}(d+2-2\ell-2j)}{\prod_{1\leq \ell\leq  \tdk, \ell\not=i}(2\ell-2i)\prod_{1\leq \ell\leq  \tdk, \ell\not=j}(2\ell-2j) }}_{\tdk\times \tdk}
}
This yields the inverse for $\tilde{A}(R)=[\tilde{a}_{ij}(R)]_{\tdk\times \tdk}$, which we denote by $\tilde{B}(R)=[\tilde{b}_{ij}(R)]_{\tdk\times \tdk}$ where 
\begin{multline} \label{fdata-B}
 \tilde{b}_{ij}(R) = \frac{R^{d+2-2i-2j}}{(d-2i)(d-2j)}\frac{1}{d+2-2i-2j} \times \\  \frac{\prod_{1\leq \ell \leq  \tdk}(d+2-2\ell-2i) \prod_{1\leq \ell \leq \tdk}(d+2-2\ell-2j)}{\prod_{1\leq \ell\leq  \tdk, \ell\not=i}(2\ell-2i)\prod_{1\leq \ell\leq  \tdk, \ell\not=j}(2\ell-2j) }  
 \end{multline}
 We thus have the projection formula
\ant{
 &\|\mathrm{Proj}_{\tilde{W}^{\perp}}f\|^2_{\dot{H}^1(r\geq R, r^{d-1}\,dr)}\\
=&\int_R^\infty |f'(r)|^2 r^{d-1}\,dr - \sum_{i,j=1}^{\tdk}\frac{R^{d+2-2i-2j}}{d+2-2i-2j}d_id_j\int_R^\infty f'(r) r^{2i-2}\,dr \int_R^\infty \overline{f'(r)} r^{2j-2}\,dr
}
 with
\EQ{
\label{di-formula}d_j =\frac{ \prod_{1\leq \ell \leq  \tdk}(d+2-2\ell-2j)}{ \prod_{1\leq \ell\leq  \tdk, \ell\not=j}(2\ell-2j) }, \quad 1\leq j\leq \tdk=[\frac{d+2}{4}]
}
\end{enumerate}
\begin{remark} From now on, the spaces $W, \ti{W}$ will be fixed as in \eqref{g-w}, \eqref{tdw} and the formulas for $c_i, d_i$ will be fixed as in~\eqref{ci-formula},~\eqref{di-formula}.  \end{remark}

Now let us  collect some useful facts  concerning the coefficients $c_j, d_j$. 
\begin{lemma}\cite[Lemma $15$]{KLLS}\label{contour-integral} Given the coefficients  $c_j, 1\leq j\leq k=[\frac{d}{4}]$ and $d_j, 1\leq j\leq \tdk=[\frac{d+2}{4}]$ defined as in  {\eqref{ci-formula}, \eqref{di-formula}}, we have the following identities
\EQ{
\label{ci-identity-1}\sum_{j=1}^{k}\frac{c_j}{d-2m-2j}=1, 
\quad  \textrm{ for any } m\in \Z, 1\leq m\leq k 
}
\EQ{\label{ci-identity-2}\sum_{j=1}^{k}\frac{c_j}{2j}+1=\prod_{\ell=1}^k\frac{d-2\ell}{2\ell}
}
Similarly we have 
\EQ{\label{di-identity-1}\sum_{j=1}^{\tdk}\frac{d_j}{d+2-2m-2j}=1, 
\quad  \text{ for any } m\in \Z, 1\leq m\leq \tdk
}
\EQ{\label{di-identity-2}\sum_{j=1}^{\tdk}\frac{d_j}{2j}+1=\prod_{\ell=1}^{\tdk}\frac{d+2-2\ell}{2\ell}
}
\end{lemma}
Lemma~\ref{contour-integral} can be proved using contour integration. We refer the reader to~\cite[Proof of Lemma~$15$]{KLLS} for the precise details of the argument.  

\subsection{Algebraic identities}
The exterior energy estimates in Theorem~\ref{Exterior-Estimate} will play a crucial role in the proof of Theorem~\ref{thm:main}. 
Here we prove a few algebraic identities to relate the exterior energy of the  projected solution $\pi_R^{\perp} \vec u(t, r)$ with the projection coefficients. The relations proved in this subsection will be essential  ingredients in our adaptation of  the argument from~\cite[Section $5$]{KLS} to all equivariance classes.  We note that in~\cite{KLS}, the $\ell =1$ equivariant case treated there corresponds to $d = 5$ in the previous subsection and the subspace $P(R)$ from Theorem~\ref{Exterior-Estimate} is simply a plane.

Using the notation in Theorem~\ref{Exterior-Estimate}, we define  
  $\lambda_i(t,R), \mu_j(t,R)$ as the  coefficients of the orthogonal projection onto the space $P(R)$  of $\vec{u}(t, r)=(u(t,r),u_t(t,r))$ with respect to  a suitable basis of $P(R)$ as in point $(3)$ of the remarks in the previous subsection. 
   \begin{equation}\label{u-projection}
\pipp\vec{u}(t, r) =\left(u(t, r) - \sum_{i=1}^{\tilde{k} }\lambda_i(t,R) r^{2i-d}, \, \,  u_t(t, r)
-\sum_{j=1}^{k } \mu_j(t,R) r^{2j-d} \right)
\end{equation}
Here 
\EQ{ \label{k def}
k:=\left[\frac{d}{4}\right],  \quad \tilde{k}:=\left[\frac{d+2}{4}\right]
}
 and we will fix this  from now on. 
Using
\eqref{Lambda-formula}, ~\eqref{gdata-B}, and~\eqref{fdata-B}, we have the
explicit formulae
\begin{equation}\label{lambda-explicit}
\lambda_j(t,R) =\sum_{i=1}^{ \tdk} 
\frac{-R^{d+2-2i-2j}}{(d-2j) (d+2-2i-2j)}d_id_j
 \int_R^\infty
u_r(t, r) \,  r^{2i-2} \, dr,  \hspace{0.3cm}\forall 1\leq j\leq \tdk
\end{equation}
and
\begin{equation}\label{mu-explicit}
\mu_j(t,R)=
\sum_{i=1}^{k}\frac{R^{d-2i-2j}}{d-2i-2j}c_ic_j
\int_R^\infty u_t(t, r) \,  r^{2i-1} \, dr, \hspace{0.5cm}\forall 1\leq j\leq k
\end{equation}
with $c_i, d_i$   defined in
(\ref{ci-formula}) and (\ref{di-formula}). 
From \eqref{Lambda-formula}, \eqref{gdata-A}, and \eqref{fdata-A}, we also have the following identities 
\begin{equation}\label{ur-lambda}
\int_R^\infty u_r(t, r) r^{2i-2}dr = -\sum_{j=1}^{\tdk} \frac{R^{2i+2j-d-2} (d-2j)}{d+2-2i-2j}\lambda_j(t,R), \hspace{0.3cm}\forall 1\leq i\leq \tdk\end{equation}
\begin{equation}\label{ut-mu}\int_R^\infty u_t(t, r) r^{2i-1}dr = \sum_{j=1}^{ k} \frac{R^{2i+2j-d}}{d-2i-2j}\mu_j(t,R), \hspace{0.3cm}\forall 1\leq i\leq k\end{equation}
We can rewrite $\lambda_j$ using integration by parts to obtain the following formula, which will be useful in later sections. 
\begin{lemma}\label{lem:ld-explicit} 
\ant{
 \lambda_j(t,R) =\frac{d_j }{d-2j} \left(u(R)R^{d-2j}+\sum_{i=1}^{\tdk-1}\frac{(2i)d_{i+1} R^{d-2i-2j}}{(d-2i-2j)}\int_R^\infty u(t,r)r^{2i-1}dr\right)
}
\end{lemma}
 \begin{proof}
 By applying integration by parts to (\ref{lambda-explicit}), we get 
\EQ{ \label{laj} 
\lambda_j(t,R) =&\frac{d_j}{d-2j}\sum_{i=1}^{ \tdk}  
\frac{u(R)R^{d-2j}d_i}{(d+2-2i-2j)} \\
&+ \frac{d_j}{d-2j}\sum_{i=2}^{ \tdk}  \frac{(2i-2)d_iR^{d+2-2i-2j}}{(d+2-2i-2j)}\int_R^\infty u(t,r)r^{2i-3}dr
}
Suming up the first term in the expression above using (\ref{di-identity-1}) yields 
\[u(R)R^{d-2j} \sum_{i=1}^{ \tdk} 
\frac{d_i}{(d+2-2i-2j)} = u(R)R^{d-2j} \]
The lemma follows by relabeling   $i=i'+1$ in the second term on the right-hand-side of~\eqref{laj}.
 \end{proof}

\begin{lemma}\label{ldmu-identity}  For any function $\vec{u}\in \Hd$, 
let $\lambda_j(t,R)$ and $
\mu_j(t,R)$ be the projection coefficients  defined as in \textnormal{(\ref{u-projection})}.  Then  the following inequalities hold:    
\begin{equation}\label{pip-control}
\|\pip \vec{u}\|^2_{\hrr}\simeq \sum_{i=1}^{\tdk}\left(\lambda_i(t,R)
R^{2i-\frac{d+2}{2}}\right)^2 +\sum_{i=1}^k \left(\mu_i(t,R) R^{2i-\frac{d}{2}}\right)^2
\end{equation}
\begin{equation} \label{pipp-control}
  \|\pipp
\vec{u}\|^2_{\hrr}\simeq\int_R^\infty   \left( \sum_{i=1}^{\tdk}  \left(\partial_r\lambda_i(t,r)r^{2i-\frac{d+1}{2}} \right)^2  +
\sum_{i=1}^k\left(  {\partial_r\mu_i(t,r)
r^{2i-\frac{d-1}{2}}} \right)^2 \right) \,dr \end{equation}
Here $X\simeq Y $ means $c_1 Y\leq X \leq c_2 Y$, for some positive constants $c_1, c_2$, which depend only on $d$ -- in particular, the constants are  independent of $t$ and $R$. 
\end{lemma}
\begin{proof}

From \eqref{proj0} as well as  \eqref{gdata-A} and~\eqref{fdata-A} we
see that 
\begin{equation}\begin{aligned}\|\pip \vec{u}\|^2_{\hrr}=
& \sum_{i,j=1}^{\tilde{k}}\lambda_i(t,R)\lambda_j(t,R)\langle r^{2i-d},
r^{2j-d}\rangle_{\dot{H}^1(r\geq R, r^{d-1}dr)} \\
&+\sum_{i,j=1}^k \mu_i(t,R)\mu_j(t,R)\langle r^{2i-d},
r^{2j-d}\rangle_{L^2(r\geq R,
r^{d-1}dr)}\end{aligned}\end{equation} 
Using the notation from point $(3)$ in Section~\ref{sec:ext}, we know that the  Gram matrices $A(R)$ and $\tilde{A}(R)$ are positive definite with 
\begin{equation}\label{AR-matrix}A(R)=\bmat{\langle r^{2i-d}, r^{2j-d}\rangle}_{L^2(r\geq R,
r^{d-1}dr)} = \bmat{\frac{R^{2i+2j-d}}{d-2i-2j}}_{k\times k}
\end{equation}
\begin{equation}\label{TAR-matrix}\tilde{A}(R)=\bmat{\langle r^{2i-d},
r^{2j-d}\rangle}_{\dot{H}^1(r\geq R, r^{d-1}dr)}
 =\bmat{\frac{R^{2i+2j-d-2}(d-2i)(d-2j)}{d+2-2i-2j}}_{\tilde{k}\times \tilde{k}}
\end{equation} 
We have the equality 
  \begin{align*}\|\pip \vec{u}\|^2_{\hrr} =&
 \sum_{i,j=1}^{\tilde{k}}(\lambda_i(t,R)R^{2i-\frac{d+2}{2}}) (\lambda_j(t,R) R^{2j-\frac{d+2}{2}})\frac{ (d-2i)(d-2j)}{d+2-2i-2j} \\
&+\sum_{i,j=1}^k (\mu_i(t,R)R^{2i-\frac{d}{2}})(\mu_j(t,R)R^{2j-\frac{d}{2}}) \frac{1}{d-2i-2j}\end{align*}
which implies~(\ref{pip-control}) since $A(1) = \bmat{\frac{1}{d-2i-2j}}_{k\times k}$ and $ \tilde{A}(1) = \bmat{\frac{(d-2i)(d-2j)}{d+2-2i-2j}}_{\tilde{k}\times \tilde{k}}$ are positive definite.

 To prove (\ref{pipp-control}),  we note that  we have the explicit formula
\begin{equation*}\begin{aligned}\|\pipp \vec{u}\|^2_{\hrr}=& \|\vec{u}\|_{\hrr}^2-\|\pip
\vec{u}\|^2_{\hrr}\\
 =& \int_R^\infty (u_r)^2r^{d-1}dr -
\sum_{i,j=1}^{\tilde{k}}\lambda_i(t,R)\lambda_j(t,R)
\frac{R^{2i+2j-d-2}(d-2i)(d-2j)}{d+2-2i-2j}
 \\
 &+\int_R^\infty (u_t)^2r^{d-1}dr  - \sum_{i,j=1}^k
\mu_i(t,R)\mu_j(t,R)\frac{R^{2i+2j-d}}{d-2i-2j}
\end{aligned}\end{equation*}
Differentiating  the above with respect to $R$ gives 
\EQ{
\label{Channel-derivative}
 \partial_R\|\pipp
\vec{u}\|^2_{\hrr} &=
  -(\partial_Ru(t, R))^2R^{d-1} \\
 & - 2
\sum_{i,j=1}^{\tilde{k}}\partial_R\lambda_i(t,R)\lambda_j(t,R)
\frac{R^{2i+2j-d-2}(d-2i)(d-2j)}{d+2-2i-2j}  \\& +
\sum_{i,j=1}^{\tilde{k}}\lambda_i(t,R)\lambda_j(t,R)  {R^{2i+2j-d-3}
(d-2i)(d-2j)}
 \\
 & - u_t^2(t, R)R^{d-1}  - 2\sum_{i,j=1}^k
\partial_R\mu_i(t,R)\mu_j(t,R)\frac{R^{2i+2j-d}}{d-2i-2j} \\&+ \sum_{i,j=1}^k \mu_i(t,R)\mu_j(t,R)
{R^{2i+2j-d-1}}
}
We seek to replace $\p_R u(R)$ and $\p_t u(R)$ in the above expression with expressions involving $\la_j(t, R)$ and $\mu_j(t, R)$. With this in mind we differentiate  the relation (\ref{ur-lambda}) with respect to $R$, which gives 
\[ R^{2i-2}\partial_R u(t, R)  = \sum_{j=1}^{\tdk}\partial_R\lambda_j(t,R) \frac{R^{2i+2j-d-2} (d-2j)}{d+2-2i-2j} - \sum_{j=1}^{\tdk}\lambda_j(t,R) R^{2i+2j-d-3} (d-2j) \]
Dividing through by $R^{2i -2}$ yields the following expression for $ \p_R u(R)$, 
\begin{equation}\label{f-lambda-formula}\partial_R u(t, R)   =
\sum_{j=1}^{\tdk}\partial_R\lambda_j(t,R) \frac{R^{ 2j-d} (d-2j)}{d+2-2i-2j} -
\sum_{j=1}^{\tdk}\lambda_j(t,R) R^{2j-d-1} (d-2j) 
\end{equation} 
Note that this identity holds for all $1\leq i\leq \tdk$, which means that for each $1 \le i, m \le  \ti k$ we have 
\begin{equation}\label{f-identity}\sum_{j=1}^{\tdk}\partial_R\lambda_j(t,R) \frac{R^{ 2j-d}
(d-2j)}{d+2-2i-2j} =\sum_{j=1}^{\tdk}\partial_R\lambda_j(t,R) \frac{R^{ 2j-d}
(d-2j)}{d+2-2m-2j}\hspace{0.2cm} 
\end{equation}
Similarly we can differentiate (\ref{ut-mu}) to obtain 
\begin{equation}\label{g-mu-formula}-u_t(t, R)   = \sum_{j=1}^{
k}\partial_R\mu_j(t,R) \frac{R^{2j-d+1}}{d-2i-2j} - \sum_{j=1}^{ k}\mu_j(t,R)
R^{2j-d},  \, \, \forall 1\leq i\leq k\end{equation}
By the same logic as above we see that for all $1 \le i, m \le k$  we have the  identity
\begin{equation}\label{g-identity} \sum_{j=1}^{ k}\partial_R\mu_j(t,R)
\frac{R^{2j-d+1}}{d-2i-2j} = \sum_{j=1}^{ k}\partial_R\mu_j(t,R)
\frac{R^{2j-d+1}}{d-2m-2j}\hspace{0.2cm} 
\end{equation}
Now we can plug the formulae (\ref{f-lambda-formula}) and  (\ref{g-mu-formula})   into
(\ref{Channel-derivative}) (note that by (\ref{f-identity}) and (\ref{g-identity}) we are free to choose  $i=1$ in both expressions) to
obtain 
\begin{equation}\begin{aligned}
\label{derivative-identity-0} & \partial_R\|\pipp
\vec{u}\|^2_{\hrr} = \\
 =& - \sum_{i,j=1}^{\tdk}\partial_R\lambda_i(t,R)\partial_R\lambda_j(t,R)
R^{2i-d}R^{2j-d}R^{d-1} \\& - \sum_{i,j=1}^k\frac{\partial_R\mu_i(t,R)\partial_R\mu_j(t,R)
R^{2i+2j-d+1}}{(d-2i-2)(d-2j-2)} \\
 =& - (\sum_{i=1}^{\tdk} \partial_R\lambda_i(t,R)R^{2i-d} )^2 R^{d-1} -
\left(\sum_{i=1}^k \frac{\partial_R\mu_i(t,R)
R^{2i-d+1}}{d-2-2i}\right)^2R^{d-1} \end{aligned}\end{equation} 
Upon further investigation of the relations \eqref{f-identity} and \eqref{g-identity}, and plugging these into the last line above we obtain a system of identities indexed by $m_1\in \{1, \dots, \ti k\}$ and  $m_2\in \{1, \dots,  k\}$.
\begin{equation}\begin{aligned}
\label{derivative-identity} - \partial_R\|\pipp
\vec{u}\|^2_{\hrr} &= 
  \left(\sum_{i=1}^{\tdk} \partial_R\lambda_i(t,R)R^{2i-\frac{d+1}{2}}\frac{d-2i}{d+2-2m_1-2i} \right)^2   \\
 & \quad +
\left(\sum_{i=1}^k \frac{\partial_R\mu_i(t,R)
R^{2i-\frac{d-1}{2}}}{d-2m_2-2i}\right)^2 \end{aligned}\end{equation}
We claim that \eqref{derivative-identity-0} and~\eqref{derivative-identity} together imply that 
\begin{equation} \label{pipp-2}
  -\partial_R\|\pipp
\vec{u}\|^2_{\hrr}\simeq \sum_{i=1}^{\tdk}  \big(\partial_R\lambda_i(t,R)R^{2i-\frac{d+1}{2}} \big)^2  +
\sum_{i=1}^k \big(  {\partial_R\mu_i(t,R)
R^{2i-\frac{d-1}{2}}} \big)^2 \end{equation}
 with constants depending only on dimension $d$. Once we establish~\eqref{pipp-2} we simply integrate from $R$ to $\infty$ to obtain (\ref{pipp-control}). Hence it remains to prove~\eqref{pipp-2}, which, given \eqref{derivative-identity-0} and~\eqref{derivative-identity}, is an immediate consequence,  of the following lemma. 

 \begin{lemma}\label{algebra-fact}  Let $d\geq 7$ be an odd integer. Given numbers $a_i\in\R, 1\leq i\leq \tdk=[\frac{d+2}{4}]$ satisfying 
 \begin{equation}\label{a-identity}\sum_{j=1}^{\tdk}a_j \frac{ 
(d-2j)}{d+2-2i-2j} =\sum_{j=1}^{\tdk}a_j \frac{ (d-2j)}{d+2-2m-2j}\hspace{0.2cm} \forall 1\leq i,m\leq
\tdk\end{equation}
it follows that  
\begin{equation}\label{a-inequality}
\bigg(\sum_{j=1}^{\tdk}a_j   \bigg)^2 \simeq \sum_{j=1}^{\tdk}a_j^2 
\end{equation}
Similarly, if $b_i\in\R, 1\leq i\leq k=[\frac{d}{4}]$ satisfies  
\begin{equation}\label{b-identity} \sum_{j=1}^{ k}b_j
\frac{1}{d-2i-2j} = \sum_{j=1}^{ k}b_j\frac{1}{d-2m-2j}\hspace{0.2cm} \forall 1\leq i, m\leq k
\end{equation}
then we have 
\begin{equation}\label{b-inequality}
\bigg( \sum_{j=1}^{ k}b_j
\frac{1}{d-2-2j}  \bigg)^2  \simeq \sum_{j=1}^{k}b_j^2 \end{equation} with constants depending only on $d$.
 \end{lemma}
 \begin{remark} When $d=3, 5$, (\ref{derivative-identity}) implies (\ref{pipp-control}) because $k,\tdk  \in \{0, 1\}$ 
 \end{remark}
   \begin{proof}
Setting $ i =1$ and letting   $m$ run through $2,\cdots \tdk$ in  (\ref{a-identity}), we obtain  a system of equations 
\begin{equation}\label{a-system-0}\sum_{j =1}^{\tdk}  \frac{a_j}{d+2-2m-2j}=0, \hspace{0.5cm}2\leq m\leq \tdk\end{equation}
We can rewrite the above  as \begin{equation}\label{a-system}
\left( \begin{array}{c}
\frac{1}{d-4}\\
 \frac{1}{d-6}\\
 \ldots\\
 \frac{1}{d-2\tdk}
\end{array} \right)a_1 + M
\left( \begin{array}{c}
a_2 \\
a_3 \\ \ldots\\a_{\tdk}
\end{array} \right)=0\end{equation}
where \[M= \bmat{\frac{1}{d+2-2m-2j}}_{2\leq m, j\leq \tdk}=\bmat{\frac{1}{d'-2m'-2j'}}_{1\leq m', j'\leq \tdk-1}\]
with $d'=d-2, m'=m-1,  j'=j-1$. 
This is precisely  the matrix (\ref{gdata-A}) with $R=1$ for  dimension $d'=d-2$ and  $k'=[\frac{d'}{4}]=[\frac{d+2}{4}]-1=\tdk-1$.

Since $M $ is positive definite and invertible, it follows that 
\EQ{\label{a1aj}
 a_1^2 \simeq  \sum_1^{\tdk} a_j^2
 }
To prove (\ref{a-inequality}), it will then suffice  to show that $(\sum_1^{\tdk}a_j)^2$ is a nonzero constant multiple of $a_1^2$.
Indeed, from    (\ref{gdata-A}) and (\ref{gdata-B}), the inverse of $M$ is given by  
\[M^{-1}=  \bmat{ \frac{1}{d'-2i-2j}
\frac{\prod_{1\leq l\leq k' } (d'-2j-2l)}{\prod_{1\leq l\leq k', l\not =j} (2l-2j)}\frac{\prod_{1\leq l\leq k' } (d'-2i-2l)}{\prod_{1\leq l\leq k', l\not =i} (2l-2i)}}_{1\leq i,j\leq k'} \]
which implies using~\eqref{a-system} that 
\begin{equation*}\begin{aligned}a_{i+1}=&-a_1 \sum_{j=1}^{k'} \frac{1}{(d'-2i)(d'-2i-2j)}\\&\times
\frac{\prod_{1\leq l\leq k' } (d'-2j-2l)}{\prod_{1\leq l\leq k', l\not =j} (2l-2j)}\frac{\prod_{1\leq l\leq k' } (d'-2i-2l)}{\prod_{1\leq l\leq k', l\not =i} (2l-2i)}\end{aligned}\end{equation*}
for $1\leq i\leq k'=\tdk-1$.
So we can check that 
\begin{equation} \label{sum-a}
\sum_1^{\tdk} a_j = a_1 \left(1 - \sum_{i,j=1}^{k'} \frac{c_i'c_j'}{(d'-2i)(d'-2i-2j)}\right)  \end{equation}
where here, from  (\ref{ci-formula}) we have 
 \[c_{i}'=\frac{\prod_{1\leq l\leq k' } (d'-2i-2l)}{\prod_{1\leq l\leq k', l\not =i} (2l-2i)}\]
In (\ref{sum-a}) we first sum in $j$ using (\ref{ci-identity-1}) and then sum in $i$ using (\ref{ci-identity-2}). We are left with 
\[\sum_{j=1}^{\tdk} a_j =a_1 \prod_{l=1}^{k'}\frac{ 2l} {d'-2l}\]
which proves that $(\sum_{j=1}^{\tdk} a_j)^2  = C(d) a_1^2$. In light of~\eqref{a1aj} we have finished the proof of~\eqref{a-inequality}.


We argue similarly to prove~\eqref{b-inequality}. First we write down a system of equations for $b_j$ by setting  $i=1$, and letting $m$ run through $2,\dots, k$ in~\eqref{b-identity}. 
\[\sum_{1}^k \frac{b_j}{(d-2-2j)(d-2m-2j)}=0, \hspace{0.5cm} 2\leq m\leq k\]
 Define $\bar{b}_j=\frac{b_j}{d-2-2j}, \bar{d}=d-2$, and $ \bar{k}=[\frac{\bar{d}+2}{4}]=[\frac{d}{4}]=k$. We then have  
\[\sum_{1}^k \frac{\bar{b}_j}{ (\bar{d}+2-2m-2j)}=0, \hspace{0.5cm} 2\leq m\leq k\]
which is exactly the same as (\ref{a-system-0}). Hence the same proof gives us 
\[(\sum_{j=1}^{k}\bar{b}_j)^2\simeq\sum_{j=1}^{k}\bar{b}^2_j\] Transfering the above formula back to the $b_j$ notation yields 
\[(\sum_{j=1}^{k}\frac{b_j}{d-2-2j})^2\simeq \sum_{j=1}^{k}(\frac{b_j}{d-2-2j})^2\simeq \sum_{j=1}^{k}b_j^2\]
as desired 
\end{proof}
This completes the proof of Lemma~\ref{ldmu-identity}. 
 \end{proof}

 \section{Rigidity Argument} \label{sec:rig} 
 The remainder of this paper will be devoted to showing that the critical element $\vec u_*$ from Proposition~\ref{critical-element} does not exist. In particular, we will prove the following rigidity theorem. 
 
\begin{thm}[Rigidity Theorem]\label{Rigidity} Let $\vec{u}(t)\in
\Hd=\dot{H}^1\times L^2(\R^{d }_*)$ be a  global solution to
\textnormal{(\ref{u eq})} such that the trajectory
\[\mathcal{K}:=\{\vec{u}(t) | t\in \R\}\] is pre-compact in
$\Hd $.  Then $u\equiv 0$.
\end{thm}
Note that the pre-compactness of the trajectory $\K$  implies that the $\HH$-norm of $\vec u(t)$ decays on the exterior cone $\{r \ge R+ \abs{t}\}$ as $t \to \pm \infty$.  
\begin{cor} \label{Exterior-Decay}Given $\vec u(t)$ as in Theorem~\ref{Rigidity} and any $R\geq 1$,  we have
\begin{equation} \label{extvan}
\|\vec{u}(t)\|_{\Hd(r\geq R+|t|)}\rightarrow 0 \text{ as }
t\rightarrow \pm \infty.
\end{equation}
\end{cor}

Following the outline of the argument in~\cite[Section $5$]{KLS}, we divide 
 the proof of Theorem~\ref{Rigidity}  into several steps. The main idea is to combine Theorem~\ref{Exterior-Estimate} with Corollary~\ref{Exterior-Decay} to establish the precise spacial asymptotic behavior of $\vec u(t)$. In particular we will show that $\vec u(0)$ has same the spacial decay as $\frac{1}{r}(Q(r) - n\pi)$ where $Q$ is a solution to the elliptic equation~\eqref{HarmonicMap} as in Lemma~\ref{lem:Q-asymptotic} -- note that this is better decay than what is expected for generic energy data. Indeed, we prove that 
\EQ{
\label{initial-a} u_0(r) &=  \vartheta r^{2-d}+O(r^{3-2d}) \text{ as } r\rightarrow \infty
\\
\int_r^\infty u_1(s)s^{2i-1} ds&=O(r^{2i+2-2d}) \text{ as } r\rightarrow \infty,  \quad  \forall 1\leq i\leq k
}
where $\vartheta$ is some constant, and $k := [\frac{d}{4}] $. 

We then argue by contradiction to show that $\vec u(t) = (0, 0)$ is the only solution with pre-compact trajectory $\K$ as in Theorem~\ref{Rigidity} and data that decay like~\eqref{initial-a}.

  \subsection{Rigidity argument. Step I:  estimating $\pi^{\perp}_R \vec u$ in $\HH(r \ge R)$}

 The goal of this section is to prove the following consequence of Theorem~\ref{Exterior-Estimate} and Corollary~\ref{Exterior-Decay}. 
 
 \begin{prop}~\label{Decay-Thm} Given a  radial global solution $\vec u(t) $ to \textnormal{(\ref{u eq})} with a pre-compact trajectory, 
 there exist  a number $R_0>1$ such that for every $R>R_0$,  we have the following estimate uniformly in time $t\in\R$. 
\begin{equation}\label{decay-estimate}\begin{aligned}\|\pi_R^\perp \vec{u}(t)\|_{\mathcal{H}(r\geq
R)}\lesssim&  R^{1-d}\|\pip \vec{u}(t)\|_{\hrr}\\ &
+R^{-\frac{d}{2}}\|\pip \vec{u}(t)\|_{\hrr}^2 + R^{-1}\|\pip
\vec{u}(t)\|_{\hrr}^3
\end{aligned}\end{equation}  Here $\hrr:=\dot{H}^1\times L^2 (\R^d\backslash B(0,R))$, $\pi_R$  and
$\pi_R^\perp$  are defined as in
Theorem~\ref{Exterior-Estimate}.   
\end{prop}

 In order to prove Proposition~\ref{Decay-Thm}  we require  a preliminary result concerning the nonlinear evolution of a modified Cauchy problem that  is adapted to capture the dynamics  of our pre-compact  solution $\vec u(t)$ restricted to the exterior cone $\Omega_R:= \{(t, r) \mid r \ge R+ \abs{t}\}$. Since we will only deal with  the evolution  -- and in particular we plan to make use of the vanishing property \eqref{extvan} -- on $\Omega_R$, we can  alter the nonlinearity and the potential term in \eqref{u eq} on the interior cone $\{1 \le r \le R + \abs{t}\}$ without affecting the flow on the exterior cone -- this is a consequence of the finite speed of propagation.  In particular, we can make the potential and the nonlinearity small on the interior of the cone so that for small initial data we can treat the  potential and nonlinearity as small perturbations.  This idea originates in~\cite{DKM4} and is used for the first time in~\cite{DKM5} to prove rigidity theorems.


To be precise we study the modified exterior wave equation
\EQ{
&h_{tt} - h_{rr} - \frac{2\ell+2}{r}h_r   = \ti{\N}_R(t,r,h), \hspace{1cm} r\geq 1\\
&h(t,1)=0, \quad \forall t\in \R\\
&\vec{h}(0)=(h_0, h_1)\label{Modified}
}
where $d  = 2 \ell +3$ and 
where we set 
\begin{equation*}
\begin{aligned}
V_R(r) &: = \left\{\begin{aligned}V(|t|+R),& ~~~~~  r\leq |t|+R\\
V(r), & ~~~~~r\geq |t|+R \end{aligned}\right.\\
F_R(r,h) &:= \left\{\begin{aligned}F(|t|+R, h)
,& ~~~~~  r\leq |t|+R\\
F(r,h)
, &~~~~~ r\geq |t|+R \end{aligned}\right.\\
G_R(r,h)&:=\left\{\begin{aligned}G(|t|+R, h),&~~~~~   r\leq |t|+R\\
G(r,h)
, &~~~~~  r\geq |t|+R \end{aligned}\right.\\
\tilde{\N}_R(t,r, h) &:= -V_R(r)h + F_R(r,h)  +G_R(r,h)
 \end{aligned}
\end{equation*}

Again, we remarks the the point of this modification is that everything coincides with equation (\ref{u eq}) in the exterior region
$\Omega_R=\{(t,r)| r \geq |t|+R\}$. Hence by finite speed of propagation, the solutions to both problems  will agree on
$\Omega_R$.

We note that the asymptotic behavior for $Q$ from 
(\ref{Q-asymptotic}) together with the explicit formulas for $V(r), F(r,h), G(r,h)$ from (\ref{VFG}), yield the estimates 
\begin{equation}\label{VFGR-bound}
\begin{aligned}
|V_R(r)|& \lesssim \left\{\begin{aligned} (|t|+R)^{-2\ell-4},& ~~~~~  \text{ for }r\leq |t|+R\\
r^{-2\ell-4}, & ~~~~~\text{ for }r\geq |t|+R
\end{aligned}\right.\\
|F_R(r,h)| &\lesssim \left\{\begin{aligned} (|t|+R)^{-3}h^2
,& ~~~~~  \text{ for }r\leq |t|+R\\
r^{-3}h^2
, &~~~~~ \text{ for }r\geq |t|+R \end{aligned}\right.\\
|G_R(r,h)|&\lesssim\left\{\begin{aligned}(|t|+R)^{2\ell-2}h^3,&~~~~~ \text{ for }  r\leq |t|+R\\
r^{2\ell-2}h^3
, &~~~~~ \text{ for } r\geq |t|+R \end{aligned}\right.\\
\end{aligned}
\end{equation}

The advantage of the modified equation (\ref{Modified}) is that by
choosing $R$ large, we can make the potential term $V_R(r)$
and nonlinearities $F_R(r, h), G_R(r, h )$ small in space and time and thus treat 
them as small perturbations.
\begin{lemma}\label{linear-decay-bound}
There exists a $R_*>0$ large  and $\delta_*>0$ small enough, such
that for all $R>R_*$ and initial data $\vec h(0) = (h_0, h_1)$  with
\begin{equation*}
\|\vec{h}(0)\|_{\mathcal{H}}\leq \delta_*\end{equation*}
there exists a
unique global solution to the modified  equation \textnormal{(\ref{Modified})}.
In addition, $\vec{h}(t)$ satisfies
\begin{equation*}
\|h\|_{L^{\frac{2(d+1)}{d-2}}_{t, x}( I \times ( \R^d_*))}\lesssim  \|\vec{h}(0)\|_{\mathcal{H}}\lesssim
\delta_*\end{equation*}
 with $d=2 \ell +3$.

Also, if we let $h_L(t):=S_0(t)\vec{h}(0)$ denote solution to free
exterior wave equation with initial data
$\vec{h}(0)$, we have
 \begin{equation}\label{decay2}\sup_{t\in
\mathbb{R}}\|\vec{h}-\vec{h}_L\|_{\mathcal{H}}\lesssim R^{1-d}
\|\vec{h}(0)\|_{\mathcal{H}} +
R^{-\frac{d}{2}}\|\vec{h}(0)\|_{\mathcal{H}}^2 +
R^{-1}\|\vec{h}(0)\|_{\mathcal{H}}^3.
\end{equation}
\end{lemma}

\begin{proof} 
We sketch how to deduce~\eqref{decay2}. We simplify the
notation  by writing
\[\|u\|_{S}:=\|u\|_{L^{\frac{2(d+1)}{d-2}}_{t, x}( I \times ( \R^d_*))}
\hspace{0.5cm}\|u\|_{W}: =\|u\|_{L^\infty_t\dot{H}_0^1(\R\times\R^d_*)}\]
By Duhamel's formula and Strichartz estimates from Theorem~\ref{Strichartz} we have 
\EQ{ \label{hhL}
\|h-h_L\|_{L^\infty_t \mathcal{H}}\lesssim \|V_Rh + F_R\|_{L^1_t L^2_x(\R \times \R^d_*)} + \| G_R\|_{\mathcal{Y}(\R \times \R^d_*)}
}
where $\mathcal{Y}(\R)$ is any sum of energy-admissible dual Strichartz spaces.   Arguing as in Remark~\ref{rem:st} (see e.g., ~\eqref{grupw} and~\cite{BCLPZ}) one controls the term involving  $ G_R$ by  
$
R^{-1}\| h\|_{\mathcal{X}(\R)}^3
$
where $\mathcal{X}(\R)$ is the intersection of all energy-admissible Strichartz spaces. 
Using (\ref{VFGR-bound}) we see that
\begin{equation*}
\|V_R(r)h \|_{L^1_tL^2_x}\lesssim\|V_R\|_{L^\frac{2(d+1)}{d+4}_tL^{\frac{2(d+1)}{3}}_x}\|h\|_{L^{\frac{2(d+1)}{d-2}}_tL_x^{\frac{2(d+1)}{d-2}}}
\lesssim R^{1-d}\|h\|_{S}\end{equation*} 
and similarly using~\eqref{VFGR-bound} and the Strauss estimate we have 
\ant{
\| F_R\|_{L^1_t L^2_x} \lesssim R^{-\frac{d}{2}} \| h\|_W \|h \|_S.
}
By the standard contraction mapping and continuity arguments we can find $R_*>0$   and $\delta_*>0$, so that  
\[ \| h\|_{\mathcal{X}}, \|h\|_{S}, \|h\|_{W}\lesssim  \|\vec{h}(0)\|_{\mathcal{H}.}\]
By~\eqref{hhL} and the above (\ref{decay2}) follows as well.
\end{proof}
We can now prove Proposition~\ref{Decay-Thm}. 

\begin{proof}[Proof of Proposition~\ref{Decay-Thm}] The proof is similar to~\cite[Lemma 5.3]{KLS}. We first prove the case $t=0$.
For each $R> 1$, we truncate the initial data, defining 
$\vec{u}_R(0)=(u_{0,R}, u_{1,R})$ by 
\[ u_{0,R}(r):=\left\{\begin{aligned}  u_0(r),& ~~~~~  \text{ for }r\ge  R\\
u_0(R)\frac{r-1}{R-1}, & ~~~~~\text{ for }r\le  R
\end{aligned}\right.\]
\[ u_{1,R}(r):=\left\{\begin{aligned}  u_1(r),& ~~~~~  \text{ for }r\ge  R\\
0, & ~~~~~\text{ for }r\le  R
\end{aligned}\right.\]
Note that  
\begin{equation}\label{URU}\|\vec{u}_R(0)\|_{\Hd}\lesssim
\|\vec{u}(0)\|_{\Hd(r\geq R)}\end{equation}
and hence there exists
$R_0> R_*$, so that for all $R\geq R_0$
\begin{equation*}
\|\vec{u}_R(0)\|_{\Hd}\leq \delta_*\end{equation*}
Here $R_*$ and $\delta_*$ are the constants in
Lemma~\ref{linear-decay-bound}.

Denote by $\vec{u}_R(t)$ the solution to (\ref{Modified}) with
initial data $\vec{u}_R(0)$.  By finite speed of
propagation,\[\vec{u}_R(t,r)=\vec{u}(t,r), \hspace{0.5cm}\forall
t\in\R, \, \,  r\geq R+|t|\]
where $\vec u(t)$ is our solution to~\eqref{u eq} as in Proposition~\ref{Decay-Thm}. 

Now denote by $\vec{u}_{R,L}(t)=S_0(t)\vec{u}_R(0)$ the free evolution of this data. From
(\ref{decay2}) in Lemma~\ref{linear-decay-bound}, we obtain 
\[\sup_t\|\vec{u}_R-\vec{u}_{R,L}\|_{\mathcal{H}}\lesssim R^{1-d}\|\vec{u}_R(0)\|_{\HH} +R^{-\frac{d}{2}}\|\vec{u}_R(0)\|_\HH^2 +R^{-1}\|\vec{u}_R(0)\|_\HH^3\]
Hence we have
\begin{equation*}\begin{aligned}&\|\vec{u}(t)\|_{\mathcal{H}(r\geq
R+|t|)}= \|\vec{u}_{R}(t)\|_{\mathcal{H}(r\geq R+|t|)} \\
 \geq& \|\vec{u}_{R,L}(t)\|_{\mathcal{H}(r\geq R+|t|)} -
C_0(R^{1-d}\|\vec{u}_R(0)\|_\HH +R^{-\frac{d}{2}}\|\vec{u}_R(0)\|^2_\HH
+R^{-1}\|\vec{u}_R(0)\|^3_\HH)\end{aligned}\end{equation*} 
Now, the
left hand side vanishes as $\abs{t} \to \infty$  by Corollary~\ref{Exterior-Decay}. 
Using the exterior energy estimates from Theorem~\ref{Exterior-Estimate} (choosing either $t \to  \infty$ or $t \to - \infty$ according to~\eqref{extd}), we deduce that 
 \begin{equation}
\label{u0control}\|\pipp \vec{u}_R(0)\|_{\hrr} \lesssim
R^{1-d}\|\vec{u}_R(0)\|_\HH +R^{-\frac{d}{2}}\|\vec{u}_R(0)\|^2_\HH
+R^{-1}\|\vec{u}_R(0)\|^3_\HH\end{equation} 
From~(\ref{URU}) 
and the fact that by definition 
\[\|\pipp \vec{u}_R(0)\|_{\hrr} =\|\pipp \vec{u}(0)\|_{\hrr} \] 
see that 
 \[
\label{u0control-1}\|\pipp \vec{u}(0)\|_{\hrr} \lesssim
R^{1-d}\|\vec{u}(0)\|_{\hrr} +R^{-\frac{d}{2}}\|\vec{u}(0)\|_{\hrr}^2
+R^{-1}\|\vec{u}(0)\|_{\hrr}^3\]
Finally, writing 
\[\|\vec{u}(0)\|_{\hrr}^2=\|\pip \vec{u}(0)\|_{\hrr}^2+\|\pipp \vec{u}(0)\|^2_{\hrr}\]
and by choosing  $R_0$ large enough, we can absorb all of the $\|\pipp
\vec{u}(0)\|_{\hrr}$ terms to the left hand side of (\ref{u0control}),
and obtain
\begin{equation*}
\|\pipp \vec{u}(0)|_{\hrr} \lesssim R^{1-d}\|\pip \vec{u}(0)\|
+R^{-\frac{d}{2}}\|\pip \vec{u}(0)\|^2 +R^{-1}\|\pip
\vec{u}(0)\|^3 \end{equation*}
This completes the proof for $t = 0$. For $t=t_0$ we can repeat the same argument by setting 
\[ \tilde{u}_{0,R}(r):=\left\{\begin{aligned}  u(t_0, r),& ~~~~~  \text{ for }r\ge  R\\
u(t_0, R)\frac{r-1}{R-1}, & ~~~~~\text{ for }r\le  R
\end{aligned}\right.\]
\[ \tilde{u}_{1,R}(r):=\left\{\begin{aligned}  u_1(t_0, r),& ~~~~~  \text{ for }r\ge  R\\
0, & ~~~~~\text{ for }r\le  R
\end{aligned}\right.\]
The key point here is  that by the pre-compactness of the
trajectory $\K$, we can find $R_0=R_0(\delta_*)$ such that $\forall
R\geq R_0$, we have
\[\|\vec{u}(t)\|_{\hrr}\leq \delta_*\]
uniformly in $t\in \R$.
\end{proof}

\subsection{Rigitity argument. Step II:  asymptotics for $\vec u(0, r)$}\label{Section:FA}

 In this step we use Proposition~\ref{Decay-Thm} to establish the asymptotic behavior of $\vec{u}(0,r)$ as $r \to \infty$ described in \eqref{initial-a}. To be precise we prove the following proposition. 
\begin{prop}\label{prop:as} Let $\vec u(t)$ be as in Theorem~\ref{Rigidity} with $\vec u(0) = (u_0, u_1)$. Then there exists $\vartheta_1  \in \R$ so that  
\begin{align*}
&r^{d-2} u_0(r)  \to \vartheta_1 \mas r \to \infty \\
&\int_r^\infty u_1(s)s^{2i-1} ds \to 0 \mas r \to \infty  \quad  \forall 1\leq i\leq k
\end{align*}
where  $k := [\frac{d}{4}] $. Moreover, we have the following estimates for the rates of convergence, 
\ant{
&\abs{r^{d-2} u_0(r) -   \vartheta_1}  = O(r^{-d+1}) \mas r \to \infty\\
&\abs{\int_r^\infty u_1(s)s^{2i-1} ds} =O(r^{2i+2-2d}) \mas r \to \infty\quad  \forall 1\leq i\leq k
}
\end{prop}

We remark that the analog of Proposition~\ref{prop:as} for $\ell =1$ and $d =5$ was proved in~\cite[Lemma~5.3]{KLS}. Here the proof of Proposition~\ref{prop:as} will consist of a rather lengthy argument which is complicated by the increasing dimension of the subspace $P(R)$ in Theorem~\ref{Exterior-Estimate}. To highlight the structure of the argument and illustrate the key differences from~\cite{KLS}, we will first treat the case $ \ell =2$ and $d=7$ where the dimension of $P(R)$ is $3$. We then divide the argument for general $\ell  >2$ into two cases determined by the parity of $\ell$ as there are subtle differences that arise when $\ell$ is odd as opposed to even.   

First, we establish  difference estimates for the coefficients $\la_j(t, R)$ and $\mu_j(t, R)$ of the projection of $ \vec u(t)$ onto $P(R)$, which  hold for all $\ell \ge 2$.

 Let $u(t)$ be as in Theorem~\ref{Rigidity} with the projection coefficients $\la_j(t, R)$ and $\mu_j(t, R)$ defined as in (\ref{u-projection}).   We recall the conclusions from 
 Lemma~\ref{ldmu-identity}, where $k=[\frac{d}{4}],$ and $ \ti k  = [\frac{d+2}{4}].$ 
\EQ{ \label{all-pip-identity}
&\|\pip \vec{u}\|^2_{\hrr}\simeq \sum_{i=1}^{\tdk}\left(\lambda_i(t,R)
R^{2i-\frac{d+2}{2}}\right)^2 +\sum_{i=1}^k \left(\mu_i(t,R) R^{2i-\frac{d}{2}}\right)^2
 \\
&  \|\pipp \vec{u}\|^2_{\hrr}\simeq \\  \quad &  \int_R^\infty  \left( \sum_{i=1}^{\tdk}  \left(\partial_r\lambda_i(t,r)r^{2i-\frac{d+1}{2}} \right)^2  +
\sum_{i=1}^k\left(  {\partial_r\mu_i(t,r)
r^{2i-\frac{d-1}{2}}} \right)^2\right)\,dr 
}
with uniform-in-time constants which  depend only on dimension $d=2\ell+3$. 
 Hence we can rewrite Proposition~\ref{Decay-Thm} as 
\begin{lemma}\label{all-decay} There exists $R_0>1$ so that for all $R>R_0$ we have 
\begin{equation*}
\begin{aligned} 
&\int_R^\infty  \left( \sum_{i=1}^{\tdk}  \left(\partial_r\lambda_i(t,r)r^{2i-\frac{d+1}{2}} \right)^2  +
\sum_{i=1}^k\left(  \partial_r\mu_i(t,r)
r^{2i-\frac{d-1}{2}} \right)^2 \right)\,dr 
\\
\lesssim & \sum_{i=1}^{\tdk}\left(R^{4i-3d}\lambda_i^2(t,R) + R^{8i-3d-4}\lambda_i^4(t,R) +R^{12i-3d-8}\lambda_i^6(t,R)\right)\\
&+ \sum_{i=1}^k \left(R^{4i+2-3d}\mu_i^2(t,R) + R^{8i-3d}\mu_i^4(t,R) +R^{12i-3d-2}\mu_i^6(t,R)\right)
 \end{aligned}\end{equation*}
\end{lemma}

Let  $ \delta_1>0$ be a small number to be determined precisely later such that $\delta_1<\delta_*$. By the pre-compactness of the trajectory $\K$, we can find $R_1=R_1(\delta_1)>R_0$ (Here  $\delta_*$ and $R_0$ are as in Lemma~\ref{linear-decay-bound} and Proposition~\ref{Decay-Thm}) such that 
\begin{equation}\begin{aligned}\label{small-initial-all}\|\vec{u}(t)\|_{\hrr}&\leq \delta_1  \hspace{0.5cm}\forall R\geq R_1, \forall t\in \R\\
R_1^{1- d} &\leq \delta_1.\end{aligned}\end{equation}
An immediate consequence of   (\ref{all-pip-identity}) and (\ref{small-initial-all}) is that the following estimates hold uniformly in time. 
\begin{equation}\label{small-delta-control}|\lambda_i(t,r)|
r^{2i-\frac{d+2}{2}},  |\mu_j(t,r)| r^{2j-\frac{d}{2}}\leq\delta_1\hspace{0.2cm} \forall r\geq R_1, 1\leq i\leq \tdk, 1\leq j\leq k\end{equation}
We prove the following difference estimates.  
\begin{lemma}\label{lem:all-difference}
Let $R_1$ be as in~\eqref{small-initial-all}.    For 
  all $r, r'$ such that $R_1\leq r\leq r'\leq 2r$, the following difference estimates hold  uniformly in $t\in\R$. 
\begin{equation}\begin{aligned}\label{difference-all-ld}
&|\lambda_j(t,r)-\lambda_j(t,r')|
\\
 \lesssim & r^{-2j-d} \sum_{i=1}^{\tdk}[r^{2i+1}|\lambda_i(t,r)| + r^{4i-1}|\lambda_i(t,r)|^2 +r^{6i-3}
 |\lambda_i(t,r)|^3]\\
+& r^{-2j-d}\sum_{i=1}^k [r^{2i+2}|\mu_i(t,r)| + r^{4i+1}|\mu_i(t,r)|^2 +r^{6i}|\mu_i(t,r)|^3]
\end{aligned}\end{equation}
and
\begin{equation}\begin{aligned}\label{difference-all-mu}
&|\mu_j(t,r)-\mu_j(t,r')| 
\\
 \lesssim &   r^{-2j-d-1}\sum_{i=1}^{\tdk}[r^{2i+1}|\lambda_i(t,r)| + r^{4i-1}|\lambda_i(t,r)|^2 +r^{6i-3}
 |\lambda_i(t,r)|^3]\\
+&  r^{-2j-d-1}\sum_{i=1}^k [r^{2i+2}|\mu_i(t,r)| + r^{4i+1}|\mu_i(t,r)|^2 +r^{6i}|\mu_i(t,r)|^3]
\end{aligned}\end{equation} 
\end{lemma}
\begin{remark}\label{domination} A quick observation is that for each index $i$ in the summations above,   
if  $|\lambda_i(t,r)|r^{2i-2}\gtrsim 1$, then the cubic term dominates the growth rate in the terms involving $\lambda_i(t,r)$.  When $|\lambda_i(t,r)|r^{2i-2}\lesssim 1$, the linear term dominates. 
Similarly, 
 if $|\mu_i(t,r)|r^{2i-1}\gtrsim 1$, the cubic term dominates the growth rate in the terms involving $\mu_i(t,r)$. When $|\mu_i(t,r)|r^{2i-1}\lesssim 1$, the linear term dominates. 
\end{remark}
\begin{proof} The lemma is a simple consequence of Lemma~\ref{all-decay}. 
\begin{align*}
&|\lambda_j(t,r)-\lambda_j(t,r')|^2 =  \abs{\int_r^{r'}\partial_s\lambda_j(t,s)ds}^2
\\
\lesssim &  \abs{\int_r^{r'}(\partial_s\lambda_j(t,s)s^{2j-\frac{d+1}{2}})^2ds}
 \abs{\int_r^{r'}( s^{-2j+\frac{d+1}{2}})^2ds} \\
 \lesssim & r^{-4j+d+2} \sum_{i=1}^{\tdk}[r^{4i-3d}\lambda_i^2(t,r) + r^{8i-3d-4}\lambda_i^4(t,r) +r^{12i-3d-8}\lambda_i^6(t,r)]\\
&+r^{-4j+d+2} \sum_{i=1}^k [r^{4i+2-3d}\mu_i^2(t,r) + r^{8i-3d}\mu_i^4(t,r) +r^{12i-3d-2}\mu_i^6(t,r)]
\end{align*}
and 
\begin{align*}
&|\abs{\mu_j(t,r)-\mu_j(t,r')}^2 =  \abs{\int_r^{r'}\partial_s\mu_j(t,s)ds}^2
\\
\lesssim &  \abs{\int_r^{r'}(\partial_s\mu_j(t,s)s^{2j-\frac{d-1}{2}})^2ds}
 \abs{\int_r^{r'}( s^{-2j+\frac{d-1}{2}})^2ds} \\
 \lesssim & r^{-4j+d} \sum_{i=1}^{\tdk}[r^{4i-3d}\lambda_i^2(t,r) + r^{8i-3d-4}\lambda_i^4(t,r) +r^{12i-3d-8}\lambda_i^6(t,r)]\\
&+r^{-4j+d} \sum_{i=1}^k [r^{4i+2-3d}\mu_i^2(t,r) + r^{8i-3d}\mu_i^4(t,r) +r^{12i-3d-2}\mu_i^6(t,r)]
\end{align*}
This completes the proofs of (\ref{difference-all-ld}) and (\ref{difference-all-mu}).
\end{proof}
The following corollary is an immediate consequence of~(\ref{small-delta-control}).
 \begin{cor}\label{cor:all-difference} Let $R_1$ be as above.    For 
  all $r, r'$ such that $R_1\leq r\leq r'\leq 2r$, the following difference estimates hold uniformly in $t\in\R$. 
  \begin{align}\label{all-difference-delta}
 |\lambda_j(t,r)-\lambda_j(t,r')|
\lesssim & \delta_1 [\sum_{i=1}^{\tdk} r^{2i-2j}|\lambda_i(t,r)| + \sum_{i=1}^k  r^{2i-2j+1} |\mu_i(t,r)|]\\
|\mu_j(t,r)-\mu_j(t,r')| 
 \lesssim & \frac{\delta_1}{r} [\sum_{i=1}^{\tdk} r^{2i-2j}|\lambda_i(t,r)| + \sum_{i=1}^k  r^{2i-2j+1} |\mu_i(t,r)|] \label{all-dd-mu}
  \end{align}
 \end{cor}

Next, we make a few observations relating $\vec u(t, r)$ and the projection coefficients $\la_j(t, r)$ and $\mu_j(t, r)$.  

\begin{lemma}\label{u-ld-mu}  For each fixed $R>1$ and for all $(t, r) \in \Omega_R = \{ r \ge R + \abs{t}\}$ we have 
\begin{align}\label{u-ld}
&u(t,r)=\sum_{j=1}^{\tdk}  \lambda_j(t,r)r^{2j-d}
\\
\label{u-mu}
&\int_r^\infty u_t(t,s) s^{2i-1}ds = \sum_{j=1}^{ k}\mu_j(t,r) \frac{r^{2i+2j-d}}{d-2i-2j},  \quad \forall 1\leq i \leq k  
\\
\label{mu-exp}
&\mu_j(t,r)=
\sum_{i=1}^{k}\frac{r^{d-2i-2j}}{d-2i-2j}c_ic_j
\int_r^\infty u_t(t,s) s^{2i-1}ds,  \quad  \forall 1\leq j\leq k
\\
\label{u-ldj} 
& \lambda_j(t,r) =\frac{d_j }{d-2j} \left(u(t,r)r^{d-2j}+\sum_{i=1}^{\tdk-1}\frac{(2i)d_{i+1}r^{d-2i-2j}}{(d-2i-2j)}\int_r^\infty u(t,s)s^{2i-1}ds \right),  
\end{align}
with the last line holding for all $ j\leq \tdk$ and where $d_j$ is defined in~\eqref{di-formula}. 
 Moreover, for any $1\leq j, j'\leq \tdk$, we have 
 \begin{multline}\label{ldjldj}
  |\lambda_j(t_1,r)-\lambda_j(t_2,r)| \\ \lesssim  r^{2j'-2j}\abs{\lambda_{j'}(t_1,r)-\lambda_{j'}(t_2,r)} +  \sum_{m=1}^k\abs{    r^{2m-2j}\int_{t_2}^{t_1}\mu_m(t,r)dt} 
 \end{multline}
\end{lemma}

\begin{remark}
In~\cite{KLS}, where $\ell = 1$, $d =5$ and $k= \ti k =1$, the authors used the notation $\la_1(t, r) = v_0(t, r)$ and $\mu_1(t, r) = v_1(t, r)$; see~\cite[equation $5.20$]{KLS} which contains the analogs of~\eqref{u-ld} and~\eqref{u-mu} above. 
To avoid possible confusion, we remark that the expression~\eqref{u-ld} does not mean that we have proved that  $u(t, r_0)$ is an element of $P(r_0)$ -- in this case one would see $\la_j(t, r_0)r^{2j-d}$ on the right hand side of~\eqref{u-ld} rather than $\la_j(r_0)r_0^{2j-d}$. A similar remark holds for~\eqref{u-mu} as well. 
\end{remark} 

\begin{proof} 

 First, note that  (\ref{u-ld}) follows by setting $i=1$ in~(\ref{ur-lambda}).  Next, we observe that (\ref{u-mu})  and (\ref{mu-exp}) are just restatements of  (\ref{ut-mu}) and (\ref{mu-explicit}) and  that (\ref{u-ldj}) was proved in Lemma~\ref{lem:ld-explicit}.  So we are left to prove (\ref{ldjldj}).

Choosing times $t_1 \neq t_2$ and plugging  (\ref{u-mu}) into (\ref{u-ldj}) yields 
\begin{align*}& \frac{d-2j}{d_j}r^{2j-d}(\lambda_j(t_1,r)-\lambda_j(t_2,r))\\ = &(u(t_1,r)-u(t_2,r)) 
+ \sum_{i=1}^{\tdk-1}\frac{(2i)d_{i+1}r^{-2i}}{(d-2i-2j)}\int_r^\infty (u(t_1,s)-u(t_2,s))s^{2i-1}ds\\
= &(u(t_1,r)-u(t_2,r)) 
+ \sum_{i=1}^{\tdk-1}\frac{(2i)d_{i+1}r^{-2i}}{(d-2i-2j)}\int_{t_2}^{t_1}\int_r^\infty u_t(t,s) s^{2i-1}ds\\
= &(u(t_1,r)-u(t_2,r)) 
+ \sum_{i=1}^{\tdk-1}\sum_{m=1}^k \frac{(2i)d_{i+1} r^{2m-d}}{(d-2i-2j)(d-2i-2m)}\int_{t_2}^{t_1}\mu_m(t,r)dt
\end{align*}
Performing the same computation for $j'$  then gives 
\begin{align*}&\frac{d-2j}{d_j}r^{2j-d}(\lambda_j(t_1,r)-\lambda_j(t_2,r))-\sum_{i=1}^{\tdk-1}\sum_{m=1}^k \frac{(2i)d_{i+1}  r^{2m-d}\int_{t_2}^{t_1}\mu_m(t,r)dt}{(d-2i-2j)(d-2i-2m)} \\
=& \frac{d-2j'}{d_{j'}}r^{2j'-d}(\lambda_{j'}(t_1,r)-\lambda_{j'}(t_2,r))-\sum_{i=1}^{\tdk-1}\sum_{m=1}^k \frac{(2i)d_{i+1}  r^{2m-d}\int_{t_2}^{t_1}\mu_m(t,r)dt}{(d-2i-2j')(d-2i-2m)} 
\end{align*}
From the above we obtain  
\begin{align*}& (\lambda_j(t_1,r)-\lambda_j(t_2,r))\\ =&\frac{d_j(d-2j')}{d_{j'}(d-2j)}r^{2j'-2j}(\lambda_{j'}(t_1,r)-\lambda_{j'}(t_2,r)) \\ &+\frac{d_j }{(d-2j)} \sum_{i=1}^{\tdk-1}\sum_{m=1}^k \frac{2(j-j')(2i)d_{i+1} r^{2m-2j}}{(d-2i-2j)(d-2i-2j')(d-2i-2m)}\int_{t_2}^{t_1}\mu_m(t,r)dt\end{align*}
which implies (\ref{ldjldj}).
\end{proof}

Next,   we   record a computation regarding  an averaged difference of the  $\mu_j$'s at two distinct times $t_1 \neq t_2$.  
\begin{lemma}\label{mu-computation} For each $1\leq j\leq k$, 
, $t_1\not= t_2$, and $R$  large enough, we have 
\begin{equation*}
\frac{1}{R}\int_R^{2R}\mu_j(t_1,r)-\mu_j(t_2, r) dr =  \sum_{i=1}^{k}\frac{c_ic_{j}}{d-2-2j}\int_{t_1}^{t_2} \left(I(i,j)  + II(i,j) \right) \, dt
\end{equation*}
with $I(i,j), II(i,j)$ as follows
\begin{equation}\label{I-formula}\begin{aligned}
I(i,j)  = & \frac{1}{R}\int_R^{2R} r^{d-2i-2j}\int_r^\infty  [u_{ss}(t,s)+\frac{2\ell+2}{s}u_s(t,s)]s^{2i-1}\,ds\,dr\\
 = & -\frac{1}{R} (u(t,r)r^{d-2j-1})\Big|_{r=R}^{r=2R} + (2i-2j-1)\frac{1}{R}\int_{R}^{2R} u(t,r)r^{d-2j-2} dr\\
& - \frac{(2\ell-2i+3)(2i-2)}{R}\int_R^{2R}r^{d-2i-2j}\int_r^\infty u(t,s)s^{2i-3}ds\,dr
\end{aligned}\end{equation}
\begin{equation}\label{II-formula}
II(i,j)=  \frac{1}{R}\int_R^{2R}  r^{d-2i-2j}\int_r^\infty\!\!\!\! \int_{t_1}^{t_2}[-V(s)u(t,s)+ \N(s,u(t,s))]s^{2i-1} \,ds\,dr
\end{equation}  
\end{lemma}
 \begin{proof} Using the  explicit formula for $\mu_j$ in (\ref{mu-exp}), we have  \begin{align*}
&\frac{1}{R}\int_R^{2R}\mu_j(t_1,r)-\mu_j(t_2, r) dr\\
 =& \sum_{i=1}^{k}\frac{c_ic_{j}}{d-2i-2j}\frac{1}{R}\int_R^{2R} r^{d-2i-2j}\int_r^\infty (u_t(t_1,s)-u_t(t_2, s))s^{2i-1}ds dr  \\
 =& \sum_{i=1}^{k}\frac{c_ic_{j}}{d-2i-2p_0}\frac{1}{R}\int_R^{2R}r^{d-2i-2j}\int_r^\infty\!\!\!\! \int_{t_2}^{t_1}u_{tt}(t,s)s^{2i-1}\,dt\,ds\,dr  \\
 =&  \sum_{i=1}^{k}\frac{c_ic_{j}}{d-2i-2j}\int_{t_2}^{t_1}\frac{1}{R}\int_R^{2R} r^{d-2i-2j}\int_r^\infty  [u_{ss}(t,s)+\frac{2\ell+2}{s}u_s(t,s)]s^{2i-1}\,ds\,dr\,dt \\
  +&  \sum_{i=1}^{k}\frac{c_ic_{j}}{d-2i-2j}\int_{t_2}^{t_1}\frac{1}{R}\int_R^{2R} r^{d-2i-2j}\int_r^\infty  [ \N(s,u(t,s)) -V(s)u(t,s)]s^{2i-1}\,dsdrdt  \\
 = &  \sum_{i=1}^{k}\frac{c_ic_{j}}{d-2i-2j}\int_{t_2}^{t_1} I(i,j)  + II(i,j)dt
\end{align*}
here we used the equation  (\ref{u eq}) for $u$.

Now we use integration by parts and notice $d=2\ell+3$, we get 
\begin{align*}
I(i,j)  = & \frac{1}{R}\int_R^{2R} r^{d-2i-2j}\int_r^\infty  [u_{ss}(t,s)+\frac{2\ell+2}{s}u_s(t,s)]s^{2i-1}\,ds\,dr\\
 = & -\frac{1}{R} (u(t,r)r^{d-2j-1})\Big|_{r=R}^{r=2R} + (2i-2j-1)\frac{1}{R}\int_{R}^{2R} u(t,r)r^{d-2j-2} dr\\
& - \frac{(2\ell-2i+3)(2i-2)}{R}\int_R^{2R}r^{d-2i-2j}\int_r^\infty u(t,s)s^{2i-3}ds\,dr
\end{align*}
as desired. 
  \end{proof}

Now we begin the process of proving Propostion~\ref{prop:as}  using Lemma~\ref{lem:all-difference} and Corollary~\ref{cor:all-difference} as the main tools. The goal is to understand the asymptotic behavior of the projection coefficients~$\la_j(t, r)$ and $\mu_j(t, r)$.  We proceed in iterative cycles. In each cycle, we first show a pair $\lambda_{i}(t,r),\mu_{i-1}(t,r)$ converges as $r\rightarrow \infty$. Then we then show that the limits must be identically $0$.  By feeding  this information back into (\ref{difference-all-ld}) and (\ref{difference-all-mu}) we enter the next cycle where the goal is to show that   $\lambda_{i-1}(t, r)$ and $\mu_{i-2}(t, r)$ converge to $0$ as $r \to \infty$.  Eventually we show all of the coefficients have limits as $r\rightarrow \infty$, and that  all of these  limits must be $0$ with the possible exception of $\lambda_1$.

We will distinguish between the cases when $\ell$ is even or $\ell$ is odd, as there is a slight difference in the computation.   To clarify the exposition we illustrate the method by working out the details of the simplest case not covered in~\cite{KLS}, namely  $\ell=2$. 
 
\subsubsection{\textbf{Proof of Proposition~\ref{prop:as} when  $\ell=2$:}}  \quad \\
When $\ell =2$ we have  $d= 7, k= 1, \tdk = 2,$ and we have projection coefficients $\la_1, \la_2$ and $\mu_1$ with 
\begin{equation*}
\pipp \vec{u}(t, r)= (u(t, r) -\lambda_1(t,R)  {r^{-5}} -\lambda_2(t,R)  {r^{-3}}, u_t-\mu(t,R) r^{-5})\end{equation*}

Recall that our goal is to prove the following result, which is just Proposition~\ref{prop:as} specialized to the case $\ell = 2$. 

\begin{prop}\label{prop:ell2} Let $\vec u(t)$ be as in Theorem~\ref{Rigidity} with $\vec u(0) = (u_0, u_1)$.  When $\ell=2$ we have 
\begin{equation*}
r^5u_0(r) =\vartheta  +O(r^{-6}) \text{ as } r\rightarrow \infty\end{equation*}
\begin{equation*}
\int_r^\infty u_1(s)s ds =  O(r^{-10}) \text{ as } r\rightarrow\infty\end{equation*}
\end{prop}

We record the conclusions of  Lemma~\ref{lem:all-difference} and Corollary~\ref{cor:all-difference} for $\ell = 2$.  For any $r, r'$ such that $R_1\leq r\leq r'\leq 2r$,   the following estimates hold uniformly in time. 
\begin{equation}\begin{aligned}\label{true-difference1}
 |\ld(t, r')-\ld(t, r)| \lesssim&  r^{-6}|\ld(t,r)| + r^{-6}|\ld(t,r)|^2 + r^{-6}|\ld(t,r)|^3
 \\
  &+ r^{-4}|\ldd(t,r)| + r^{-2}|\ldd(t,r)|^2 +|\ldd(t,r)|^3 \\ &+
  r^{-5}|\mu(t,r)| + r^{-4}|\mu(t,r)|^2 + r^{-3}|\mu(t,r)|^3\\
\end{aligned}\end{equation}
\begin{equation}\label{true-difference2}\begin{aligned}
 |\ldd(t,r')-\ldd(t,r)|
 \lesssim  &
r^{-8}|\ld(t,r)|+ r^{-8}|\ld(t,r)|^2 + r^{-8}|\ld(t,r)|^3
 \\
  &+ r^{-6}|\ldd(t,r)| + r^{-4}|\ldd(t,r)|^2 +r^{-2}|\ldd(t,r)|^3\\
 &+  r^{-7}|\mu(t,r)| + r^{-6}|\mu(t,r)|^2 + r^{-5}|\mu(t,r)|^3\end{aligned}\end{equation}
 \begin{equation}\label{true-difference3}\begin{aligned}
 |\mu(t, r')-\mu(t,r)|
 \lesssim &   r^{-7}|\ld(t,r)| + r^{-7}|\ld(t,r)|^2 + r^{-7}|\ld(t,r)|^3
  \\
  &+ r^{-5}|\ldd(t,r)| + r^{-3}|\ldd(t,r)|^2 +r^{-1}|\ldd(t,r)|^3\\
 &+r^{-6}|\mu(t,r)| + r^{-5}|\mu(t,r)|^2 + r^{-4}|\mu(t,r)|^3\\
\end{aligned}\end{equation}
\begin{equation}\label{difference-relation}\left\{\begin{aligned}
 |\ld(t,r')-\ld(t,r)|  \lesssim & 
 \delta_1 \big(  | \ld(t,r)| +r^2 |\ldd(t,r)| +r  |\mu(t,r)| \big)
 \\
 |\ldd(t,r')-\ldd(t,r)|\lesssim &   r^{-2}\delta_1 \big(  | \ld(t,r)| +r^2 |\ldd(t,r)| +r  |\mu(t,r)| \big)
  \\
 |\mu(t,r')-\mu(t,r)|\lesssim &  r^{-1}\delta_1 \big(  | \ld(t,r)| +r^2 |\ldd(t,r)| +r  |\mu(t,r)| \big)
\end{aligned}\right.\end{equation}
\begin{lemma}[$\epsilon$-growth]\label{one-step-control} Given any small fixed number  $\epsilon>0$, the following estimates hold uniformly in $t \in \R$ with constants $C= C( \e)$. 
\begin{equation}\label{epsilon-control}\ld(t,r)\lesssim r^{3\epsilon}, \hspace{0.5cm}\ldd(t,r)\lesssim r^{\epsilon},\hspace{0.5cm} \mu(t,r)\lesssim r^{\epsilon}\end{equation}
\end{lemma}
\begin{proof} Fix a small constant $\epsilon >0$.  From (\ref{difference-relation}) and the triangle inequality we obtain for any $r>R_1$, 
\begin{equation}\label{total-estimate}\begin{aligned}
 |\ld(t,2r)|  \leq  &
  (1+C\delta_1) | \ld(t,r)| +C\delta_1r^2 |\ldd(t,r)| +C\delta_1r  |\mu(t,r)|
 \\
 |\ldd(t,2r) |\leq  & Cr^{-2}\delta_1  | \ld(t,r)| + (1+C\delta_1) |\ldd(t,r)| + Cr^{-1}\delta_1   |\mu(t,r)|
  \\
 |\mu(t,2r) |\leq &  r^{-1}C \delta_1  | \ld(t,r)| +C \delta_1 r  |\ldd(t,r)| +(1+C\delta_1)  |\mu(t,r)|
\end{aligned}\end{equation}
Now fix  $r_0>R_1$, and define 
\[H_n := \frac{|\ld(t,2^nr_0)|}{(2^nr_0)^2} + |\ldd(t,2^nr_0)| + \frac{|\mu(t,2^nr_0)|}{2^nr_0}\]
From~\eqref{total-estimate} we can deduce that  
\[H_{n+1}\leq (1+3C\delta_1) H_n   \]
Now choose $\delta_1$ small enough so that $1+3C\delta_1< 2^\epsilon$.  Iterating the previous line gives
\begin{equation}\label{H-bound} H_n\leq (1+3C\delta_1)^n H_0\lesssim (2^nr_0)^\epsilon
\end{equation}
Notice that (\ref{small-delta-control}) ensures that once we fix $r_0$,   $H_0$ is  bounded for all $t\in \R$, and hence (\ref{H-bound}) holds uniformly in time.   This means that 
\begin{equation}\label{1step-control}\ld(t,2^nr_0)\lesssim (2^nr_0)^{2+\epsilon},\hspace{0.2cm}  \ldd(t,2^nr_0)\lesssim (2^nr_0)^\epsilon, \hspace{0.2cm}\mu(t,2^nr_0)\lesssim (2^nr_0)^{1+\epsilon}\end{equation}
Now we  feed (\ref{1step-control}) into the estimate for $\mu(t,r)$ in (\ref{true-difference3}). This yields 
\[|\mu(t,2^{n+1}r_0)|\lesssim  (1+C\delta_1)|\mu(t,2^nr_0)| + (2^nr_0)^{-1+\epsilon}\]
Iterating the above as before gives the improved bound 
\begin{equation*}
|\mu(t,2^nr_0)|\lesssim (2^nr_0)^{\epsilon}
\end{equation*}
Next,  we feed  (\ref{1step-control}) with the improvement above into the estimate for $\ld(t,r)$ in  (\ref{true-difference1}). We have  
\[|\ld(t,2^{n+1}r_0)|\lesssim (1+C\delta_1)|\ld(t,2^nr_0)| + (2^nr_0)^{3\epsilon}\]
which again yields the improvement 
\begin{equation*}
|\ld(t,2^nr_0)|\lesssim (2^nr_0)^{3\epsilon}
\end{equation*}
after iterating. 
In summary, we have proved  
\[\ld(t,2^nr_0)|\lesssim (2^nr_0)^{3\epsilon}, \hspace{0.3cm}\ldd(t,2^nr_0)\lesssim (2^nr_0)^{\epsilon},\hspace{0.3cm} |\mu(t,2^nr_0)|\lesssim (2^nr_0)^{\epsilon}\]
Finally,   (\ref{epsilon-control}) follows by combining the above with the  difference estimates (\ref{true-difference1}), (\ref{true-difference2}), and (\ref{true-difference3}).
\end{proof}
%
\begin{lemma}\label{lem:tr} There exist uniformly bounded functions $\vartheta(t)$ and $\varrho(t)$ so that 
\label{ldd-mu-limit} 
\begin{equation}\label{ldd-limit}|\ldd(t,r)-\vartheta_2(t)|=O(r^{-2}) \text{ as } r\rightarrow \infty\end{equation}
\begin{equation}\label{mu-limit}|\mu(t,r)-\varrho(t)| = O(r^{-1}) \text{ as } r\rightarrow \infty\end{equation}
where the implicit constants in  $O(\cdot)$ are also uniform in time. 
 \end{lemma}
\begin{proof}
We let   $\epsilon$ and $r_0$ be as in the proof of Lemma~\ref{one-step-control}, 
 and plug (\ref{epsilon-control}) into  (\ref{true-difference2}). This gives  
\[|\ldd(t,2^{n+1}r_0)-\ldd(t,2^nr_0)|\lesssim (2^nr_0)^{-8+9\epsilon}+(2^nr_0)^{-2+3\epsilon}+(2^nr_0)^{-5+3\epsilon}\]
This implies that 
\[\sum_n|\ldd(t,2^{n+1}r_0)-\ldd(t,2^nr_0)|<\infty\]
Therefore,  $\ldd(t,2^nr_0)$ has limit as $n \to \infty$, which we denote by  $\vartheta_2(t)$. Next, we have 
\begin{align*}|\vartheta(t)-\ldd(t, r_0)|&=\lim_{n\rightarrow \infty}|\ldd(t,2^{n+1}r_0)-\ldd(t,r_0)|\\
&\lesssim\lim_{m\rightarrow\infty} \sum_{l=1}^n|\ldd(t,2^{l+1}r_0)-\ldd(t,2^lr_0)| 
\lesssim r_0^{-2+3\epsilon}\sum_{l=1}^\infty (2^l)^{-2+3\epsilon} \end{align*}
Since (\ref{small-delta-control}) implies that  $\ldd(t, r_0)$ uniformly bounded,  the above means that  $\vartheta_2(t)$ and hence $\ldd(t,2^nr_0)$ are uniformly bounded. 
Using the fact that $\ldd(t, 2^nr_0)$ is bounded, we can upgrade  (\ref{true-difference2}) to 
\[|\ldd(t,2^{n+1}r_0)-\ldd(t,2^nr_0)|\lesssim (2^nr_0)^{-8+9\epsilon}+(2^nr_0)^{-2} +(2^nr_0)^{-5+3\epsilon}\]
and therefore  \[|\ldd(t,2^nr_0)-\vartheta_2(t)|\lesssim \sum_{l\geq n}|\ldd(t,2^{l+1}r_0)-\ldd(t,2^lr_0)|\lesssim (2^nr_0)^{-2}\]
The fact that $\abs{\la_2(t, r)  -  \vartheta(t)} = O(r^{-2})$ as $r \to \infty$  now follows from difference estimates (\ref{true-difference2}).

Similarly, we plug (\ref{epsilon-control}) and the fact that we now know that $\ldd(t,r)$ is  bounded  into (\ref{true-difference3}). This   yields  
\[|\mu(t,2^{n+1}r_0)-\mu(t,2^nr_0)|\lesssim (2^nr_0)^{-7+9\epsilon}+(2^nr_0)^{-1}+(2^nr_0)^{-4+3\epsilon}\]
Arguing as above we deduce that $\mu(t,2^nr_0)$ has limit $\varrho(t)$, which is bounded in $t \in \R$,  and 
\[|\mu(t,2^nr_0)-\varrho(t)|\lesssim (2^nr_0)^{-1}\]
Using the difference estimate (\ref{true-difference3}) as above we obtain  (\ref{mu-limit}).
\end{proof}
From Lemma~\ref{one-step-control} and Lemma~\ref{ldd-mu-limit} we deduce the following asymptotic behavior for $\vec{u}(t, r)$ as $r \to \infty$. 
\begin{lemma}\label{lem:1step-decay} The following holds uniformly in time. 
\begin{equation}\label{1step-decay}
 r^3 u(t,r)=  {\vartheta_2(t)} +O(r^{-2+3\epsilon}) \text{ as } r\rightarrow\infty\end{equation}
\end{lemma}
\begin{proof}  Using the formula  (\ref{u-ld}) 
from  Lemma~\ref{u-ld-mu},~\eqref{1step-decay} follows immediately from the $r^{3\epsilon}$-control 
of $\ld(t,R)$ from (\ref{epsilon-control}) together with the conclusion of 
Lemma~\ref{ldd-mu-limit}.
 \end{proof}
  We will improve the asymptotics of $\vec{u}$ by showing $\vartheta_2(t)=\varrho(t)=0$.
\begin{lemma}The limit $\vartheta_2(t)$ is independent of time and from now on we will write $\vartheta_2 = \vartheta_2(t)$. 
\end{lemma}
\begin{proof} 
Fix times  $t_1\not=t_2$.  Using (\ref{ldd-limit})    and (\ref{ldjldj}) with $j=2, j'=1$, we get 
\begin{align*}
&|\vartheta_2(t_1)-\vartheta_2(t_2)| \\  = &|\ldd(t_1,R)-\ldd(t_2,R)| +O(R^{-2}) \\
\lesssim & R^{-2}|\lambda_1(t_1, R)-\lambda_1(t_2,R)| + R^{-2}\abs{\int_{t_2}^{t_1} \mu(t,R)dt} +O(R^{-2})\\
 \lesssim & |t_1-t_2| O(R^{-2+\epsilon}) +O(R^{-2+3\epsilon})
\end{align*}
where in the last step we used the $r^{3\epsilon}$-control on $\lambda_1(t,r)$ and $r^{\epsilon}$-control of $\mu(t,r)$ as in (\ref{epsilon-control}),  which hold uniformly in time. 
It follows that  $\vartheta_2(t_1)=\vartheta_2(t_2)$ by letting $R\rightarrow \infty$. 
\end{proof}
\begin{lemma}\label{2lCycle4} $\vartheta_2=0$. Moreover,  $\varrho(t)$ is independent of time, and from now on we will write  $\varrho = \varrho(t)$.
\end{lemma}
\begin{proof} 
From (\ref{mu-limit})  we have 
\[\varrho(t_1)-\varrho(t_2)  =\frac{1}{R}\int_R^{2R} \varrho(t_1)-\varrho(t_2) dr = \frac{1}{R}\int_R^{2R} \mu(t_1)-\mu(t_2) dr + O(R^{-1})\]
Now using  Lemma~\ref{mu-computation} with $j=1$ (note that  when $d=7$, we have $c_1=3$), we obtain 
\[\frac{1}{R}\int_R^{2R} \mu(t_1)-\mu(t_2) dr  = 3 \int_{t_2}^{t_1}I(1,1) + II(1,1)dt\]
Recall that  
\[I(1,1) = -\frac{1}{R} (u(t,r)r^{4})\Big|_{r=R}^{r=2R} -\frac{1}{R}\int_{R}^{2R} u(t,r)r^{3} dr\]
Plugging  in  (\ref{1step-decay}) we see that 
\[I(1,1) = -2\vartheta_2+ O(R^{-2+3\epsilon})\]
To estimate $II(1, 1)$, we note that 
from the point-wise estimates for  $V$ and  $\N$ in  (\ref{Vbound}) and (\ref{FGbound}) together with (\ref{1step-decay}), we get 
\[|-V(r)u +\N(r,u)|\lesssim r^{-8}|u| +r^{-3}|u|^2 +r^2 |u|^3\lesssim r^{-7} \]
Hence, 
\begin{equation*}\begin{aligned} 
| II(1,1) |=&\abs{\frac{1}{R}\int_R^{2R} r^3\int_r^\infty  [-V(s)u(t,s)+ \N(s,u(t,s))]s \,ds\,dr }
\\ = &  O(R^{-2}) 
\end{aligned}
\end{equation*}
Therefore we must have  
\begin{equation}\label{defeat-theta-rho}|\varrho(t_1)-\varrho(t_2)|= 6   | (t_1-t_2) \vartheta_2| + |t_1-t_2|O(R^{-2+3\epsilon}) + O(R^{-1})\end{equation}
Recalling that  $\varrho(t)$ is bounded uniformly in time, we rewrite the above as an expression for $\vartheta_2$. Leting $R\rightarrow \infty$ and then letting  $|t_1-t_2| \to \infty$, we obtain \[|\vartheta_2|= \frac{|\varrho(t_1)-\varrho(t_2)|}{6|t_1-t_2|} + O(R^{-2+3\epsilon}) + \frac{1}{|t_1-t_2|}O(R^{-1}) \displaystyle\underset{R, |t_1-t_2|\rightarrow \infty}{\longrightarrow} 0\]
which means that $\vartheta_2=0$. 
 
With the knowledge that  $\vartheta_2=0$, we see from  (\ref{defeat-theta-rho}) that 
\[|\varrho(t_1)-\varrho(t_2)|=   |t_1-t_2|O(R^{-2+3\epsilon}) + O(R^{-1})\underset{R \rightarrow \infty}{\longrightarrow}  0\]
for any fixed $t_1 \neq t_2$. Therefore  $\varrho(t) = \varrho$ is independent of time.  
\end{proof}

\begin{lemma} We must have $\varrho=0$.
\end{lemma}
\begin{proof} Suppose $\varrho\not=0$. Recall from~\eqref{mu-exp} and Lemmas~\ref{lem:tr} and \ref{2lCycle4}  that $\varrho$ satisfies  
\[3R^3\int_R^\infty u_t(t,s)sds =\varrho+ O(R^{-1})\] uniformly in time. 
It follows that we can fix $R$ large enough so that  $3R^3\int_R^\infty u_t(t,s)sds$ has the same sign as $\varrho$ and 
\[\abs{3R^3\int_R^\infty u_t(t,s)sds } \geq \frac12 |\varrho|\]
Integrating in time from $t = 0$ to $t=T$ gives 
\[\abs{\int_0^T 3R^3\int_R^\infty u_t(t,s)sds \,dt } \geq \frac{T}{2} |\varrho|\]
Carrying out the $t$-integration on the left hand side   and using  (\ref{1step-decay}) with the knowledge that $\vartheta_2(T)  = \vartheta_2(0) = \vartheta=0$, we see that 
\[|\int_0^T 3R^3\int_R^\infty u_t(t,s)sds \,dt | = |3R^3\int_R^\infty [u(T,s)-u(0,s)]sds\,dt |\lesssim O(R^{3\epsilon})\]
Therefore 
\[\frac{T}{2} |\varrho|\lesssim R^{3\epsilon}\]
which gives a contradiction if $\varrho \neq 0$ since $R$ is fixed and we are free to choose $T$ as large as we like. 
\end{proof}

  Now we are ready to prove that the leading coefficient $\ld(t,r)$ has a limit as $r \to \infty$.  At this point we it will suffice to consider the case $t =0$, and in the following we will simplify notation by writing $\la_1(r):= \la_1(0, r)$, $\la_2(r):= \la_2(0, r)$ and $\mu(r):= \mu(0, r)$. 
  \begin{lemma}[Existence of limit for the  leading coefficient $\la_1$]\label{lemma:ld-limit} There exist $\vartheta_1 \in \R$ so  that 
  \begin{equation}\label{ld-limit}
|  \ld(r) -\vartheta_1|=O(r^{-6}) \text{ as  }r\rightarrow \infty. 
  \end{equation}
\end{lemma}
\begin{proof} Using (\ref{ldd-limit}), (\ref{mu-limit}) and the fact that $\vartheta_2=\varrho=0$, we have improved decay rates for $\ldd(r)$ and $\mu(r)$, namely  
\EQ{ \label{la2mu}
|\ldd(r) |\lesssim r^{-2},  \quad |\mu(r)|\lesssim r^{-1}
}
Fixing a large $r_0$ and feeding~\eqref{la2mu} and the $r^{3\epsilon}$-control of $\la_1(r)$  back  into the difference estimate (\ref{true-difference1}), 
we obtain
\[|\ld(2^{n+1}r_0)-\ld(2^nr_0)|\lesssim (2^nr_0)^{-6+3\epsilon}\]
Arguing as as in the proof of Lemma~\ref{ldd-mu-limit}, we deduce that there exists $\vartheta_1 \in \R$ so that  $\ld(2^nr_0) \to \vartheta_1$ as $ n \to \infty$. Moreover, 
\[|\ld(2^nr_0)-\vartheta_1|\lesssim (2^nr_0)^{-6} \text{ as } r\rightarrow \infty\]
Finally,~\eqref{ld-limit} follows from another application of  the difference estimate (\ref{true-difference1}).
\end{proof}

We are now ready to complete the proof of Proposition~\ref{prop:ell2}. 

\begin{proof}[Proof of Proposition~\ref{prop:ell2}] Inserting the conclusions of Lemma~\ref{ldd-mu-limit} and  Lemma~\ref{lemma:ld-limit}, along with the facts that $\vartheta_2 = \varrho = 0$, into the difference estimates (\ref{true-difference2}) and (\ref{true-difference3}), we obtain the improved decay rates 
\begin{equation}\label{final-decay} |\ldd(r) |=O(r^{-8}), \quad \, 
 |\mu(r) | = O(r^{-7}) \text{ as } r\rightarrow \infty\end{equation}
  It then follows from   Lemma~\ref{lemma:ld-limit}, (\ref{final-decay}), and the identities  (\ref{u-ld}) and (\ref{u-mu}) with $t = 0$ that 
\begin{align*} & r^5u_0(r) =\vartheta_1  +O(r^{-6}) \text{ as } r\rightarrow \infty 
\\
&\int_r^\infty u_1( s)s ds =  O(r^{-10}) \text{ as } r\rightarrow\infty\end{align*}
 This completes the proof. 
 \end{proof}

\subsubsection{\textbf{Proof of Proposition~\ref{prop:as} when $\ell\geq 2$ is even:}}  \quad \\ 
Here we  have  \,$d=2\ell+3,\, k=[\frac{d}{4}]=\frac{\ell}{2},\, \tdk=[\frac{d+2}{4}]=\frac{\ell}{2}+1$.  We also note that  $\tdk=k+1$, $d=4\tdk-1=4k+3$.  

Recall that in the case  $\ell=2$, we had projection coefficients $\lambda_1, \lambda_2, \mu$.  We  first  showed that  $\lambda_2(t,r) \to 0$ and $ \mu(t,r) \to 0$ as $r \to \infty$. This then allowed us to extract a limit for $\la_1(r) = \la_1(0, r)$ as $r \to \infty$,  which in turn implied the desired asymptotics (\ref{initial-a}).  \hfill $\Box$
 
 \vspace{\baselineskip}

Now for an arbitrary even equivariance class $\ell\geq 2$, we have coefficients $\lambda_1, \ldots \lambda_{\tdk}$ and $\mu_1, \ldots \mu_k$.  We first will show inductively that each pair $\lambda_{j}(t,r)$, $\mu_{j-1}(t,r)$, with $2\leq j\leq \tdk$  satisfies 
 \ant{
\lambda_{j}(t,r) \to 0 \mas r \to \infty, \quad \mu_{j-1}(t,r) \to 0 \mas r \to \infty
}
This will then allow us to extract a limit for $\la_1(0, r)$ as $r \to \infty$. The argument is nearly identical to the case $\ell=2$, and simply requires more bookkeeping. 
  
  First we prove an initial growth estimate that is analogous to  Lemma~\ref{one-step-control}.  


 \begin{lemma}[$\epsilon$-control]\label{Cycle1} Given any small fixed number $0<\epsilon\ll 1$, we can find $\delta_1$ in \textnormal{(\ref{small-initial-all})} so that the projection coefficients defined in \textnormal{(\ref{u-projection})} satisfy the following  estimates uniformly in time. 
 \begin{equation}\label{1step-growth-even}  \left\{ \begin{aligned}
|\lambda_{\tdk}(t,r)\lesssim & \,r^{\epsilon},  \\
 |\mu_k(t,r)|\lesssim & \,r^{\epsilon}\\
 | \lambda_i(t, r)| \lesssim&  \,r^{2\tdk-2-2i+3\epsilon}\hspace{1cm}\forall 1\leq i < \tdk \\ 
 | \mu_i(t, r)| \lesssim& 
 \,r^{2\tdk-3-2i+3\epsilon} \hspace{1cm}\forall 1\leq i < k
   \end{aligned}\right.\end{equation}   
 \end{lemma}
 \begin{proof} From (\ref{all-difference-delta}),  for $1\leq j\leq \tdk$, we have for $r>R_1$
 \begin{align}\label{ld-j-d}
   |\lambda_j(t,2r)| 
\leq&  |\lambda_j(t,r)|+ C \delta_1 \left(\sum_{i=1 }^{\tdk} r^{2i-2j}|\lambda_i(t,r)| 
+ \sum_{i=1}^k  r^{2i-2j+1} |\mu_i(t,r)|\right)
 \end{align}
 And from (\ref{all-dd-mu}) for $1\leq j\leq k$, we have
 \begin{align}\label{mu-j-d}|\mu_j(t,2r) | 
 \leq&  |\mu_j(t,r) | + C\frac{\delta_1}{r} \left(\sum_{i=1}^{\tdk} r^{2i-2j}|\lambda_i(t,r)| + \sum_{i=1 }^k  r^{2i-2j+1} |\mu_i(t,r)|\right)
 \end{align}
 Now fix $r_0>R_1$  and define 
\begin{equation}\label{H:even}H^e_n = \sum_{j=1}^{\tdk} (2^nr_0)^{2j-2\tdk}|\lambda_j(t,2^nr_0)| 
+ \sum_{j=1}^k  (2^nr_0)^{2j-2\tdk+1} |\mu_j(t,2^nr_0)|\end{equation}
One can check using~\eqref{ld-j-d} and~\eqref{mu-j-d} that 
\EQ{ \label{Hen} 
H^e_{n+1}\leq (1 +C(k+\tdk)\delta_1)H^e_n
}
Now, given any fixed, small  $\epsilon>0$, 
we can find $\delta_1$ in (\ref{small-initial-all}) small enough such that $1 +C(k+\tdk)\delta_1<2^{\epsilon}$. 
Iterating~\eqref{Hen} we have 
\begin{equation*}
H^e_n\leq   (2^n)^{\epsilon} H^e_0\end{equation*}
Using (\ref{H:even}), it follows that    
\begin{equation}\label{base-growth-even}
 |\lambda_i(t,2^nr_0)| \leq  (2^nr_0)^{2\tdk-2i+\epsilon }, \hspace{0.5cm}  
|\mu_i(t,2^nr_0)|\leq (2^nr_0)^{2\tdk-2i-1+\epsilon}\end{equation}
We remark that if we compare (\ref{base-growth-even})  with (\ref{small-delta-control}),   we have achieved a nontrivial  improvement in the growth rate.

Now we insert  (\ref{base-growth-even})   back into (\ref{difference-all-ld}) and (\ref{difference-all-mu}).  Using the fact that $d=4\tdk-1$, we have 
 \begin{equation*} 
 \left\{\begin{aligned}
 |\lambda_j(t,2^{n+1}r_0) -\lambda_j(t, 2^nr_0)| \leq& C\delta_1|
 \lambda_j(t, 2^nr_0)|+ C(2^nr_0)^{2\tdk-2-2j+3\epsilon}\\
 |\mu_i(t,2^{n+1}r_0) -\mu_i(t, 2^nr_0)| \leq& C\delta_1|\mu_i(t, 2^nr_0)|+ C
(2^nr_0)^{2\tdk-3-2i+3\epsilon}\\
 \end{aligned} \right.\end{equation*}
From this we see that 
 \[ |\lambda_j(t,2^{n+1}r_0)|\leq    (1+C\delta_1)|
 \lambda_j(t, 2^nr_0)|+ C(2^nr_0)^{2\tdk-2-2j+3\epsilon}\]
 from which, using again that  $(1+C\delta_1)<2^\epsilon$, we obtain 
  \[ |\lambda_j(t,2^{n}r_0)|\leq    (2^\epsilon)^n|
 \lambda_j(t, r_0)|+ C\sum_{m=1}^n(2^mr_0)^{2\tdk-2-2j+3\epsilon}(2^{\epsilon})^{n-m}\]
We remark that when $j<\tdk$, the second term involving the summation is dominant, but if  $j=\tdk$, the first term is dominant. One can perform a similar  calculation for $\mu_j$. In summary, we have 
 \begin{equation}\label{1step-even}  \left\{ \begin{aligned}
|\lambda_{\tdk}(t,2^{n}r_0)\lesssim & (2^nr_0)^{\epsilon}\\
 \mu_k(t,2^{n}r_0)|\lesssim & (2^nr_0)^{\epsilon},\\
 | \lambda_i(t, 2^nr_0)| \lesssim&  (2^nr_0)^{2\tdk-2-2i+3\epsilon}\hspace{1cm}\forall 1\leq i <\tdk\\
 | \mu_i(t, 2^nr_0)| \lesssim& 
 (2^nr_0)^{2\tdk-3-2i+3\epsilon} \hspace{1cm}\forall 1\leq i <k
   \end{aligned}\right.\end{equation}   
 which is  an improvement over (\ref{base-growth-even}). We note that  since $r_0$ is fixed, all of the bounds are uniform in time. 
 
We use the difference estimates   (\ref{difference-all-ld}),(\ref{difference-all-mu}) to pass from  (\ref{1step-even}) to   (\ref{1step-growth-even}) 
 for any $2^nr_0 < r< 2^{n+1}r_0$ .
 
 We note that  that if  we  feed   (\ref{1step-even}) back into the system  (\ref{difference-all-ld}), (\ref{difference-all-mu}) 
 again, we will not achieve any direct  improvement  because the last term involving $|\lambda_{\tdk}|^3$ dominates the growth in the difference estimate (\ref{difference-all-ld}). 
\end{proof}

Next, we use Lemma~\ref{Cycle1} as the base case for an induction argument.  The goal is to establish the  following proposition, which  indicates that  the projection coefficients go to $0$ as $r\rightarrow\infty$ uniformly in time, with the possible exception of $\lambda_1(t,r)$. 
 \begin{prop}\label{weak-induction}
 Suppose that the equivariance class  $\ell\geq 2$ is even.  Let
 $\epsilon>0$ be the small fixed constant   from Lemma~\ref{Cycle1}.  Let $\lambda_j(t,r) $ and  $\mu_{j}(t,r) $ be the projection coefficients defined as in \textnormal{(\ref{u-projection})} for a solution   $\vec u(t)$ to \textnormal{(\ref{u eq})} as in Theorem~\ref{Rigidity}. 
  Then  the following the estimates hold uniformly in time
  \begin{equation}\label{final-est}\left\{\begin{aligned}
   |\lambda_{j}(t,r)|&\lesssim   r^{-2j +3\epsilon}\hspace{1cm} \forall \,2\leq j  \leq  \tdk\\
    |\lambda_1(t,r)|&\lesssim  r^{\epsilon}, \\
   |\mu_{j}(t,r)|&\lesssim    r^{-2j -1 +3\epsilon}  \hspace{1cm} \forall \,1\leq j  \leq  k
  \end{aligned}\right.\end{equation}
  \end{prop}

  As we mentioned above, we will prove Proposition~\ref{weak-induction} inductively. Indeed, Proposition~\ref{weak-induction} is a consequence of   the following Proposition by setting  $P = k$ below. 

   \begin{prop} \label{weak-stat} Under the same hypothesis as Proposition~\ref{weak-induction} the following estimates hold true for $P=0,1,\ldots k$,  uniformly in time. 
  \EQ{
  \label{induction-ld}\begin{cases} 
   |\lambda_{j}(t,r)|\lesssim   r^{2(\tdk-P-j)-2 + \app\epsilon}  \quad & \forall 1\leq j \leq \tdk, \text{ and } j\not=\tdk-P\\
    |\lambda_j(t,r)|\lesssim  r^{\epsilon}, &\mif j=\tdk-P
   \end{cases} 
}
  \EQ{
  \label{induction-mu}\begin{cases} 
   |\mu_{j}(t,r)|\lesssim    r^{2(k-P-j) -1 + \app\epsilon}\quad & \forall 1\leq j \leq k, \text{ and } j\not=k-P\\
  |\mu_j(t,r)|\lesssim  r^{\epsilon}, &\mif   j=k-P\\
 \end{cases} 
 }
  \end{prop}
  
   \begin{proof}[Proof of Proposition~\ref{weak-stat}]
 First, observe that  Lemma~\ref{Cycle1} covers the case  $P=0$.  Now we argue by induction.
 Suppose that~\eqref{induction-ld} and~\eqref{induction-mu}  are true for $P$ with $0\leq P\leq k-1$. We show that they must also  then hold for $P+1$. 
  We divide  the remainder of the proof into several  lemmas, namely    Lemma~\ref{Cycle2}  --  Lemma~\ref{Cycle7}. 
  
  \begin{lemma}\label{Cycle2} There exist bounded functions $\vartheta_{\tdk-P}(t), \varrho_{k-P}(t)$ such that 
  \begin{align}
  |\lambda_{\tdk-P}(t,r)-\vartheta_{\tdk-P}(t)| &\lesssim  O(r^{-4P-2}) \label{even-tdk-l} \\
|\mu_{k-P}(t,r)-\varrho_{k-P}(t)| &\lesssim O(r^{-4P-1}) \label{even-k-l}
  \end{align}
 where $O(\cdot)$ is uniform in time $t\in \R.$
  \end{lemma}
  \begin{proof} 
    First we insert our induction hypothesis (\ref{induction-ld}) into (\ref{difference-all-ld}).      
A quick computation shows that the cubic terms always dominate the growth rate in each summation term when $P\leq k-1$. Therefore, for $R_1$ as in (\ref{small-initial-all}) and for  $R_1< r<r'<2r$, we have  
      \begin{align}& |\lambda_{j}(t,r') -\lambda_{j}(t, r)|
\lesssim    r^{6\tdk-6P-9 -2j-d+9\epsilon} + r^{6\tdk-6P-2j-d-3+3\epsilon}\label{ld-P-difference}
 \end{align} 
  Here the second term on the right-hand-side above comes from the term involving $|\lambda_{\tdk-P}|^3$ in~\eqref{difference-all-ld},  which grows faster than the term involving  $|\mu_{k-P}|^3$. The first term on the right-hand-side arises by considering the remaining terms in~\eqref{difference-all-ld}.  
  
 Fix $r_0>R_1$, and   
  set $j=\tdk-P$ in~(\ref{ld-P-difference}).  Using the fact $d=4\tdk-1$, we get 
    \[|\ldp(t,2^{n+1}r_0)-\ldp(t,2^nr_0)|\lesssim (2^nr_0)^{-4P-2+3\epsilon}\] 
    This implies that the series below converges,  \[\sum_n|\ldp(t,2^{n+1}r_0)-\ldp(t,2^nr_0)|<\infty\]
    which in turn  implies that there exist a function $\thpt$ so that 
    \ant{
    \lim_{n \to \infty} \ldp(t,2^nr_0)  = \thpt 
    }
   Then, we have  
\begin{align*}|\thpt-\ldp(t, r_0)|&=\lim_{n\rightarrow \infty}|\ldp(t,2^{n+1}r_0)-\ldp(t,r_0)|\\
&\lesssim\lim_{n\rightarrow\infty} \sum_{m=1}^n|\ldd(t,2^{m+1}r_0)-\ldd(t,2^mr_0)| \\
&\lesssim   r_0^{-4P-2+3\epsilon}\sum_{m=1}^\infty (2^m)^{-4P-2+3\epsilon} \end{align*}
 We can conclude from the above that $\thpt$ is uniformly bounded since  (\ref{small-delta-control}) implies that $\ldp(t, r_0)$ is  uniformly bounded. 
As usual, we use  the difference estimate (\ref{ld-P-difference}), to conclude that, in fact, we have 
 $$\lim_{r\rightarrow\infty}\ldp(t,r)=\thpt,$$ which implies that 
 \ant{
 \abs{\ldp(t,r)} \lesssim 1 
 } 
 uniformly in time. 
 
  Similarly 
 we plug   (\ref{induction-mu}) into   (\ref{difference-all-mu}),   and show that  for $R_1< r<r'<2r$
   \begin{align*}& |\mup(t,r') -\mup(t, r)|  
\lesssim    r^{-4P-1+3\epsilon}
 \end{align*} 
Arguing as above, this implies that  there exists bounded function
 $\rpt$, so  that $$\lim_{r\rightarrow\infty}\mup(t,r)=\rpt$$ uniformly in time. 
 
 Using the  boundedness of $\ldp(t,r), \mup(t,r)$ in  (\ref{difference-all-ld}) (\ref{difference-all-mu}), we deduce that  for any $r>R_1$
  \[|\ldp(t,2^{n+1}r)-\ldp(t,2^nr)|\lesssim (2^nr)^{-4P-2}\] 
\[ |\mup(t,2^{n+1}r) -\mup(t, 2^nr)|  
\lesssim    (2^nr)^{-4P-1 } 
\]
and it follows that 
  \[|\ldp(t,r)-\thpt|\lesssim \sum_{n=0}^\infty (2^nr)^{-4P-2}\lesssim r^{-4P-2}\] 
\[ |\mup(t, r) -\rpt|  
\lesssim   \sum_{n=0}^\infty (2^nr)^{-4P-1 } \lesssim r^{-4P-1 }
\]
as desired.   
  \end{proof}
  Now we use the new information on $\lambda_j, \mu_j$ to write down a preliminary estimate regarding the asymptotic behavior  of $ u(t, r)$ as $r \to \infty$ assuming our induction hypothesis. 
  \begin{lemma}\label{Cycle3} We have the following preliminary estimates for $u(t, r)$: 
 \begin{align}
  r^{-2(\tdk-P)+d} u(t,r)=& \thpt  +O(r^{-2+ \app\epsilon})\label{EV-u-a}
 \end{align}
 As usual, the above holds uniformly in time. 
   \end{lemma}
   \begin{proof} The proof follows by plugging the estimates in (\ref{induction-ld}) along with the conclusion of  Lemma~\ref{Cycle2} into the formula (\ref{u-ld}) in Lemma~\ref{u-ld-mu}.
\end{proof} 
 \begin{remark}  The induction hypothesis (\ref{induction-ld}) and the corresponding asymptotics (\ref{EV-u-a}) mean that if we count backwards from $\tdk$ to $1$ to find the first $\lambda_j(t, r)$ that doesn't decay to $0$, then $u(t,r)$ decays like $r^{2j-d}$.  This coincides with the fact that $\lambda_j$ is the projection coefficient for $r^{2j-d}$.
   \end{remark}
 \begin{lemma}\label{Cycle4}
 $\thpt$  is independent of time and from now one we will write  $ \thpt = \thp$.
 \end{lemma}
 \begin{proof} 
Fix $t_1 \neq t_2$. Using (\ref{even-tdk-l}) and (\ref{ldjldj}) with $j=\tdk-P, j'=\tdk-P-1=k-P \geq 1$, together with (\ref{induction-ld}) and (\ref{induction-mu}) we have
 \begin{align*}
& |\thp(t_1)-\thp(t_2)|\\
=& \abs{\ldp(t_1,r)-\ldp(t_2,r)} + O(r^{-4P-2})\\
\lesssim & r^{-2} |\lambda_{k-P}(t_1,r)-\lambda_{k-P}(t_2,r)| +\abs{\sum_{m=1}^k \int_{t_2}^{t_1}r^{2m-2(\tdk-P)}\mu_m(t,r)dt} + O(r^{-4P-2})\\
\lesssim &r^{-2+3\epsilon}  +O\big(r^{-2+3\epsilon} (1+|t_1-t_2|)\big)  
 \end{align*}
   Letting $r\rightarrow \infty$ we see that 
 \[\thp(t_1)=\thp(t_2)\]
 as desired. 
  \end{proof}

 \begin{lemma}\label{Cycle5} $\thp=0$ and $\rpt$ is independent of time -- from now on we will write $\rpt = \rp$.  
 \end{lemma}
 \begin{proof}  We will use Lemma~\ref{mu-computation}
to prove both statements. Using
  (\ref{even-k-l})  we have 
\begin{align*}\rp(t_1)-\rp(t_2)  = & \frac{1}{R}\int_R^{2R} (\rp(t_1)-\rp(t_2) )\,  dr \\ = &\frac{1}{R}\int_R^{2R}( \mup(t_1)-\mup(t_2))  \, dr + O(R^{-4P-1})\end{align*} 
Now,  by setting $j=k-P$ in Lemma~\ref{mu-computation}, we obtain 
\[\frac{1}{R}\int_R^{2R}( \mup(t_1)-\mup(t_2) )dr  =   \sum_{i=1}^{k}\frac{c_ic_{j}}{d-2i-2j}\int_{t_1}^{t_2} (I(i,k-P)  + II(i,k-P)) \, dt\]
with the formulas for  $I(i,k-P), \,  II(i,k-P)$ given by  (\ref{I-formula})  and (\ref{II-formula}).
 
  From the point-wise estimates  for $V$ and  $\N$  in (\ref{Vbound}) and (\ref{FGbound}),  along with the  asymptotics  for $\vec u(t, r)$ in~\eqref{EV-u-a}, we see that 
\[|-V(r)u +\N(r,u)|\lesssim r^{-2\ell-4}|u| +r^{-3}|u|^2 +r^{2\ell-2} |u|^3\lesssim r^{-2\tdk-3-6P}  \]
Hence with $j=k-P$ we have 
\begin{equation*}\begin{aligned} 
II(i,k-P)=& \abs{  \frac{1}{R}\int_R^{2R} r^{d-2i-2j}\int_r^\infty  \big(-V(s)u(t,s)+ \N(s,u(t,s))\big)\, s^{2i-1} \,ds\,dr }
\\ \lesssim &     O(R^{-2-4P}) 
\end{aligned}
\end{equation*}
Next we examine  the main term 
\begin{equation*}\begin{aligned}
I(i,j)   = & -\frac{1}{R} (u(t,r)r^{d-2j-1})\Big|_{r=R}^{r=2R} + (2i-2j-1)\frac{1}{R}\int_{R}^{2R} u(t,r)r^{d-2j-2} dr\\
& - \frac{(2\ell-2i+3)(2i-2)}{R}\int_R^{2R}r^{d-2i-2j}\int_r^\infty u(t,s)s^{2i-3}ds\,dr
%
\end{aligned}
\end{equation*}
We plug in   (\ref{EV-u-a}) and observe  that for $j=k-P$
\[r^{d-2j-2}u(t,r)= \thp  + O(r^{-2+3\epsilon})\]
\[r^{d-2i-2j}\int_r^\infty u(t,s)s^{2i-3}ds =   \frac{\thp}{d-2i-2j}+ O(r^{-2+3\epsilon})\]
From these estimates we obtain
\[I(i,j)=\thp \frac{-2j(d-2j-2)}{d-2i-2j}+ O(R^{-2+3\epsilon})\]
Hence we get that the entire contribution from the term involving $I(i,j), 1\leq i\leq k$ is given by 
\[\sum_{i=1}^k\frac{c_ic_{j}}{d-2i-2j}\frac{-2j(d-2j-2)}{d-2i-2j} \thp (t_1-t_2) +O(R^{-2+3\epsilon}) \]
We will prove in Remark~\ref{number1} that  the coefficient in front of $(t_2- t_1)\thp$ is nonzero. We assume that this is so for the moment and  denote its absolute value  by $C>0$. Hence
\begin{align}\label{rho-theta-p}|\rp(t_1)-\rp(t_2)|= C  \abs{t_2- t_1}|\thp|  +  O(R^{-1}(1+|t_1-t_2|))
\end{align}
First we  let  $R\rightarrow \infty$ to obtain 
\[|\rp(t_1)-\rp(t_2)|= C  \abs{t_2- t_1}|\thp| \]
Then by taking 
  $|t_1-t_2|$ arbitrarily large, and  by the boundedness of 
 $\rp(t)$, we have 
  \[|\thp|= \frac{1}{C} \frac{|\rp(t_1)-\rp(t_2)|}{|t_1-t_2|}  \displaystyle\underset{ |t_1-t_2|\rightarrow \infty}{\longrightarrow} 0\]
which means $\thp=0$. 
 
But since $\thp=0$, we can go back to  (\ref{rho-theta-p}) to deduce that for fixed times $t_1 \neq t_2$ we have 
\[|\rp(t_1)-\rp(t_2)|=  O(R^{-1}(1+|t_1-t_2|))\underset{R \rightarrow \infty}{\longrightarrow}  0\]
which means that  $\rpt$ is independent of time.  
 \end{proof}
 \begin{remark}\label{number1} Here we prove that for  $1\leq j\leq k$ we have 
 \EQ{ \label{-C}
 \abs{\sum_{i=1}^k\frac{c_ic_{j}}{d-2i-2j}\frac{-2j(d-2j-2)}{d-2i-2j}} =  C \not=0
 }
 which was needed in the proof of Lemma~\ref{Cycle5} above. 
 
  We will  show that 
 $\sum_{i=1}^k\frac{c_i}{(d-2i-2j)^2}=\frac{1}{c_j}$.
 Arguing as in the proof of Lemma~\ref{contour-integral}, we set 
 $\al(z)=\Pi_{m=1}^{k}(z-x_m), \beta(z)=\Pi_{m=1}^{k}(z-y_m)$ with $x_m=2m, y_m=d-2m$ and 
 consider the contour integral
 \[\frac{1}{2\pi i}\oint_{\ga}\frac{\beta(z)}{\al(z)}\frac{1}{(z-y_j)^2}dz =\sum_{i=1}^{k}\Res(\frac{\beta(z)}{\al(z)}\frac{1}{(z-y_j)^2}, x_i) + \Res(\frac{\beta(z)}{\al(z)}\frac{1}{(z-y_j)^2}, y_j)\]
 where $\gamma$ is a large circle centered at the origin. The left-hand-side goes to $0$ as we let the radius of the circle $\gamma$ tend to $\infty$, and we can compute 
 \[\Res(\frac{\beta(z)}{\al(z)}\frac{1}{(z-y_j)^2}, x_i) =\frac{-c_i}{(d-2j-2i)^2} \]
 \[\Res(\frac{\beta(z)}{\al(z)}\frac{1}{(z-y_j)^2}, y_j)=\frac{\prod_{1\leq m \leq k, m \not=j}(y_j-y_m)}{\prod_{m=1}^k(y_j-x_m)}=\frac{1}{c_j}\]
 Therefore  $  \sum_{i=1}^k\frac{c_i}{(d-2i-2j)^2}=\frac{1}{c_j}$. In fact,  the number $C$ in~\eqref{-C} is given by $2j(d-2j-2)$, and we can check that this matches the corresponding number in Lemma~\ref{2lCycle4} in the case $\ell = 2$.
 \end{remark}

  \begin{lemma}\label{Cycle6} $\rp=0$.
 \end{lemma}
 \begin{proof} Suppose $\rp\not=0$. Using (\ref{even-k-l}), we see that when $R$ is large enough $\mup(t,R)$ will have the same sign as $\rp$ and 
 \begin{equation*}
 |\mup{(t,R)}| \geq \frac12 |\rp|.\end{equation*}
 Integrating $\mup(t,R)$ in time from $t=0$ to $t = T$ gives  
 \begin{align}\label{con-1}
 \abs{\int_0^T \mup(t,R)dt}\geq \frac12 T|\rp|
 \end{align}
On the other hand, using the explicit formula (\ref{mu-exp}) for $\mup$ we can deduce that  
 \begin{align}
 \abs{\int_0^T \mup(t,R)dt }\lesssim &\sum_{i=1}^k \abs{R^{d-2i-2(k-P)}\int_R^\infty\int_0^T u_t(t,s)s^{2i-1}dt ds }\nonumber\\
\lesssim &\sum_{i=1}^k \abs{R^{d-2i-2(k-P)} \int_R^\infty (u(T,s)-u(0,s))s^{2i-1}dt ds } \nonumber\\
\lesssim & O(R^{\app\epsilon})\label{con-2}
 \end{align}
 where we have used (\ref{EV-u-a}) and the fact that  $\thp=0$ above. Therefore, combining~\eqref{con-1} and~\eqref{con-2} we can find a large fixed $R$ so that for all $T$ we have 
 \ant{
 \frac12 T|\rp| \lesssim R^{\app\epsilon}
 }
 Letting  $T\rightarrow\infty$ above  gives a contradiction if $\rp \neq 0$.  Hence $\rp=0$. 
 \end{proof}
 
 Now that we have proved that $\ldp(t,r), \mup(t,r) \to 0 $ as $ r \to \infty$  we can feed this information   back into the difference estimates  (\ref{difference-all-ld}), (\ref{difference-all-mu}) to obtain improved decay. Here we make some simple observations (to prepare for the next lemma)
regarding an argument that we have used often.  
 
 \begin{remark}\label{rem:no-better} 
 We make two observations regarding  arguments that were used for example, in the proofs of Lemma~\ref{Cycle1} and Lemma~\ref{Cycle2} and will be used again below.  
 \begin{enumerate}
\item  Starting with the difference estimates (\ref{difference-all-ld}) and~\eqref{all-difference-delta},   suppose we can prove  
\begin{equation}\label{rem-est1}|\lambda_j(t,2^{n+1}r_0)-\lambda_j(t,2^nr_0)|\lesssim (2^nr_0)^a\end{equation}
 and/or
\begin{equation}\label{rem-est2}|\lambda_j(t,2^{n+1}r_0)|\leq (1+C\delta_1)\lambda_j(t,2^nr_0) +(2^nr_0)^a\end{equation}
for some number $a \in \R$ and where  $r_0>R_1$ is fixed. For a fixed $\epsilon>0$ we then  choose  $\delta_1, R_1$ in  (\ref{small-initial-all}) so that  $1+C_1\delta<2^\epsilon$.
\begin{itemize} 
\item  If  $a\geq \epsilon,$ then we can show by iterating (\ref{rem-est2})  that $|\lambda_j(t,r)|\lesssim r^a$ uniformly in time. See for example the proof of Lemma~\ref{Cycle1} for an argument of this nature. 

\item If $a < 0$, then  (\ref{rem-est1}) is enough to conclude that $\lambda_j(t,r)$ has limit $\vartheta_j(t)$ as $r \to \infty$, which is bounded in time, and moreover,  $$|\lambda_j(t,r)-\vartheta_j(t) |\lesssim r^a$$ uniformly in time. This further shows that $|\lambda_j(t,r)| \lesssim 1$. See for example the proof Lemma~\ref{Cycle2} for such an argument. 
\end{itemize}
\item By the observations in Remark~\ref{domination}, we note  that when we feed
  (\ref{induction-ld}) and  (\ref{induction-mu}) with $0\leq P\leq k-1$ into  (\ref{difference-all-ld}) and (\ref{difference-all-mu}), the cubic terms always dominate. But, when we insert  (\ref{induction-ld}) and (\ref{induction-mu}) with $P=k$ into  (\ref{difference-all-ld}) and (\ref{difference-all-mu}), the linear terms will dominate.
  \end{enumerate}
 \end{remark}
 
  \begin{lemma}\label{Cycle7}   If 
  ~\eqref{induction-ld} and~\eqref{induction-mu} are true for $P$ with $1\leq P\leq k-1$, then they are  also true for $P+1$.
 \end{lemma}
 \begin{proof}  We insert (\ref{even-tdk-l}), (\ref{even-k-l})  (recalling that we have proved that $\thp=\rp=0$) as well as the  induction hypothesis regarding the other  coefficients, i.e.,  (\ref{induction-ld}) and (\ref{induction-mu}),  into the system (\ref{difference-all-ld}) and (\ref{difference-all-mu}). This gives,   for $ R_1< r <r'<2r$,
 \begin{align}
 & |\lambda_j(t, r')-\lambda_j(t, r)|\nonumber\\
 \lesssim & r^{-2j-d+ 6(\tdk-P)-9 +9\epsilon} + r^{-2j-d +2(\tdk-P)-4P-1} + r^{-2j-d +6(\tdk-P)-12P-9}\label{true-est-ld}
 \end{align}
 where the first term on the right hand side arises from  the summations involving $|\lambda_i|$  with indices  $i\not=\tdk-P$ and the summations involving $|\mu_i|$ with  indices $i\not=k-P$. The second term comes from the linear terms involving $|\ldp|, |\mup|$, and the third term comes from  cubic terms involving $|\ldp|, |\mup|$.
 
 By  comparing the growth rates of each term on the right-hand-side of~\eqref{true-est-ld} with $2(\tdk-P-1 -j)-2+\app\epsilon$, we can simplify the above to  
  \begin{align*}
 |\lambda_j(t, r')-\lambda_j(t, r)|
 \lesssim  r^{2(\tdk-P-1 -j)-2+\app\epsilon}
 \end{align*}
 Noting that  (\ref{induction-ld}) implies $\lambda_j(t,r)\rightarrow 0$ for $\tdk-P <j\leq \tdk$, 
 we have, using the second bullet point in Remark~\ref{rem:no-better}, 
  \begin{align*} |\lambda_j(t, r)| &\lesssim  r^{2(\tdk-P-1 -j)-2+\app\epsilon},  \text{ when }\tdk-P\leq j\leq \tdk  \\
  |\lambda_j(t, r)| &\lesssim 1\lesssim r^{\epsilon},  \hspace{1.8cm}\text{ when }j=\tdk-(P+1)\end{align*}
However, if  $j<\tdk- (P+1)$, we have  $2(\tdk-P-1 -j)-2+\app\epsilon>\epsilon$, and hence, using the first bullet point in Remark~\ref{rem:no-better}, that   
    \[ |\lambda_j(t, r)| \lesssim  r^{2(\tdk-P-1 -j)-2+\app\epsilon},  \text{ when }1\leq j <  \tdk-(P+1) \]
The same argument shows that  
  \begin{align*}
 |\mu_j(t, r')-\mu_j(t, r)|
 \lesssim  r^{2(\tdk-P-1 -j)-1+\app\epsilon} =r^{2(k-P-1-j)+1+\app\epsilon}
 \end{align*}
and hence  
\begin{align*} |\mu_j(t, r)|
 &\lesssim   r^{2(k-P-1-j)+1+\app\epsilon}\text{ when } 1\leq j\leq k,  j \not=k-(P+1) \\
 |\mu_j(t, r)| &\lesssim 1 \lesssim r^{\epsilon} \hspace{1.8cm}\text{ when } j =k-(P+1)\end{align*}
 This completes the proof of Lemma~\ref{Cycle7}. 
 \end{proof}



Finally, we note that the work in Lemma~\ref{Cycle2} though Lemma~\ref{Cycle7} verifies the inductive step. This completes the proof of Proposition~\ref{weak-stat} (and hence of Proposition~\ref{weak-induction}).  
\end{proof}


The last step before completing the proof of Proposition~\ref{prop:as} is to show that $\la_1(r):= \la_1(t, 0)$ has a limit as $r \to \infty$. In what follows we will restrict to time $t =0$ and write $\la_j(r)  := \la_j(0, r)$ and $\mu_j(r) := \mu_j(t, r)$. 

\begin{lemma}\label{final-a} There exists a number  $\vartheta_1 \in \R$ so  that 
\begin{align}\label{final-ld-1}
|\lambda_1(r)-\vartheta_1|   = O( r^{1-d}) \mas r \to \infty
\end{align}
Also we have the improved decay estimates 
\begin{align}
 |\lambda_j(r)|& \lesssim r^{-2j-d+3}, \hspace{1cm} 2\leq j\leq \tdk \label{final-ld-2}\\
        |\mu_j(r)|& \lesssim r^{-2j-d+2},\hspace{1cm} 1\leq j\leq k  \label{final-mu}
\end{align}
\end{lemma}
  \begin{proof}
  First we use the estimates from    (\ref{final-est}) in (\ref{difference-all-ld}) for $\lambda_1$, which for a fixed  $r_0>R_1$, yields 
  \begin{align*}|\lambda_1(2^{n+1}r_0)-\lambda_1(2^nr_0)|\lesssim& (2^nr_0)^{1-d }|\lambda_1(2^nr_0)|^3 +(2^nr_0)^{-1-d+3\epsilon}\\
  \lesssim& (2^nr_0)^{1-d+3\epsilon}\end{align*}
As noted in Remark~\ref{rem:no-better}, this 
is enough to conclude that  there exists $\vartheta \in \R$  so that 
  $$\lambda_1(r) \to \vartheta  \mas r \to \infty$$ 
  and moreover   that $\abs{\lambda_1(r)} \lesssim 1$ is  bounded. 
  Using the new information that $\la_1(r)$ is bounded in (\ref{difference-all-ld}), we obtain
    \begin{align*}|\lambda_1(2^{n+1}r)-\lambda_1(2^nr)|\lesssim& (2^nr)^{1-d } \end{align*}
    Therefore,  \[|\lambda_1(r)-\vartheta_1| \lesssim \sum_{n  \ge 0}  \abs{\lambda_1(2^{n+1}r)-\lambda_1(2^nr)} \lesssim r^{1-d}\]
    
   Finally, combining the fact that  $\abs{\lambda_1(r)} \lesssim 1$ is bounded with (\ref{final-est})  we can deduce from  (\ref{difference-all-ld}) (\ref{difference-all-mu}), the following improved  decay for all of the other coefficients
    \EQ{ \label{last}
  &  |\lambda_j(2^{n+1}r)-\lambda_j(2^nr)|\lesssim (2^nr)^{-2j-d+3} \\
       & |\mu_j(2^{n+1}r)-\mu_j(2^nr)|\lesssim (2^nr)^{-2j-d+2}
       } 
Moreover, we know from (\ref{final-est})  that  then $\lambda_j(r) \to 0$ for $2\le j \le \ti k$ and  $\mu_j(r) \to 0$ for  each $ 1\leq j\leq k $ and  hence  (\ref{final-ld-2}) and  (\ref{final-mu}) follow from~\eqref{last}.
        \end{proof}
       
       We can now finish the proof of Proposition~\ref{prop:as} when $\ell \ge 2$ is even. 
       
       \begin{proof}[Proof of Proposition~\ref{prop:as} when $\ell \ge 2$ is even]
       
  We  insert  (\ref{final-ld-1}), (\ref{final-ld-2}), and (\ref{final-mu}) into the identities (\ref{u-ld}), (\ref{u-mu}) for $t = 0$. This gives 
        \begin{equation*}
        \begin{aligned}
r^{d-2}u_0(r) &= \vartheta_1  +O(r^{-d+1})  \mas r \to \infty\\
\int_r^\infty u_1(s)s^{2i-1}ds & =O(r^{2i + 2-2d}) \mas r \to \infty
\end{aligned}\end{equation*}
as desired. This completes the proof.      
  \end{proof}
  \subsubsection{\textbf{Proof of Proposition~\ref{prop:as} when $\ell\geq 2$ is odd}}  \quad \\
  In this case we  have $d=2\ell+3, k=\tdk=\frac{\ell+1}{2}$, and thus  $d=4k+1$. The proof is very similar to the argument given when $\ell \ge2$ is even, with a few subtle differences due to the different numerology involving $d, k,$ and $ \ti k$. In particular, in the current situation there is an extra term $\mu_k(t, r)$ which we must show tends to zero as $r \to \infty$ before beginning  the induction argument, which deals with the pairs $\la_{j}(t, r)$, $\mu_{j-1}(t, r)$ for $2 \le j \le k$. 
  
  We will give an outline of the proof, highlighting the slight differences that arise in the argument. However,  we will omit details and simply refer to the corresponding argument in the previous subsection.

 We begin  with the $r^{\epsilon}$-control, which is the analog of Lemma~\ref{Cycle1}. 
  \begin{lemma}[$\epsilon$-control]\label{OCycle1} Given any small fixed number $0<\epsilon\ll 1$,  we can find $\delta_1$ as in  \textnormal{(\ref{small-initial-all})} such that the projection coefficients defined in \textnormal{(\ref{u-projection})} satisfy the following  estimates   uniformly in time. 
 \begin{equation}\label{1step-growth-odd}  \left\{ \begin{aligned}
|\lambda_{k}(t,r)\lesssim & \,r^{\epsilon}\\
 |\mu_k(t,r)|\lesssim & \,r^{\epsilon}\\
 | \lambda_i(t, r)| \lesssim&  \,r^{2k-2i-1+3\epsilon}\hspace{1cm}\forall 1\leq i < k\\
 | \mu_i(t, r)| \lesssim& 
 \,r^{2k-2i-2+3\epsilon} \hspace{1cm}\forall 1\leq i < k
   \end{aligned}\right.\end{equation}   
 \end{lemma}
 \begin{proof}  From (\ref{all-difference-delta}) and (\ref{all-dd-mu}), we have for all $r \ge R_1$ 
 \begin{align*}
   |\lambda_j(t,2r)| 
\leq&  |\lambda_j(t,r)|+ C \delta_1 \left(\sum_{i=1 }^{k} r^{2i-2j}|\lambda_i(t,r)| 
+ \sum_{i=1}^k  r^{2i-2j+1} |\mu_i(t,r)|\right)
\\
|\mu_j(t,2r) | 
 \leq&  |\mu_j(t,r) | + C\frac{\delta_1}{r} \left(\sum_{i=1}^{k} r^{2i-2j}|\lambda_i(t,r)| + \sum_{i=1 }^k  r^{2i-2j+1} |\mu_i(t,r)|\right) 
 \end{align*}
 Now fix   $r_0>R_1$  and define
\begin{equation}\label{H:odd}
H^o_n = \sum_{j=1}^{k} (2^nr_0)^{2j-2k-1}|\lambda_j(t,2^nr_0)| 
+ \sum_{j=1}^k  (2^nr_0)^{2j-2k} |\mu_j(t,2^nr_0)|\end{equation} 
Iterating, we obtain 
\begin{equation*}
H^o_{n+1}\leq (1+2kC\delta_1)H^o_n\end{equation*}
Now given $\epsilon>0$,  we can find $\delta_1$ small enough so that  $(1+2kC\delta_1)<2^\epsilon$, and hence\[H^o_n\leq 2^{n\epsilon}H^o_0\]
This then  implies that for $1\leq j\leq k$, 
\begin{equation}\label{base-growth-odd} |\lambda_j(t,2^nr_0)| \lesssim  (2^nr_0)^{2k+1-2j +\epsilon}, \hspace{0.5cm}  |\mu_j(t,2^nr_0)|\lesssim (2^nr_0)^{2k-2j+ \epsilon}\end{equation}
which is an improvement over (\ref{small-delta-control}). 

Arguing as in the proof of Lemma~\ref{Cycle1}, using (\ref{difference-all-ld})  from Lemma~\ref{lem:all-difference} we obtain 
\begin{align*} &|\lambda_j(t,2^{n+1}r_0) -\lambda_j(t, 2^nr_0)| 
\lesssim C\delta_1|\lambda_j(t, 2^nr_0)| + C'(2^nr_0)^{2k-2j-1+3\epsilon}
\\
 &|\mu_i(t,2^{n+1}r_0) -\mu_i(t, 2^nr_0)| \leq C\delta_1|\mu_i(t, 2^nr_0)|+ C' (2^nr_0)^{2k-2i-2+3\epsilon}
 \end{align*}
which yields (using also \ref{base-growth-odd}) in the case of  $\mu_k$), 
\begin{equation*} 
\left\{\begin{aligned}
 |\mu_{k}(t,2^{n}r_0)|\lesssim & (2^nr_0)^{\epsilon} \\
 |\lambda_k(t,2^{n+1}r_0)|\leq &(1+C\delta_1) |\lambda_k(t, 2^nr_0)| +C' (2^nr_0)^{-1+3\epsilon} \\
 |\lambda_j(t,2^{n+1}r_0)|\leq& (1+C\delta_1) |\lambda_j(t, 2^nr_0)|  +C'(2^nr_0)^{2k-2j-1+3\epsilon}, \, \quad  \forall 1\leq j <k\\
 |\mu_i(t,2^{n+1}r_0)|\leq  &(1+C\delta_1) |\mu_i(t, 2^nr_0)|  +C'(2^nr_0)^{2k-2i-2+3\epsilon},  
 \quad \, \forall 1\leq i <k
\end{aligned} \right.\end{equation*}
The usual argument then gives  
 \begin{equation*} 
\left\{\begin{aligned}
   |\lambda_k(t,2^nr_0)|\lesssim & (2^nr_0)^{\epsilon},\hspace{1cm} |\mu_{k}(t,2^{n}r_0)| \lesssim  (2^nr_0)^{\epsilon}\\
 |\lambda_j(t,2^nr_0)| \lesssim & (2^nr_0)^{2k-2j-1+3\epsilon} \hspace{0.2cm} \forall 1\leq j <k\\ |\mu_i(t,2^nr_0)|\lesssim& (2^nr_0)^{2k-2i-2+3\epsilon}\hspace{0.2cm}\forall 1\leq i <k\end{aligned} \right. \end{equation*}
which is an  improvement over (\ref{base-growth-odd}). Lemma~\ref{OCycle1} then follows from another application of the difference estimates. 
 \end{proof}
 \begin{remark} We remark on two subtle differences between the $\ell$ even and $\ell$ odd cases. One is the choice of slightly different functions $H^e_n$ (\ref{H:even}) and $H^o_n$ (\ref{H:odd}) above.  Another occurs  when we plug the $\epsilon-$control Lemma~\ref{Cycle1} or~\ref{OCycle1} into the difference  estimates (\ref{difference-all-ld}) (\ref{difference-all-mu}).  
When $\ell$ is even, the dominant growth rate  comes from the  $\lambda_j$ terms, while when $\ell$ is odd, the dominant growth rate   comes from the $\mu_j$ terms. 

 \end{remark}
 
 As we mentioned  above,  we first show  that $\mu_k(t,r)\rightarrow 0$ as $r\rightarrow \infty$ before performing the  inductive argument on pairs $\lambda_j, \mu_{j-1}$ as in the even case. 
 
 \begin{lemma}\label{OCycle2} There exists a bounded function $\varrho_k(t)$ for each $t\in \R$ so  that 
 \begin{equation}\label{muk} |\mu_k(t,r)-\varrho_k(t)|=O(r^{-2}) \text{ as } r\rightarrow \infty
 \end{equation}
 where the $O(\cdot)$ above  is uniform in time. 
 \end{lemma}
 \begin{proof} By the now usual argument, we fix  $r_0>R_1$ and plug (\ref{1step-growth-odd}) into the difference estimates (\ref{difference-all-mu}), 
 \[|\mu_k(t, 2^{n+1}r_0) -\mu_k(t, 2^nr_0)| \lesssim (2^nr_0)^{-2+3\epsilon}\]
 As in Remark~\ref{rem:no-better} we can deduce from the above that $\mu_k(t, 2^nr_0)$ converges to a  bounded function $\varrho_k(t)$,  which in turn implies that $\mu_k(t,2^nr_0)$ is uniformly bounded.   From the difference relation (\ref{difference-all-mu}) we conclude that in fact $\lim_{r \to \infty} \mu_k(t, r)  = \varrho_k(t)$. 
 
Using this new information together with (\ref{1step-growth-odd}) we can deduce from  (\ref{difference-all-mu}) that  
 \[|\mu_k(t, 2^{n+1}r) -\mu_k(t, 2^nr)| \lesssim (2^nr)^{-2}\]
whence
 \[|\mu_k(t, r) -\varrho_k(t)| \lesssim r^{-2}\]
 uniformly in time. 
 \end{proof}
 \begin{lemma} $\varrho_k(t)$ is independent of time and from now on we write  $\varrho_k(t) = \varrho_k $.
 \end{lemma}
 \begin{proof} First we note that (\ref{u-ld}) together with (\ref{1step-growth-odd}) imply  that 
 \begin{equation*}
 |u(t,r)|=\abs{\sum_{j=1}^k \lambda_j(t,r)r^{2j-d}} \lesssim r^{2k-d+\epsilon}\end{equation*}
Thus, we have 
(\ref{EV-u-a})
\begin{equation}\label{II-bound}|-V(r)u +\N(r,u)|\lesssim r^{-2\ell-4}|u| +r^{-3}|u|^2 +r^{2\ell-2} |u|^3\lesssim r^{-2k-7+3\epsilon}\end{equation}
Using Lemma~\ref{OCycle2} and Lemma~\ref{mu-computation} we obtain  
\begin{align*}\varrho_k(t_1)-\varrho_k(t_2) & =\frac{1}{R}\int_{2R}^{R} \mu_k(t_1)-\mu_k(t_2)dr + O(R^{-2})\\
& = \sum_{i=1}^{k}\frac{c_ic_{k}}{d-2i-2k}\int_{t_1}^{t_2} I(i,k)  + II(i,k)dt
 + O(R^{-2})\end{align*}
with $I, II$ as in (\ref{I-formula}), (\ref{II-formula}).

 From (\ref{II-bound}) we can  estimate $II$ as follows  
\begin{align*}
|II(i,k)|= & | \frac{1}{R}\int_R^{2R} r^{d-2i-2k}\int_r^\infty  [-V(s)u(t,s)+ \N(s,u(t,s))]s^{2i-1}ds  dr|\\
\lesssim& O(R^{-6+3\epsilon} )
\end{align*}
For the main term $I(i, k)$,  we have  
\begin{align*}
I(i,k)  
 = & -\frac{1}{R} (u(t,r)r^{d-2k-1})\Big|_{r=R}^{r=2R} + (2i-2k-1)\frac{1}{R}\int_{R}^{2R} u(t,r)r^{d-2k-2} dr\\
& - \frac{(2\ell-2i+3)(2i-2)}{R}\int_R^{2R}r^{d-2i-2k}\int_r^\infty u(t,s)s^{2i-3}ds\,dr\\
=& O(r^{-2+\epsilon})
\end{align*}
In summary, we have proved that  
\[|\varrho_k(t_1)-\varrho_k(t_2)|=O(R^{-2+\epsilon}|t_1-t_2|) +O(R^{-2}) \]
By letting  $R\rightarrow \infty$, we see that  $\varrho_k(t_1)=\varrho_k(t_2)$ for any $t_1\not=t_2.$   \end{proof}
\begin{lemma} $\varrho_k=0$ \label{lem:rhok}
\end{lemma}
\begin{proof} The proof is identical to the argument in the proof of Lemma~\ref{Cycle6}. We omit the details.
\end{proof}

 Now we can give slight upgrades to  the estimates for each of the $\lambda_j, \mu_j$.
\begin{lemma}\label{OCycle3}The following estimates hold true uniformly in time. 
\begin{align*}|\lambda_j(t,r)| & \lesssim \max(r^{\epsilon}, r^{2k-2j-4+3\epsilon})\hspace{1cm} 1\leq j\leq k\\
|\mu_j(t,r)| &\lesssim \max(r^{\epsilon}, r^{2k-2j-5+3\epsilon})\hspace{1cm} 1\leq j < k\\
|\mu_k(t,r)|&\lesssim r^{-5+3\epsilon}\end{align*}
\end{lemma}
\begin{proof} We plug  the estimates (\ref{1step-growth-odd}) together with~(\ref{muk}) and the conclusion of Lemma~\ref{lem:rhok},  into the difference estimates (\ref{difference-all-mu}). For all  $R_1 < r \leq r' <2r$ we have 
\[|\mu_k(t, r)-\mu_k(t, r')|\lesssim r^{-5+3\epsilon}\]
Since we have  already proved that  $\mu_k(t,r) \to 0 $ as $r \to \infty$ it follows that 
\[|\mu_k(t,r)|\lesssim r^{-5+3\epsilon}\] 
Now we plug the above  together with  (\ref{1step-growth-odd}) into the difference estimates (\ref{difference-all-ld}), (\ref{difference-all-mu}).  Noticing that  the dominant growth rate come from the terms involving $|\lambda_k|^3$ we obtain  
\[|\lambda_j(t, r)-\lambda_j (t, r')|\lesssim r^{2k-2j-4+3\epsilon}\]
\[|\mu_j(t, r)-\mu_j (t,r')|\lesssim r^{2k-2j-5+3\epsilon}\]
which yields  (see Remark~\ref{rem:no-better}), 
\[|\lambda_j(t,r)|\lesssim \max(r^{\epsilon}, r^{2k-2j-4+3\epsilon})\hspace{1cm} 1\leq j\leq k\]
\[|\mu_j(t,r)|\lesssim \max(r^{\epsilon}, r^{2k-2j-5+3\epsilon})\hspace{1cm} 1\leq j < k\]
This completes the proof. 
\end{proof}

We can now formulate  the analog of Proposition~\ref{weak-induction} for the case when $\ell$ is odd.  
 \begin{prop}\label{O-weak-induction}
 Suppose that the equivariance class  $\ell \ge 2$ is odd.  Let  
 $\epsilon>0$ be a small fixed  constant from  Lemma~\ref{OCycle1}.  Let $\lambda_j(t,r) $ and  $\mu_{j}(t,r) $ be the projection coefficients defined as in \textnormal{(\ref{u-projection})} for a solution   $u$ to \textnormal{(\ref{u eq})}. 
 
 If $\vec{u}(t)$ has pre-compact  trajectory in $\Hd$, then   the following the estimates hold uniformly in time
  \begin{equation*}
  \left\{\begin{aligned}
   |\lambda_{j}(t,r)|&\lesssim   r^{-2j -2+\epsilon}\hspace{1cm} \forall \,2\leq j  \leq  k\\
    |\lambda_1(t,r)|&\lesssim  r^{\epsilon}, \\
   |\mu_{j}(t,r)|&\lesssim    r^{-2j -3 +\epsilon}  \hspace{1cm} \forall \,1\leq j  \leq  k
  \end{aligned}\right.\end{equation*}
  \end{prop}
  As in the case of even $\ell$, Proposition~\ref{weak-induction} is proved inductively. Indeed, Proposition~\ref{weak-induction} is a consequence of   the following Proposition by setting  $P = k-1$ below.
\begin{prop} \label{O-weak-stat} The following estimates hold true uniformly in time for $P=0, 1,\ldots k-1$.
     \begin{equation}\label{induction-ld-o}\left\{\begin{aligned}
   |\lambda_{j}(t,r)|&\lesssim   r^{2(k-P-j)-2 + \epsilon}\hspace{1cm} \forall 1\leq j \leq k, \text{ and } j\not=k-P\\
    |\lambda_j(t,r)|&\lesssim  r^{\epsilon}, \hspace{5cm} j=k-P
   \end{aligned}\right.\end{equation}
    \begin{equation}\label{induction-mu-o}\left\{\begin{aligned}
   |\mu_{j}(t,r)|&\lesssim    r^{2(k-P-j) -3 + \epsilon}\hspace{1cm} \forall 1\leq j \leq k, \text{ and } j\not=k-P-1\\
  |\mu_j(t,r)|&\lesssim  r^{\epsilon}, \hspace{4cm}   j=k-P-1\\
  \end{aligned}\right.\end{equation}
  \end{prop}

\begin{proof}[Proof of Proposition~\ref{O-weak-stat}] We proceed by  induction. The base case $P=0$  follows from the slightly stronger conclusions of Lemma~\ref{OCycle3}. Now assume that~\eqref{induction-ld-o} and~\eqref{induction-mu-o} hold for $P$ with $0  \le P \le k-2$. We show that they must then hold for $P+1$. 
  
 We note that when we plug (\ref{induction-ld}) and (\ref{induction-ld}) into (\ref{difference-all-ld}), (\ref{difference-all-mu}), the cubic terms  dominate the growth rate, as long as $P\leq k-2$;  see Remark~\ref{domination}. 
 
 As in the proof of Proposition~\ref{weak-stat}, we proceed in a sequence of steps. As the proof is nearly identical to the proof of Proposition~\ref{weak-stat} we will omit nearly all the details below. 
  
  \noindent\textbf{Step $1$:} There exist bounded functions $\vartheta_{k-P}(t), \varrho_{k-P-1}(t)$ such that 
  \begin{align*}
  |\lambda_{k-P}(t,r)-\vartheta_{k-P}(t)| &\lesssim  O(r^{-4P-4}) 
  \\
|\mu_{k-P-1}(t,r)-\varrho_{k-P-1}(t)| &\lesssim O(r^{-4P-3}) 
  \end{align*}
 where $O(\cdot)$ is uniform in time $t\in \R.$
  \begin{proof} The proof is identical to the proof of Lemma~\ref{Cycle2}.
  \end{proof}
     

 
  Now we use information on $\lambda_j, \mu_j$ to  find the asymptotic behavior for $u$. 
 
 \noindent \textbf{Step $2$:} We have the following asymptotics for $ u(t, r)$ which hold uniformly in time. 
 \begin{align*}
  u(t,r)=& \thpto r^{2(k-P)-d} +O(r^{2(k-P)-d-2+ \epsilon})
  \end{align*}
 \begin{proof} The proof follows by plugging  the estimates in~(\ref{induction-ld-o})  along with the conclusions of Step $1$  into (\ref{u-ld}).
\end{proof} 

 \noindent \textbf{Step $3$:} $\thpto$  is independent of time and from now on we will write  $\thpto = \thpo$.
 \begin{proof} The proof follows by arguing exactly as in the proof of Lemma~\ref{Cycle4}. 
 \end{proof}

 \noindent \textbf{Step $4$:} $\thpo=0$,  and $\rpto$ is independent of time. 
 \begin{proof}  The proof is identical to the argument given for Lemma~\ref{Cycle5}. 
 \end{proof}
 \noindent\textbf{Step $5$:} $\rpo=0$.
 \begin{proof}  This follows from the same argument as the one used to prove Lemma~\ref{Cycle6}. 
 \end{proof}
 

 Finally, we use the same argument as the proof of Lemma~\ref{Cycle7} to complete the proof of the induction step, namely that if~\eqref{induction-ld-o} and~\eqref{induction-mu-o} hold for $P$, then they also hold for $P+1$. 
This completes the proof of Proposition~\ref{O-weak-stat} and hence of Proposition~\ref{O-weak-induction}.
 \end{proof}
Finally, we can argue as in the proof of Lemma~\ref{final-a} to establish the following lemma. Again we write $\la_j(t):= \la_j(0, r)$ and $\mu_j(r):= \mu_j(0, r)$. 
\begin{lemma}\label{final-a-o} There exists  $\vartheta_1 \in \R$ so that 
\begin{align}\label{final-ld-1-o}
|\lambda_1(r)-\vartheta_1|  = O( r^{1-d} ) \mas r \to \infty
\end{align}
Moreover,  we have the improved decay estimates 
\begin{align}
 |\lambda_j(r)|& \lesssim r^{-2j-d+3}, \hspace{1cm} 2\leq j\leq k \label{final-ld-2-o}\\
        |\mu_j(r)|& \lesssim r^{-2j-d+2},\hspace{1cm} 1\leq j\leq k  \label{final-mu-o}
\end{align}
\end{lemma}


         
         We are now able to complete the proof of Proposition~\ref{prop:as}. 
          
          \begin{proof}[Proof of Proposition~\ref{prop:as} when $\ell \ge 2$ is odd]
     
  To conclude, we  plug  (\ref{final-ld-1-o}), (\ref{final-ld-2-o}), and (\ref{final-mu-o}) into the identities (\ref{u-ld}), (\ref{u-mu}) to  obtain
       \ant{
&r^{d-2}u_0(r) = \vartheta_1  +O(r^{-d+1})  \mas r \to \infty\\
&\int_r^\infty u_1(s)s^{2i-1}ds  =O(r^{2i + 2-2d}) \mas r \to \infty
}
as desired. 
  \end{proof}
  
 
\subsection{Rigidity argument. Step III:  nonexistence of the critical element}\label{kill}
 Finally, we complete the proof of Proposition~\ref{Rigidity} by showing that $\vec u(t) = (0, 0)$. We divide this argument into two cases depending on the value of the  number $\vartheta_1$ from Proposition~\ref{prop:as}.  
\\

\noindent \textbf{Case 1: $\vartheta_1=0 \Longrightarrow \vec u(0) \equiv (0, 0)$:} Here we show that if $\vartheta_1 = 0$ then $ \vec u(t) $ must be identically zero. As a preliminary step we first prove that $\vec u(0)$ must be compactly supported. 


\begin{lemma}\label{compact-support} Let $\vec u(t)$ be as in Proposition~\ref{Rigidity} and let $\vartheta_1$ be as in Proposition~\ref{prop:as}. 
 If $\vartheta=0$, then $\vec{u}(0)=(u_0, u_1)$ must be compactly supported. 
\end{lemma}
\begin{proof} The proof is similar to~\cite[Proof of Lemma~$5.13$]{KLS}, which was inspired by~\cite{DKM5}. If $\vartheta_1=0, $ we see from 
Lemma~\ref{final-a} and  Lemma~\ref{final-a-o}, that 
\begin{equation*}\begin{aligned}
|\lambda_i(r)| &= O(r^{-2i-d+3}),1\leq i\leq \tdk \\ |\mu_i(r)| &=O(r^{-2i-d+2}), 1\leq i \leq k 
\end{aligned}\end{equation*}
Therefore,  
\begin{equation}\label{sum}
\sum_{j=1}^{\tdk}|\lambda_j(2^n {r}_0)|+\sum_{j=1}^k |\mu_j( 2^n{r}_0)|\lesssim (2^nr_0)^{1-d}
\end{equation}
Using  the difference relations (\ref{difference-all-ld})  and (\ref{difference-all-mu}) we can deduce that  
 for  any $r_0>R_1$ 
\begin{align*}
|\lambda_j(2^{n+1}r_0)|&\geq (1-C\delta_1)|\lambda_j(2^nr_0)| \\
& -  C(2^nr_0)^{2\tdk-d}\left(\sum_{i=1,i\not= j}^{\tdk}  |\lambda_i(2^nr_0)| +
\sum_{i=1,i\not= j}^{\tdk} |\mu_i(2^nr_0)|\right)\\
|\mu_j(2^{n+1}r_0)|&\geq (1-C\delta_1)|\mu_j(2^nr_0)| \\
& -  C(2^nr_0)^{2\tdk-d}\left(\sum_{i=1,i\not= j}^{\tdk}  |\lambda_i(2^nr_0)| +
\sum_{i=1,i\not= j}^{\tdk} |\mu_i(2^nr_0)|\right)
\end{align*}
Now, choose  ${r}_0$ large enough and $\delta_1>0$ small enough so that $C(\delta_1 + 2\tdk(r_0)^{2\tdk-d})<\frac14$. Iterating the above then yields 
\EQ{ \label{34}
\sum_{j=1}^{\tdk}|\lambda_j(2^n {r}_0)|+\sum_{j=1}^k|\mu_j( 2^n{r}_0)|\geq \left(\frac34\right)^n \left(\sum_{j=1}^{\tdk}|\lambda_j( {r}_0)|+\sum_{j=1}^k\mu_j({r}_0)|\right) 
}
Putting (\ref{sum}) together with~\eqref{34} gives 
 \[ \sum_{j=1}^{\tdk}|\lambda_j( {r}_0)|+\sum_{j=1}^k|\mu_j( {r}_0)|\lesssim  \frac{1}{3^n}\frac{1}{2^{n(d-3)}}{r_0}^{1-d}\]
This implies that 
 \[|\lambda_i( {r}_0)|=|\mu_j( {r}_0)| =0, \hspace{1cm}\forall 1\leq i\leq \tdk, 1\leq j \leq k\]
It then follows from  (\ref{all-pip-identity}) and Proposition~\ref{Decay-Thm} that 
$$\|\vec{u}(0)\|_{\mathcal{H}(r\geq  {r}_0)}=0$$ 
Hence $(\partial_r u_0, u_1)$ are compactly supported. Finally,  because we know that $$\lim_{{r\rightarrow \infty}}u_0(r)=0$$   we can conclude that $u_0$ is compactly supported as well. 
 \end{proof}

\begin{lemma}\label{0solution} Let $u(t)$ be as in Theorem~\ref{Rigidity}, and $\vartheta_1 $ as in Proposition~\ref{prop:as}. Suppose $\vartheta_1=0$. Then $\vec u\equiv (0, 0)$.
\end{lemma}
\begin{proof} We again argue as in~\cite[Proof of Lemma~$5.13$]{KLS}, which was inspired by~\cite{DKM5}.  Suppose that $\vartheta_1 = 0$. By Lemma~\ref{compact-support} we know that then $ (u_0, u_1)$ are compactly supported. Now we assume $(u_0,u_1)\not=(0,0)$, and argue by contradiction. 

Find  $\rho_0=\rho(u_0,u_1)>1$  so that 
\begin{equation*}
\rho_0:=\inf\{\rho: \|\vec{u}(0)\|_{\Hd(r\geq \rho)}=0\}\end{equation*}
Let $\varepsilon>0$ be a small number to be determined  below, and  find $\rho_1=\rho_1(\varepsilon)$, with $1<\rho_1<\rho_0$ such that 
\[0< \|\vec{u}(0)\|_{\Hd(r\geq \rho_1)}^2 <\varepsilon^2 <\delta_1^2\]
where $\delta_1$ is  as in (\ref{small-initial-all}).

By  Lemma~\ref{ldmu-identity} we have 
 \EQ{\label{all-energy-formula}
 &\|\vec{u}(0)\|_{\hrr}^2
 \simeq 
  \left(\sum_{i=1}^{\tdk}(\lambda_i(R)
R^{2i-\frac{d+2}{2}})^2 +\sum_{i=1}^k (\mu_i(R) R^{2i-\frac{d}{2}})^2\right)\\ + &\int_R^\infty \left(\sum_{i=1}^{\tdk}  (\partial_r\lambda_i(r)r^{2i-\frac{d+1}{2}} )^2  + 
\sum_{i=1}^k(  {\partial_r\mu_i(r)
r^{2i-\frac{d-1}{2}}} )^2\right)dr
}
By setting  $R=\rho_0$ above it follows that $\lambda_j(\rho_0)=\mu_j(\rho_0)=0.$

By Proposition~\ref{Decay-Thm}, and its reformulation in  Lemma~\ref{all-decay}, we see that 
\EQ{\label{all-energy-control} 
&\int_{\rho_1}^\infty  \left( \sum_{i=1}^{\tdk}  (\partial_r\lambda_i(r)r^{2i-\frac{d+1}{2}} )^2  +
\sum_{i=1}^k(  {\partial_r\mu_i(r)
r^{2i-\frac{d-1}{2}}} )^2\right)\,dr 
\\
\lesssim & \sum_{i=1}^{\tdk} \left({\rho_1}^{4i-3d}\lambda_i^2({\rho_1}) + {\rho_1}^{8i-3d-4}\lambda_i^4({\rho_1}) +\rho_1^{12i-3d-8}\lambda_i^6(\rho_1)\right)\\
&+ \sum_{i=1}^k \left({\rho_1}^{4i+2-3d}\mu_i^2({\rho_1}) + {\rho_1}^{8i-3d}\mu_i^4({\rho_1}) +{\rho_1}^{12i-3d-2}\mu_i^6({\rho_1})\right)
}
 Now we use the fundamental theorem of calculus to express the differences $ |\lambda_j(\rho_1)-\lambda_j(\rho_0)|$ and $|\mu_j(\rho_1)-\mu_j(\rho_0)| $ in terms of~\eqref{all-energy-control}  and argue exactly as in the proofs of 
 Lemma~\ref{lem:all-difference} and  Corollary~\ref{cor:all-difference} to obtain 
   \begin{align*} 
  |\lambda_j(\rho_1)-\lambda_j(\rho_0)|
\lesssim &  \, \varepsilon \left(\sum_{i=1}^{\tdk} \rho_1^{2i-2j}|\lambda_i(\rho_1)| + \sum_{i=1}^k  \rho_1^{2i-2j+1} |\mu_i(\rho_1)|\right)\\
|\mu_j(\rho_1)-\mu_j(\rho_0)| 
 \lesssim & \,  \frac{\varepsilon}{\rho_1}\left(\sum_{i=1}^{\tdk} \rho_1^{2i-2j}|\lambda_i(\rho_1)| + \sum_{i=1}^k  \rho_1^{2i-2j+1} |\mu_i(\rho_1)|\right)
  \end{align*}
Recalling that $\lambda_j(\rho_0)=\mu_j(\rho_0)=0$, and setting 
\[H := 
\sum_{j=1}^{\tdk}  \rho_1^{2j}|\lambda_j(0,\rho_1)| 
+ \sum_{j=1}^k \rho_1^{2j+1} |\mu_j(0,\rho_1)|\]
 we see that $H\lesssim \varepsilon H$. By choosing  $\varepsilon\ll1$ small enough we can ensure that  
$ H \equiv 0$. Therefore 
 \[\lambda_i(\rho_1)=\mu_j(\rho_1) =0.\hspace{1cm}\forall 1\leq i\leq \tdk, 1\leq j\leq k\]
By (\ref{all-energy-control}) and (\ref{all-energy-formula}), we then have 
\[\|\vec{u}(0)\|_{\Hd(r\geq \rho_1)}=0\]
which is a contradiction with the definition of $\rho_0$ since  $\rho_1<\rho_0$. This completes the proof. 
\end{proof}

\begin{remark}
Above we used Proposition~\ref{Decay-Thm} at time $t = 0$ to obtain~(\ref{all-energy-control}). We note that although the statement of Proposition~\ref{Decay-Thm} makes it seem like $R>R_0$ must be a large number, in fact all that is required in the proof is that the energy of the truncated data be small enough to be able to apply Lemma~\ref{linear-decay-bound}. If the initial data $(u_0, u_1)$ is compactly supported, this smallness can be achieved by simply taking $R_0$ close to the edge of the support. 
The same can be said about Lemma~\ref{linear-decay-bound}. An examination of the proof of Lemma~\ref{linear-decay-bound} reveals that in the case of compactly supported data, one can gain an extra small factor involving $(\rho_0-R_0)$, which is enough to treat the potential term perturbatively. 
\end{remark}


We next consider the case $\vartheta_1 \neq 0$. 
\\

\noindent\textbf{Case 2: $\vartheta_1 \neq 0$ is impossible.} \quad \\

By Proposition~\ref{prop:as} we have 
\[u_0=\vartheta_1 r^{2-d} +O(r^{3-2d})\]
Recall that by definition,  $r^\ell u_0=\psi_0(r)-Q(r)$. Using  the    asympotitc behavior  of $Q(r)$ from  (\ref{Q-asymptotic}), we see that in fact,  
\begin{equation*}
\psi_0(r) =n\pi -\frac{\alpha_0-\vartheta_1}{r^{\ell+1}} +O(r^{-3(\ell+1)})
\end{equation*}
By Lemma~\ref{lem:Q-asymptotic}, we can find  a unique solution $Q_{\alpha_0-\vartheta_1} \in \dot{H}^1(\R^3_*)$ to \textnormal{(\ref{HarmonicMap})} with the same asymptotics 
\begin{equation*}
Q_{\alpha_0-\vartheta_1}=n\pi -\frac{\alpha_0-\vartheta_1}{r^{\ell+1}} +O(r^{-3(\ell+1)})
\end{equation*}
Since $\vartheta_1 \neq 0$ we also note that Lemma~\ref{lem:Q-asymptotic} shows that 
 \ant{ 
 Q_{\alpha_0-\vartheta_1}(1)\not=0
 } 

To simplify notation we write $\vartheta:= \vartheta_1$ and  $ \ti{Q}:=Q_{\alpha_0-\vartheta}$. Define 
\begin{equation}\begin{aligned}\label{new-solution}
u_{\vartheta,0}(r):&=\frac{1}{r^\ell}(\psi_0(r)- \ti{Q}(r))\\
u_{\vartheta,1}(r):&=\frac{1}{r^\ell}\psi_1(r)\\
u_{\vartheta}(t,r):&=\frac{1}{r^\ell}(\psi(t,r)- \ti{Q}(r))
\end{aligned}\end{equation}
where of course $ \psi(t, r) := r u(t, r)$. 
Then, 
$u_{\vartheta}(t,r)$ solves 
\EQ{
\label{eq:vartheta}
&\partial_{tt}u_{\vartheta} - \partial_{rr}u_{\vartheta}- \frac{2\ell+2}{r}\partial_r u_{\vartheta} + V_{\vartheta}(r) u_{\vartheta}  = \N_{\vartheta}(r,u_{\vartheta}), \hspace{1cm} r\geq 1\\
&u_{\vartheta}(t,1)=-  \ti{Q}(1)\not=0, \hspace{0.5cm}\forall t\in \R\\
&\vec{u}_{\vartheta}(0) =(u_{\vartheta,0}, u_{\vartheta,1})
}
where
\begin{equation*}
\begin{aligned}
V_{\vartheta}(r) &: = \frac{\ell(\ell+1)(\cos2 \ti{Q} -1)}{r^2}
\\
\N_{\vartheta}(r,u_{\vartheta})& := F_{\vartheta}(r, u_{\vartheta}) + G_{\vartheta}(r, u_{\vartheta})\\
F_{\vartheta}(r, u_{\vartheta})&:= \frac{\ell(\ell+1)}{r^{\ell+2}}  \sin^2(r^{\ell} u_{\vartheta})\sin 2  \ti{Q}\\
G_{\vartheta}(r, u_{\vartheta})&:= \frac{\ell(\ell+1)}{2r^{\ell+2}}(2r^{\ell} u_{\vartheta}-\sin (2r^{\ell} u_{\vartheta}))\cos 2
 \ti{Q}
\end{aligned}
\end{equation*}
Note that  the  difference between (\ref{u eq})  and (\ref{eq:vartheta}) is that  the Dirichlet boundary condition at $r=1$ is not satisfied by $\vec u_{\vartheta}(t)$ and that we have replaced $Q$ by $ \tilde{Q}$.

We list several properties of $\vec u_{\vartheta}(t)$.
\begin{enumerate}
\item $ \tilde{K}:=\{\vec{u}_{\vartheta}(t) :t\in \R\}$ is pre-compact in $\dot{H}^1 \times L^2(r \ge 1)$. This follows since 
\EQ{ \label{tu}
\vec{u}_{\vartheta}(t)=\big({u}(t) +\frac{1}{r^\ell}(Q- \ti{Q}),  \, u_t(t) \big), 
}  the trajectory of  $\vec{u}(t)$ is pre-compact in $\HH$, and  
$Q-  \tilde{Q}$ is independent of time. In particular, this means that 
\EQ{ \label{ut ext}
\| \vec u_{\vartheta}(t) \|_{\HH (r \ge R + \abs{t})} \to 0 \mas \abs{t} \to \infty
}
\item $\vec {u}_{\vartheta}(0)$ has the following asymptotic behavior 
\EQ{
&\label{new-initial-a}
{u}_{\vartheta,0}(r) =  O(r^{3-2d}) \text{ as } r\rightarrow \infty
\\
&\int_r^\infty  {u}_{\vartheta,1}(s)s^{2i-1} ds=O(r^{2i+2-2d})  \text{ as } r\rightarrow \infty
}
The above follows from~\eqref{tu},  Proposition~\ref{prop:as},   Lemma~\ref{lem:Q-asymptotic}. 
\item ${u}_{\vartheta}(t,1)=-\tilde{Q}(1)\not=0$.
\end{enumerate}

To complete the proof of Proposition~\ref{Rigidity} we will prove below that   ${u}_{\vartheta}\equiv 0$.  Indeed, by  (\ref{new-solution}) this would imply that  $\psi_0( r) \equiv\ti{Q}(r)$,  which is a  contradiction because $\psi(t,1)=0$,  while $\ti{Q}(1)\not=0$.

\begin{lemma}\label{nonzero-theta} Let $\vec {u}$ be as in the Proposition~\ref{Rigidity}, and suppose $\vartheta := \vartheta_1\not=0$. Then if  $\vec u_\vartheta(t) $ is defined as in \textnormal{(\ref{new-solution})} we must have  $\vec {u}_{\vartheta}\equiv (0, 0)$.
\end{lemma}
\begin{proof} 
The argument that we will use to prove Lemma~\ref{nonzero-theta} is basically  identical  to the one presented in the previous step regarding the case $\vartheta = 0$ and we give only a very brief sketch of the proof here. We refer the reader to~\cite[Proof of Proposition~$5.15$]{KLS} for a more detailed version of this argument. 

First we can show that $\vec u_{\vartheta}(0)$ must be compactly supported. 
Indeed, we can  adapt the same argument used to prove Proposition~\ref{Decay-Thm}, using~\eqref{ut ext} in place of Corollary~\ref{Exterior-Decay}, 
to prove the analogous result for $\vec u_{\vartheta}(t)$. In particular, (\ref{decay-estimate}) holds with $\vec u_{\vartheta}(t)$ in place of $\vec u(t)$, i.e.,  there exists $R_0>0$ so that for all $R>R_0$ we have 
\EQ{ \label{pi uth}
\|\pi_R^\perp \vec{u}_{\vartheta}(t)\|_{\mathcal{H}(r\geq
R)}\lesssim&  R^{1-d}\|\pip \vec{u}_{\vartheta}(t)\|_{\hrr}\\ &
+R^{-\frac{d}{2}}\|\pip \vec{u}_{\vartheta}(t)\|_{\hrr}^2 + R^{-1}\|\pip
\vec{u}_{\vartheta}(t)\|_{\hrr}^3
} 

Next, we define the projection coefficients $\lambda_{\vartheta,j},  \mu_{\vartheta,j}$ 
for $\vec u_{\vartheta}$, 
\ant{
\pipp \vec{u}_\vartheta(t, r)= \left(u_{\vartheta}(t, r)-\sum_{j=1}^{\tdk}\lambda_{\vartheta,j}(t,R)  {r^{2j-d}} , \,  \partial_t u_{\vartheta}(t, r)-\sum_{j=1}^k\mu_{\vartheta,j}(t,R) r^{2j-d}\right)
}
using the formulas  (\ref{lambda-explicit}) and  (\ref{mu-explicit}). 
Using the asymptotics of $u_\vartheta$ in (\ref{new-initial-a}), we immediately get  the asymptotics for $\lambda_{\vartheta,j},  \mu_{\vartheta,j}$.
\begin{align*} 
|\lambda_{\vartheta,j}|=O(r^{-2j-d+3}), \hspace{0.5cm}
  |\mu_{\vartheta,j}|=O(r^{-2j-d+2})\end{align*}
Using the analog of Lemma~\ref{ldmu-identity} for $\vec u_{\vartheta}$, we can rewrite~\eqref{pi uth} in terms of $\la_{\vartheta, j}$ and $\mu_{\vartheta_j}$. It follows that Lemma~\ref{all-decay}, Lemma~\ref{lem:all-difference}, and Corollary~\ref{cor:all-difference} hold for $\la_{\vartheta, j}$ and $\mu_{\vartheta, j}$ in place of $\la_j$ and $\mu_j$. 

With this information, we can apply the exact same argument used to prove Lemma~\ref{compact-support} to deduce that $\vec u_{\vartheta}(0)$ must be compactly supported. To conclude, we can then argue exactly as in the proof of Lemma~\ref{0solution} to prove that $\vec u_{\vartheta} \equiv (0, 0)$. 
\end{proof}

\subsection{Proof of Theorem~\ref{Rigidity} and Theorem~\ref{MainThm}}
We give a brief summary of the proof of Proposition~\ref{Rigidity}. 
\begin{proof}[Proof of Propopsition~\ref{Rigidity}]
Suppose that $\vec u(t)$ is a solution to~\eqref{u eq} so that the trajectory, 
\ant{
\K:= \{ \vec u(t) \mid t \in \R \}
}
is pre-compact in $\HH$. We also note that 
\ant{
r^{\ell} u (t, r):=  \psi(t, r) - Q(r)
}
where $\psi(t, r)$ is a degree $n$, $\ell$-equivariant wave map with $\ell \ge 2$, i.e., a solution to~\eqref{EE1}.  
Then by Proposition~\ref{prop:as} there exists $\vartheta \in \R$ so that 
\ant{
&\abs{r^{d-2} u_0(r) -   \vartheta}  = O(r^{-d+1}) \mas r \to \infty\\
&\abs{\int_r^\infty u_1(s)s^{2i-1} ds} =O(r^{2i+2-2d}) \mas r \to \infty\quad  \forall 1\leq i\leq k
}
where $k  = \frac{\ell}{2}$ if $\ell \ge 2$ is even,  and $k = \frac{\ell +1}{2}$ if $\ell \ge 2$ is odd. If $\vartheta \neq 0$, then by Lemma~\ref{nonzero-theta} we must have $\psi(0, r) =  Q_{\al_0 - \vartheta}(r)$, which yields a contradiction since $Q_{\al_0 - \vartheta}(1) \neq 0$. Therefore $\vartheta = 0$, and hence $\vec u(t) = (0, 0)$ by Lemma~\ref{0solution}. This completes the proof of Theorem~\ref{Rigidity}. 
\end{proof}

Finally, we summarize the proof of Theorem~\ref{MainThm}.

\begin{proof}[Proof of Theorem~\ref{MainThm}]
We prove the equivalent reformulation,  Theorem~\ref{thm:main}. Suppose Theorem~\ref{MainThm} and hence Theorem~\ref{thm:main}, are false. Then, by Proposition~\ref{critical-element}, there exists a critical element, $\vec u_*(t)$, i.e., a nonzero solution to~\eqref{u eq} such that the trajectory $\K:= \{ \vec u_*(t) \mid t \in \R\}$ is pre-compact in $\HH$. By Theorem~\ref{Rigidity} we must then have $\vec u_*(t) \equiv (0, 0)$. Since the critical element is nonzero, we have reached a contradiction, which means that  Theorem~\ref{thm:main} and Theorem~\ref{MainThm} are true. 
 \end{proof}


\bibliography{researchbib}
\bibliographystyle{plain}

\medskip

\centerline{\scshape Carlos Kenig,  Baoping Liu, Wilhelm Schlag}
\medskip
{\footnotesize
 \centerline{Department of Mathematics, The University of Chicago}
\centerline{5734 South University Avenue, Chicago, IL 60615, U.S.A.}
\centerline{\email{cek@math.uchicago.edu, baoping@math.uchicago.edu, schlag@math.uchicago.edu}}
} 
\vspace{\baselineskip}

\centerline{\scshape  Andrew Lawrie}
\medskip
{\footnotesize
 \centerline{Department of Mathematics, University of California Berkeley}
\centerline{
859 Evans Hall
Berkeley, CA 94720}
\centerline{\email{alawrie@math.berkeley.edu}}
}

\end{document}